\documentclass{amsart}
\usepackage{mathrsfs}
\usepackage[all]{xy}
\usepackage{mathabx} 

\usepackage{wasysym}
\usepackage{amssymb}

\usepackage{mathtools}
\usepackage{mathbbol}

\usepackage{ulem} 

\newcommand{\showcomments}{yes}

\newsavebox{\commentbox}
{\ifthenelse{\equal{\showcomments}{yes}}%
{\footnotemark
    \begin{lrbox}{\commentbox}
    \begin{minipage}[t]{1.25in}\raggedright\sffamily\upshape\tiny
    \footnotemark[\arabic{footnote}]}%
{\begin{lrbox}{\commentbox}}}%
{\ifthenelse{\equal{\showcomments}{yes}}%
{\end{minipage}\end{lrbox}\marginpar{\usebox{\commentbox}}}%
{\end{lrbox}}}

\usepackage{ifpdf}
\ifpdf 
  \usepackage[pdftex]{graphicx}
  \DeclareGraphicsExtensions{.pdf,.png,.jpg,.jpeg,.mps}
  \usepackage{pgf}
  \usepackage{tikz}
\else 
  \usepackage{graphicx}
  \DeclareGraphicsExtensions{.eps,.bmp}
  \usepackage{pgf}
  \usepackage{pstricks}
\fi
\usepackage{epic,bez123}
\usepackage{wrapfig}

\DeclareMathAlphabet{\mathpzc}{OT1}{pzc}{m}{it}

\usepackage{ifpdf}
\ifpdf 
  \usepackage[pdftex]{graphicx}
  \DeclareGraphicsExtensions{.pdf,.png,.jpg,.jpeg,.mps}
  \usepackage{pgf}
\else 
  \usepackage{graphicx}
  \DeclareGraphicsExtensions{.eps,.bmp}
  \usepackage{pgf}
  \usepackage{pstricks}
\fi
\usepackage{epic,bez123}
\usepackage{wrapfig}

\usepackage{geometry} 
\newtheorem{thm}{Theorem}[section]
\newtheorem{prop}[thm]{Proposition}
\newtheorem{cor}[thm]{Corollary}
\newtheorem{lem}[thm]{Lemma}

\newtheorem*{claim}{Claim}

\theoremstyle{definition}
\newtheorem{defn}[thm]{Definition}

\theoremstyle{remark}
\newtheorem{quest}[thm]{Question}
\newtheorem{rem}[thm]{Remark}

\newcommand{\Gx }{\mathrm{Cay}(\Gamma, S)}

\newcommand{\Hdim}{{\rm Hdim}}

\newcommand{\diam }[1]{{\textbf{diam}\big(#1\big)}}

\usepackage[bookmarks=true, pdfauthor={YANG Wenyuan}]{hyperref}
\usepackage{etoolbox}
\apptocmd{\sloppy}{\hbadness 10000\relax}{}{}
\apptocmd{\sloppy}{\vbadness 10000\relax}{}{}
\usepackage{enumerate}
\begin{document}

\title{Branching Random Walks on relatively hyperbolic groups}

\author{Matthieu Dussaule}
\address{Matthieu Dussaule : D\'epartement de math\'ematiques, Facult\'e des Sciences Batiment I,
2 Boulevard Lavoisier
49045 Angers cedex 01}
\email{matthieu.dussaule@hotmail.fr}

\author{Longmin Wang}
\address{Longmin Wang : School of Statistics and Data Science, LPMC, Nankai University, 94 Weijin Road, Nankai District, Tianjin, China }
\email{wanglm@nankai.edu.cn}

\author{Wenyuan Yang}
\address{Wenyuan Yang : Beijing International Center for Mathematical Research (BICMR), Beijing University, No. 5 Yiheyuan Road, Haidian District, Beijing, China}
\email{yabziz@gmail.com}

\thanks{M.D.   has received funding from the European Research Council  (ERC) under the European Union's Horizon 2020 research and innovation  program under the Grant Agreement No 759702. L.W. is supported by National Natural Science Foundation of China (12171252).  W.Y. is supported by National Key R \& D Program of China (SQ2020YFA070059).}




\dedicatory{}


\begin{abstract}
 Let $\Gamma$ be a non-elementary relatively hyperbolic group with a finite generating set.  Consider a finitely supported admissible and symmetric probability measure $\mu$ on $\Gamma$ and a probability measure $\nu$ on $\mathbb{N}$ with mean $r$.  Let $\mathrm{BRW}(\Gamma,\nu,\mu)$ be the branching random walk on $\Gamma$ with offspring distribution $\nu$ and base motion given by the random walk with step distribution $\mu$.  It is known that for $1 < r \leq R$ with $R$ the radius of convergence for the Green function of the random walk, the population of $\mathrm{BRW}(\Gamma,\nu,\mu)$ survives forever, but eventually vacates every finite subset of $\Gamma$.  We prove that in this regime, the growth rate of the trace of the branching random walk is equal to the growth rate $\omega_\Gamma(r)$ of the Green function of the underlying random walk. We also prove that the Hausdorff dimension of the limit set $\Lambda(r)$, which is the random subset of the Bowditch boundary consisting of all accumulation points of the trace of $\mathrm{BRW}(\Gamma,\nu,\mu)$, is equal to a constant times $\omega_\Gamma(r)$.
\end{abstract}
\maketitle

\tableofcontents
\section{Introduction}

\subsection{Background and motivation}
Let $(V,E)$ be a locally finite connected infinite graph. 
A Branching Markov chain (BMC) on $(V,E)$ is defined as follows.
One starts with a single particle at a fixed vertex $v_0\in V$.
For $n\geq 1$, each particle still alive at time $n$ dies and gives birth to an independent random number of offspring particles, according to a probability measure $\nu$ on $\mathbb N = \{1, 2, 3, \ldots\}$, each of them independently moving on $(V,E)$ according to an underlying Markov chain on $(V,E)$ driven by a transition kernel $p(x,y),x,y\in V$.
Sometimes in literature, the measure $\nu$ is distributed on $\mathbb Z_{\geq 0}=\{0,1,2,\ldots \}$ but
we will always assume that there is at least one offspring particle to avoid the extinction of the system.
This is not serious restriction, as conditioning on non-extinction, one can assume that $\nu$ is distributed on $\mathbb N$, see \cite[Chapter~1]{AN04}.
We will also assume that the underlying Markov chain is irreducible, i.e.\ every vertex of the graph can be visited by the Markov chain with positive probability and that it is symmetric, i.e.\ $p(x,y)=p(y,x)$ for all $x,y\in V$.
There is a large body of work dedicated to branching Markov chains on the real line and we refer to \cite{Shi} and references therein for more details.
We also refer to \cite{AN04} for a general discussion and historical perspective on branching processes.

A branching Markov chain is called recurrent if
with positive probability (and thus with probability 1), some (and thus every) vertex of the graph is visited by infinitely many particles of the BMC.
It is called transient otherwise.
Recurrence or transience of the BMC is governed by the expectation of the offspring distribution
$$\mathbf E[\nu]=\sum_{k\geq 1}k\nu(k)$$
and by the spectral radius $\rho$ of the Markov chain defined by
$$\rho=\limsup_np_n(x,y)^{1/n},$$
which is independent of $x$ and $y$, provided the underlying Markov chain is irreducible.
Here $p_n$ is the $n$-th convolution power of $p$ defined by
$$p_n(x,y)=\sum_{z_1,\ldots,z_{n-1}\in V}p(x,z_1)p(z_1,z_2)\cdots p(z_{n-1},y).$$
More precisely, if $\mathbf E[\nu]\leq \rho^{-1}$, then the branching Markov chain is transient, otherwise, it is recurrent, see \cite{BP94, GM06} and references therein.

Let us now restrict our attention to the following context.
Consider a finitely generated group $\Gamma$ endowed with a finite generating set and a probability measure $\mu$ on $\Gamma$.
For a given probability measure $\nu$ on $\mathbb N$, the \textit{branching random walk} $\mathrm{BRW}(\Gamma,\nu,\mu)$ is the branching Markov chain on the Cayley graph of $\Gamma$ driven by $\nu$ and by the $\mu$-random walk, which is the Markov chain whose transition probability is given by $p(x,y):=\mu(x^{-1}y)$.
Since $p$ is assumed to be symmetric, $\mu$ is symmetric in the sense that $\mu(x)=\mu(x^{-1})$ for every $x\in \Gamma$.
In this case, irreducibility of the random walk means that the support of $\mu$ generates $\Gamma$ as a group.
We also say that the random walk driven by $\mu$ is \textit{admissible}.

In the transient case, we introduce the trace $\mathcal{P}$ of the branching random walk, which is the set of vertices that are ever visited by $\mathrm{BRW}(\Gamma,\nu,\mu)$.
It is a random subset of $\Gamma$ and it has been a fruitful line of research to study the geometric properties of $\mathcal{P}$.
When the group $\Gamma$ is endowed with a geometric boundary $\partial \Gamma$, one can define the limit set $\Lambda$ of $\mathcal{P}$ as the closure of $\mathcal{P}$ in $\partial \Gamma$, i.e.\ $\Lambda=\mathbf{cl}(\mathcal{P})\cap \partial\Gamma$.
In hyperbolic context, the growth rate of $\mathcal{P}$ and the Hausdorff dimension of $\Lambda$ for suited distance on the boundary $\partial \Gamma$ has been related to asymptotic quantities involving $\Gamma$ and $\mu$ as we explain below.
Let us first mention that Benjamini and M\"uller \cite{BM12} studied qualitative properties of $\mathcal{P}$ and listed a certain number of conjectures. They proved in particular that $\mathcal{P}$ has exponential growth, while their method does not give quantitative results on the growth rate.
Hutchcroft proved in \cite{Hutchcroft20} that on any non-amenable group, two independent branching random walks almost surely intersect at most finitely often, which imply that $\mathcal{P}$ has infinitely many ends almost surely.
This answers some of the questions in \cite{BM12}.
Let us also mention that in a very recent work \cite{KaimanovichWoess}, Kaimanovich and Woess studied limit behaviour of branching random walks in terms of geometric features of $\Gamma$ with a very new angle.

\subsection{Earlier results on hyperbolic groups}
We introduce the Green function defined as
$$G_r(x,y)=\sum_{n\geq 0}p_n(x^{-1}y)r^n,x,y\in \Gamma.$$
Its radius of convergence $R$ is independent of $x,y$ provided that the random walk driven by $\mu$ is admissible. We have $R=\rho^{-1}$, where $\rho$ is the spectral radius introduced above.
The groups under consideration in this paper will always be non-amenable, so by a landmark result of Kesten, $R>1$, see \cite[Corollary~12.5]{Woessbook}.
For every $r\leq R$, we set
$$H_r(n)=\sum_{x\in S_n}G_r(e,x)$$
where $S_n$ is the sphere of radius $n$ centered at $e$ for the word distance
and
$$\omega_\Gamma(r)=\limsup_{n\to \infty}\frac{1}{n}\log H_r(n).$$
We call $\omega_\Gamma(r)$ the growth rate of the Green function.
It depends both on $\mu$ and on the chosen finite generating set of $\Gamma$.
We also set $\mathcal{P}_n=\mathcal{P}\cap S_n$ and $M_n=\sharp \mathcal{P}_n$.
We define the growth rate of $\mathrm{BRW}(\Gamma,\nu,\mu)$ as $\limsup \frac{1}{n}\log M_n$.

Assume that $\Gamma$ is a finitely generated free group endowed with its standard generating set and consider a branching random walk on $\Gamma$ with $\mathbf E[\nu]=r\in [1,\rho^{-1}]$.
Let $\lambda<1$ and for $x,y\in \Gamma$, set $d_\lambda(x,y)=\lambda^{n}$, where $n$ is the biggest integer such that the prefixes of length $n$ of $x$ and $y$ coincide.
Then, $d_\lambda$ extends to a distance on $\Gamma \cup \partial \Gamma$, where $\partial \Gamma$ is the set of ends of $\Gamma$.
Liggett~\cite{Liggett96} and Hueter-Lalley \cite{HL00} proved that whenever the underlying random walk is 
a symmetric, possibly anisotropic, nearest neighbor random walk on $\Gamma$, then the limit
$$\theta(r)=\lim_{n \to \infty} \left( M_n \right)^{1/n}$$
is well defined and is almost surely a constant.
Moreover, the Hausdorff dimension of the limit set $\Lambda$ in $\partial \Gamma$ is equal to $- \log_{\lambda} \theta(r)$.
Finally, the function $\theta(r)$ is continuous and strictly increasing on $[1,\rho^{-1}]$ and has critical exponent $1/2$ at $\rho^{-1}$, which means that
there exists $C>0$ such that
$$\theta(\rho^{-1})-\theta(r)\underset{r\to \rho^{-1}}{\sim}C\sqrt{\rho^{-1}-r}$$
and
$$\log \theta(\rho^{-1})-\log \theta(r)\underset{r\to \rho^{-1}}{\sim}\frac{C}{\theta(\rho^{-1})}\sqrt{\rho^{-1}-r}.$$
Furthermore, $\log \theta(r)\leq \frac{1}{2}v$, where $v=\log (2q-1)$ is the volume growth rate of the free group $\Gamma$ with $q$ generators. 

In \cite{SWX}, the authors extended these results to all hyperbolic groups and expressed $\log \theta$ as the growth rate of the Green function, i.e.
$$\theta(r)=\mathrm{e}^{\omega_\Gamma(r)}.$$
Here, $\partial \Gamma$ is the Gromov boundary of $\Gamma$ endowed with a visual distance $d_\lambda$ satisfying
$d_\lambda(\xi,\zeta)\asymp \lambda^{(\xi|\zeta)}$, where $(\cdot|\cdot)$ is the Gromov product,  see \cite[Section~2.1]{SWX} for more details.
In particular,   the Hausdorff dimension of the limit set was proven there to be   $\omega_\Gamma(r)$ for $r<\rho^{-1}$, but the critical case $r=\rho^{-1}$ remained open.
As a particular case of our work, we will show that this is still true at $r=\rho^{-1}$, see Corollary~\ref{corohausdorffdimensionhyperbolic} below.

\subsection{Main results}
Our goal in this paper is to generalize some of the aforementioned results to the class of   relatively hyperbolic groups, whose precise definition is recalled   in Section~\ref{sectionrelativelyhyperbolicgroups}.
Briefly, a finitely generated group $\Gamma$ is called relatively hyperbolic if it admits a geometrically finite action on a proper geodesic hyperbolic space $X$.
The Bowditch boundary of $\Gamma$ is then the limit set of the orbit of a fixed base point $x_0$ in the Gromov boundary of $X$.
It is unique up to $\Gamma$-equivariant homeomorphism.
We say that $\Gamma$ is non-elementary if its Bowditch boundary is infinite.
Since their introduction by Gromov,  these groups have been extensively studied by many authors and from different point of views. Besides the class of Gromov hyperbolic groups, the following   groups of geometric and algebraic interests are the main archetypal examples: 
\begin{enumerate}

    \item
    Fundamental groups of finite volume hyperbolic manifolds, and of more general finite volume Riemaniann manifolds with negatively pinched sectional curvature. 
    \item
    Infinitely ended groups, equivalently by Stalling's theorem, all non-trivial amalgamated free products and HNN extension over finite groups.

\end{enumerate}
In particular, free products are  the simplest combinatorial examples of relatively hyperbolic groups, on which branching random walks  were studied by Candellero, Gilch and M\"uller \cite{CandelleroGilchMuller}. 
We recover some of their results in this paper.

A finitely generated group $\Gamma$ endowed with a finite generating set $S$ is equipped with the \textit{word distance}.
Its volume growth rate is the growth rate of the balls for the word distance.
We again refer to Section~\ref{sectionrelativelyhyperbolicgroups} for more details.

\begin{thm}\label{Thmvolumegrowthconst}
Let $\Gamma$ be a non-elementary relatively hyperbolic group endowed with a finite generating set.
Consider a finitely supported admissible and symmetric probability measure $\mu$ on $\Gamma$ and a probability measure $\nu$ on $\mathbb{N}$ with mean $r\in [1,R]$.
Consider a branching random walk $(\Gamma,\nu,\mu)$ starting at $e$.
Then, almost surely,
$$\limsup_{n\to\infty}\frac{1}{n}\log M_n=\omega_\Gamma(r).$$
Moreover, the function $r\mapsto \omega_\Gamma(r)$ is increasing and continuous and satisfies that $\omega_\Gamma(r)\leq v/2$, where $v$ is the volume growth rate of the group for the word distance.
\end{thm}

We then investigate the limit behaviours of the branching random walk at infinity and we compute the Hausdorff dimension of the limit set of the trace in various compactifications of relatively hyperbolic groups. Introduced by Floyd \cite{Floyd}, the Floyd boundary can be constructed as a compactification for  any locally finite graph, equipped with a rescaling of the graph distance called the Floyd distance. 
Generalizing Floyd's theorem  \cite{Floyd} for geometrically finite Kleinian groups,   Gerasimov \cite{Gerasimov} proved that for any relatively hyperbolic groups, the Bowditch boundary can be realized as a quotient of  the Floyd boundary of the Cayley graph.  The Floyd distance depending on a parameter $\lambda\in (0,1)$ is then transferred to the Bowditch boundary  as soon as $\lambda\geq \lambda_0$, where $\lambda_0$ is given by \cite[Map Theorem]{Gerasimov}.
The corresponding distance on the Bowditch boundary is called the \textit{shortcut distance} and will be described in Section~\ref{SectionFloydboundary}. 

\begin{thm}\label{ThmHdimconst}
Under the assumption of Theorem \ref{Thmvolumegrowthconst}, denote by $\Lambda(r)$ the limit set of the branching random walk in the Bowditch boundary, endowed with the shortcut distance of parameter $\lambda\in [\lambda_0,1)$.
Then, almost surely,
$$\Hdim\big(\Lambda(r)\big)=\frac{\omega_\Gamma(r)}{-\log \lambda}.$$
\end{thm}
\begin{rem}
The lower bound on the Hausdorff dimension actually holds for the whole limit set of the branching random walk in the Floyd boundary endowed with the Floyd distance, see Proposition~\ref{lowerboundHdim}.
 We can only prove the upper bound for a subset of the Floyd boundary, see Proposition~\ref{upperboundHausdorff}. In many interesting cases, the Floyd boundary is homeomorphic to the Bowditch boundary (eg. if  parabolic subgroups are amenable). 
Moreover, under the technical condition that  the volume growth rate $v_S(P)$ for every parabolic subgroup $P$ is smaller than $\omega_{\Gamma}(r)$, the above conclusion is still true for the  Floyd boundary, see Corollary~\ref{corowholeFloyd}.
In general, the identification of the Hausdorff dimension  remains open for the full limit set in the Floyd boundary.
\end{rem}


 
As above-mentioned, groups with infinitely many ends form a special class of relatively hyperbolic groups.
Such groups can be compactified with the ends boundary, introduced by Freudenthal \cite{Freudenthal}. A natural family of visual distances depending on a parameter $\lambda\in(0,1)$ on the ends boundary      will be described in Section~\ref{sectionends}.
Along the way, we prove the following Theorem, which both extends  \cite[Theorem~3.5]{CandelleroGilchMuller} to all groups with infinitely many ends and fix a gap in their proof, as we will explain in Section~\ref{sectionends}.
\begin{thm}\label{coroaccessiblegroups}
Under the assumption of Theorem \ref{Thmvolumegrowthconst}, if $\Gamma$ is  a group with infinitely many ends  and $\Lambda(r)$ is the   limit set of the branching random walk in the ends boundary, endowed with the visual distance of parameter $\lambda\in (0,1)$, then almost surely,
$$\Hdim\big(\Lambda(r)\big)=\frac{\omega_\Gamma(r)}{-\log \lambda}.$$  
\end{thm}

Finally, for hyperbolic groups, the Bowditch boundary and the Gromov boundary coincide and the shortcut distance is bi-Lipschitz to the visual distance.
We thus deduce from Theorem~\ref{ThmHdimconst} the following, which  resolves \cite[Conjecture 1.4]{SWX}.
\begin{cor}\label{corohausdorffdimensionhyperbolic}
Under the assumption of Theorem \ref{Thmvolumegrowthconst}, if   $\Gamma$ is a non-elementary hyperbolic group and if $\Lambda(R)$ is the limit set of the branching random walk with mean offspring $R$ in the Gromov boundary of $\Gamma$ endowed with a visual distance of parameter $\lambda\in [\lambda_0,1)$,
then almost surely,
$$\Hdim\big(\Lambda(R)\big)=\frac{\omega_\Gamma(R)}{-\log \lambda}.$$
\end{cor}

In the course of this investigation, we obtain much more precise information on the trace $\mathcal{P}$ of the branching random walk.
We prove that $\mathcal{P}$ almost surely tracks the logarithm neighborhood of transition points along geodesics ending at conical limit points in the limit set.
Recall that a point on a geodesic is called a transition point if it is not deep in a parabolic coset.
We refer to Section~\ref{sectionrelativelyhyperbolicgroups} for a precise definition.

\begin{thm}\label{thmlogarithmictracking}
Under the assumption of Theorem \ref{Thmvolumegrowthconst}, {there exists $0 < \kappa < \infty$ such that} 
almost surely,  for every conical point $\xi\in \Lambda(r)$, we have
$$
\limsup_{|x|\to\infty} \frac{d(x,\mathcal P)}{\log |x|}\le \kappa, 
$$
where $x$ is taken over  the set of transition points on $[e,\xi]$.
\end{thm}

This should be compared with analogous results for random walks.
Under finite first moment condition, a random walk on a hyperbolic group almost surely sub-linearly tracks geodesics from the basepoint $e$ to the limit point of the random walk in the Gromov boundary.
We refer to  Kaimanovich \cite[Theorem~7.3]{Kaimanovich} for a proof, see also \cite{Ledrappier} for the case of trees.
Sub-linear tracking of geodesics is one of the most important result in the study of random walks on groups with hyperbolic properties and is related to a celebrated multiplicative ergodic theorem of Furstenberg and Kesten \cite{FurstenbergKesten}.
It was first coined by Kaimanovich in the context of symmetric spaces \cite{Kaimanovichsymmetricspaces} and was then extended to groups with non-positive curvature by Karlsson and Margulis \cite{KarlssonMargulis}
and to more general classes of group by Tiozzo \cite{Tiozzo}, including mapping class groups.
If the random walk has finite support, then the tracking is in fact logarithmic and this is true for all weakly hyperbolic groups, i.e.\ groups admitting a non-elementary action by isometries on a Gromov hyperbolic space, see \cite[Theorem~1.3]{MaherTiozzo}.
If the group is relatively hyperbolic, then the random walk actually sub-linearly tracks transition points on the word geodesic in the Cayley graphs, see \cite[Proposition~3.2]{DussauleYang}.
Theorem~\ref{thmlogarithmictracking} can thus be thought as a generalization of these results to branching random walks on relatively hyperbolic groups and is new, even for hyperbolic groups.

\medskip
Let us finally say that we will not investigate the problem of the critical exponent of $\omega_\Gamma(r)$ in this paper.
In \cite{CandelleroGilchMuller}, the authors show that for free products this critical exponent is not always $1/2$, depending on $\mu$ and more precisely depending on whether the first derivative of the Green function is finite or infinite at its radius of convergence, see \cite[Theorem~3.10]{CandelleroGilchMuller}.
It might be possible to prove that it is in fact $1/2$ whenever the underlying random walk is spectrally non-degenerate, combining techniques of \cite{SWX}, \cite{DussauleLLT2} and of the present paper.
We refer to Definition~\ref{defnspectraldegeneracy} for the definition of spectral degeneracy of a random walk.

\subsection{Parabolic gap and purely exponential growth of Green functions}
Among the results of \cite{CandelleroGilchMuller}, the authors claimed that the Hausdorff dimension of the limit set  intersected with the set of ends  of each free factor is strictly less than that of  the  whole limit set of the branching random walk (see their Corollary~3.7).
However, their proof is incorrect on assuming that the quantity $H_r(n)$ as defined above is sub-multiplicative; we refer to Remark \ref{CGMGap}   for more details.
In our study, the sub-multiplicativity of $H_r(n)$   turns out to be a  subtle property and we propose a sufficient criterion called a parabolic gap condition, which is inspired by the work of \cite{DOP} and that we now introduce.
Note that assuming that the parabolic gap condition holds, we recover their result \cite[Corollary~3.7]{CandelleroGilchMuller}, see     Remark \ref{CGMCor}.

Let $\Gamma$ be a non-elementary relatively hyperbolic group.
Let $P$ be a parabolic subgroup.
We set
$$H_{P,r}(n)=\sum_{x\in S_n\cap P}G_r(e,x)$$
and
$$\omega_P(r)=\limsup_{n\to \infty}\frac{1}{n}\log H_{P,r}(n).$$
Thus, $\omega_\Gamma(r)$ is the growth rate of the sum of the Green functions along spheres, while $\omega_P(r)$ is the similarly defined growth rate, but restricted to the parabolic subgroup $P$.
It follows from the definition that $\omega_P\leq \omega_\Gamma$.

\begin{defn}
If $\omega_P(r)<\omega_\Gamma(r)$, we say that $\Gamma$ has a \textit{parabolic gap} along $P$ for the Green function at $r$.
If for every $P$, for every $r\in (1,R]$, $\omega_P(r)<\omega_\Gamma(r)$, then we say that $\Gamma$ has a \textit{parabolic gap} for the Green function.
\end{defn}

Note that this notion only depends on the underlying random walk driven by $\mu$, not on the offspring distribution of the branching random walk.

\begin{thm}\label{mainthmparabolicgap}
Let $\Gamma$ be a non-elementary relatively hyperbolic group.
Consider a finitely supported admissible and symmetric probability measure $\mu$ on $\Gamma$.
If $\Gamma$ has a parabolic gap for the Green function,
then the sum of Green functions along spheres is roughly multiplicative and has purely exponential growth, in the sense that for all $1\leq r\leq R$, there exist $C=C(r)\geq 1$ and $C'=C'(r)\geq 1$
such that for all $n$,
$$\frac{1}{C}H_{r} (n+m)\leq  H_r(n)H_r(m)\le C H_{r} (n+m)
$$
and
$$\frac{1}{C'}\mathrm{e}^{n\omega_\Gamma(r)}\leq H_r(n)\leq C' \mathrm{e}^{n\omega_\Gamma(r)}.$$
\end{thm}

\begin{rem}
Actually, $C$ and $C'$ can be chosen independently of $r$ in the two upper bounds, which do not require a parabolic gap to hold by Lemma~\ref{UpperBDLem}.
\end{rem}

We will prove two criteria for having a parabolic gap, see Corollary~\ref{parabolicdivergentgrowthgap} and Proposition~\ref{secondcriterionparabolicgap}.
In particular, we   prove that whenever parabolic subgroups are amenable, $\omega_P(r)<\omega_\Gamma(r)$ for all $r<R$.
Moreover, if parabolic subgroups have sub-exponential growth, then we also have $\omega_P(R)<\omega_\Gamma(R)$.

\medskip
Let us compare our notion of parabolic gap   to similar notions in different settings.
Consider a finitely generated group $\Gamma$ acting via isometries on a metric space $(X,d)$.
Then, one can endow $\Gamma$ with a left-invariant pseudo-distance by declaring
$$d(g,h)=d(g\cdot x_0,h\cdot x_0),$$
where $x_0$ is a fixed point in $X$.
 We define the volume growth rate for any subgroup $P<\Gamma$ as follows:
$$v_X(P)=\limsup_{n\to \infty} \frac{1}{n}\log \big( \sharp B(e,n)\cap P\big)=\limsup_{n\to \infty}\frac{1}{n}\log \big (\sharp \left \{g\in P: d(x_0,g\cdot x_0)\leq n\right \}\big).$$
Choosing $X$ to be the Cayley graph associated with a finite generating set $S$, endowed with the graph distance, we recover the word distance on $\Gamma$.
Then the volume growth rate of a subgroup $P$, also denoted by $v_S(P)$, is the standard terminology.

If $\Gamma$ is relatively hyperbolic, then it acts by isometries on a proper geodesic hyperbolic space $(X,d)$.
We say in this context that $\Gamma$ has a parabolic gap (also referred as critical gap in literature) if for every parabolic subgroup $P$, $v_X(P)<v_X(\Gamma)$.
This definition makes sense in larger contexts than relatively hyperbolic groups and
this property was studied a lot in literature, see for instance \cite{DOP}, \cite{DPPS}, \cite{PS18}, \cite{PTV}, \cite{Rob}, \cite{ST21}, \cite{Vidotto} and references therein.
For typical cases, we have $v_X(P)<v_X(\Gamma)$.
On the other hand, exotic examples   of  geometrically finitely groups acting on a negatively-pinched Cartan-Hadamard manifold with the critical gap property, i.e.\ $v_X(\Gamma)=v_X(P)$ for some $P$, were first constructed in \cite{DOP}, see \cite{Peigne} for other examples.

On the contrary, when endowing $\Gamma$ with a word distance given by a finite generating set $S$, this cannot happen,
since in this context, it was proved in \cite{DFW} that one always has $v_S(P)<v_S(\Gamma)$.

\bigskip
Before going further, recall that the critical exponents $v_X(\Gamma)$ and $v_X(P)$ coincide with the exponential radius of convergence of a suited Poincaré series.
Namely, define
$$\Theta_s(\Gamma)=\sum_{g\in \Gamma}\mathrm{e}^{-sd(x_0,g\cdot x_0)}=\sum_{n\geq 0}\sum_{n\leq d(x_0,g\cdot x_0)<n+1}\mathrm{e}^{-sn}.$$
Then, for $s<v_X(\Gamma)$, $\Theta_s(\Gamma)$ diverges and for $s>v_X(\Gamma)$, it converges.
Similarly, replacing $\Gamma$ by any subgroup $P$ in the above formula defines the corresponding $\Theta_s(P)$ and $v_X(P)$.

Now, consider a probability measure $\mu$ on a relatively hyperbolic group $\Gamma$ and for $r>0$, set
$$I(r)=\sum_{g\in \Gamma}G_r(e,g)G_r(g,e).$$
The following result is well-known, see for instance \cite[Proposition~1.9]{GouezelLalley}.
For every $r$, we have
$$I(r)=\frac{d}{dr}\big (rG_r(e,e)\big ).$$
Thus, for $r<R$, this quantity converges and for $r>R$, it diverges.
Consider the $r$-symmetrized Green distance $d_{G,r}$ defined as
$$d_{G,r}(g,h)=-\log \frac{G_r(g,h)}{G_r(e,e)}-\log \frac{G_r(h,g)}{G_r(e,e)}$$
which is a generalization of the Green distance introduced by Blach\`ere and Brofferio \cite{BlachereBrofferio}.
Then, the quantity $I(r)$ is exactly the Poincaré series associated with the distance $d_{G,r}$.
The only difference with the previous setting is that the parameter $r$ is part of the definition of the distance.
The quantity $-\log R$ is analogous to the critical exponent $v_X(\Gamma)$.
It is more complicated to define a notion of parabolic gap, because we cannot interpret 
$-\log R$ as the radius of convergence of the Poincaré series, which is not a power series in $r$.
However, the analogous notion which was coined in \cite{DG21} is called spectral degeneracy.
We will properly introduce this notion below and
we refer to \cite[Section~3.3]{DussauleLLT1} and \cite[Section~1.3]{DPT} for more explanations on this analogy.
Anyway, by results of Cartwright \cite{Cartwright1}, \cite{Cartwright2} and of Candellero and Gilch \cite{CandelleroGilch}, it is possible to construct both a spectrally degenerate random walk and a spectrally non-degenerate one on a relatively hyperbolic group, although every known example is in the class of free products.

\medskip
Back to our critical gap condition, we will see that the critical exponent $\omega_\Gamma(r)$ is the radius of convergence of a twisted Poincaré series $s\mapsto \Theta_{r,s}(\Gamma)$ defined by~(\ref{defPoincareseries}), involving both the Green function $G_r(e,x)$ and the word distance.
We saw above that there are sufficient conditions to have a parabolic gap, but it is difficult to tell if one can construct an example where $\omega_P(r)=\omega_\Gamma(r)$.
Answering this question would require new material.
However, the last example of Section~\ref{sectionexamples} below  suggests that it might happen (see the famous examples in the geometric context in \cite{DOP} and \cite{Peigne}).
We prove there that if there exists a finitely generated group $P$ endowed with a finitely supported admissible random walk with convergent twisted Poincaré series $\Theta_{r,s}$ at some $r\leq R$, then for a suited random walk on the free product $\Gamma=P*\mathbb Z^d$, we have $\omega_\Gamma(r')=\omega_P(r')$ for some $r'$.
See the comments at the end of Section~\ref{sectionexamples} for further details.

\subsection{Organization of the paper}
We now outline the contents of the paper and explain the overall strategy of our proofs.
In Section~\ref{sectionrelativelyhyperbolicgroups}, we recall the definition of relatively hyperbolic groups, of the Floyd distance and of the Floyd boundary.
The Floyd distance is then related with geometric properties of such groups via a number of preliminary results  used throughout the paper.
Finally, we recall the relative Ancona inequalities that will be a crucial tool.

In Section~\ref{SectiongrowthrateGreenfunction}, we study the growth rate of the Green function $\omega_\Gamma(r)$.
We prove in particular that it is increasing and continuous and bounded by $v/2$, see Corollary~\ref{corolowersemicontinuous}, Corollary~\ref{coroboundv/2}.
We also prove Theorem~\ref{mainthmparabolicgap}, i.e.\ purely exponential growth, assuming there is a parabolic gap, see Lemma~\ref{LowerBDLem}.
This is done by using classical methods for Poincar\'e series, inspired by \cite{YangSCC}.
Finally, we discuss the notion of parabolic gap through examples in the last part of the section.

Section~\ref{SectiongrowthrateBRW} is dedicated to the growth rate of the branching random walk and we prove that it is almost surely equal to the growth rate $\omega_\Gamma(r)$ of the Green function, see Proposition~\ref{P:vgrBRW}.
In particular, this ends the proof of Theorem~\ref{Thmvolumegrowthconst}.
The upper-bound $\limsup \frac{1}{n}\log M_n\leq \omega_\Gamma(r)$ is very general and does not involve relative hyperbolicity.
The lower-bound is more difficult to obtain and we have to restrict our attention to points $x\in \Gamma$ such that a geodesic from $e$ to $x$ spends a uniformly bounded amount of time $L$ in every parabolic coset.
Such geodesics are Morse in the sense of \cite{Cordes}, see also \cite{Tran} for a study a Morse geodesics in relatively hyperbolic groups.
It turns out Morse geodesic rays form a proper subset of the whole Bowditch boundary, which is typically too small to serve as a model for the Poisson boundary and thus is too small to give much information about asymptotic properties of finitely supported random walks, see in particular the comments in the introduction of \cite{QRT}, where a bigger boundary called the sub-linearly Morse boundary is introduced.
We can nevertheless prove that the growth rate of the Green function restricted to Morse directions converges to the growth rate of the whole Green function, as $L$ tends to infinity, see Lemma~\ref{omegaL}.
This is enough to adapt the arguments of \cite{SWX} for hyperbolic groups, which allows us to conclude the proof.

Finally, in the two last sections, we study the Hausdorff dimension of the limit set and we prove Theorem~\ref{ThmHdimconst}, Theorem~\ref{coroaccessiblegroups} and Corollary~\ref{corohausdorffdimensionhyperbolic}.
We start with the lower bound in Section~\ref{SectionlowerboundHausdorff}.
Following \cite{SWX}, we use the Frostman lemma and show that for every $h<\omega_\Gamma(r)/-\log \lambda$, there exists with positive probability a finite measure $\chi$ on the limit set $\Lambda$ such that the integral
$$\int \int \overline{\delta}_e(x,y)^{-h} d\chi(x)d\chi(y)$$ is finite, where $\overline{\delta}_e$ is the shortcut distance on the Bowditch boundary.
However, the proof in \cite{SWX} has a gap and we need to find a new strategy to construct the measure $\chi$, which will be defined as a random Patterson-Sullivan type measure on the limit set.
The construction is performed by using convergence results for random finite measures, whose proofs are postponed to the Appendix.
Our strategy also applies to groups with infinitely many ends and we prove Theorem~\ref{coroaccessiblegroups}.
The upper-bound for the Hausdorff dimension is proved in Section~\ref{SectionupperboundHausdorff}.
We first prove Theorem~\ref{thmlogarithmictracking}, i.e.\ logarithmic tracking by the trace of transition points on geodesic rays joining the limit set, see Lemma~\ref{LogTrackingLem}.
This allows us to cover the limit set with suited shadows and to conclude as in \cite{SWX}.
Note that the covering by shadows in \cite{SWX} only works for $r<\rho^{-1}$ and Theorem~\ref{thmlogarithmictracking} is one of the needed step to cover the case $r=\rho^{-1}$.

\section{Relatively hyperbolic groups and random walks}\label{sectionrelativelyhyperbolicgroups}
\subsection{Relative hyperbolic groups} 

Consider a finitely generated group $\Gamma$ acting properly via isometries on a proper Gromov hyperbolic space $X$.
Define the \textit{limit set} $\Lambda_\Gamma$ as the closure of $\Gamma$ in the Gromov boundary $\partial X$ of $X$, that is, fixing a base point $x_0$ in $X$, $\Lambda_\Gamma$ is the set of all possible limits of sequences $g_n\cdot x_0$ in $\partial X$, $g_n\in \Gamma$.
A point $\xi\in \Lambda_\Gamma$ is called conical if there is a sequence $g_{n}$ of $\Gamma$ and distinct points $\xi_1,\xi_2$ in $\Lambda_\Gamma$ such that
$g_{n}\xi$ converges to $\xi_1$ and $g_{n}\zeta$ converges to $\xi_2$ for all $\zeta\neq \xi$ in $\Lambda_\Gamma$.
A point $\xi\in \Lambda_\Gamma$ is called \textit{parabolic} if its stabilizer in $\Gamma$ is infinite, fixes exactly $\xi$ in $\Lambda_\Gamma$ and contains no loxodromic element.
A parabolic limit point $\xi$ in $\Lambda_\Gamma$ is called \textit{bounded parabolic} if its stabilizer in $\Gamma$ is infinite and acts co-compactly on $\Lambda_\Gamma \setminus \{\xi\}$.
Say that the action of $\Gamma$ on $X$ is \textit{geometrically finite} if the limit set only consists of conical limit points and bounded parabolic limit points.

There are in literature several equivalent definitions of relatively hyperbolic groups.
Following Bowditch \cite{Bowditch}, we say that a finitely generated group $\Gamma$ is \textit{relatively hyperbolic} with respect to a collection of subgroups $\mathbb P_0$ if it acts via a geometrically finite action on a proper geodesic Gromov hyperbolic space $X$, such that the stabilizers of parabolic limit points for this action are exactly the conjugates of the groups in $\mathbb P_0$,
which are called maximal parabolic subgroups or simply parabolic subgroups if there is no ambiguity.
By \cite[Proposition~6.15]{Bowditch}, there is a finite number of conjugacy classes of parabolic subgroups, so in other words, $\mathbb P_0$ needs to be finite.

The limit set $\Lambda_\Gamma$ is called the \textit{Bowditch boundary} of $\Gamma$.
It is unique up to equivariant homeomorphism and we will denote it by $\partial_\mathcal{B}\Gamma$ in the sequel.
A relatively hyperbolic group is called \textit{non-elementary} if its Bowditch boundary is infinite; equivalently, if some parabolic subgroup $P\in \mathbb P_0$ is of infinite index in $\Gamma$.

Relatively hyperbolic groups are modelled on finite co-volume Kleinian groups.
In this case, the group acts via a geometrically finite action on the hyperbolic space $\mathbb H^n$ and there is a collection of separated horoballs such that the action on the complement of these horoballs is co-compact.
The parabolic subgroups are exactly the stabilizers of the horoballs.
Moreover, the Bowditch boundary is the ideal boundary $\mathbb S^{n-1}$ of $\mathbb H^n$ and parabolic limit points are the centers of the horoballs.

In \cite{Bowditch}, Bowditch gives another definition of relative hyperbolicity, mimicking the above geometric description of Kleinian groups.
Given a hyperbolic space $X$, one can define a coarse notion of horoballs.
A finitely generated group $\Gamma$ acting properly via isometries on a proper geodesic hyperbolic space $X$ is \textit{relatively hyperbolic} if only if there exists a $\Gamma$-invariant family of sufficiently separated horoballs centered at points in the Gromov boundary of $X$ such that $\Gamma$ acts co-compactly on the complement of these horoballs.
The parabolic subgroups are then exactly the stabilizers of these horoballs.
We also refer to \cite{Osin}, \cite{Farb}, \cite{DrutuSapir} and references therein for alternative definitions of relatively hyperbolic groups.

We set
$$\mathbb P=\{gP: g\in \Gamma, P\in\mathbb P_0\}$$ 
and we call $\mathbb P$ the collection of all parabolic cosets.

 Let us also fix some notations.
Given a finite generating $S$ on $\Gamma$, let $\Gx$ be the Cayley graph with respect to $S$.
The graph combinatorial distance is called the \textit{word distance.}
We denote the $n$-sphere centered at the identity $e$ by
$S_n=\{x\in \Gamma: d(e, x)=n\}$.
We will frequently write $|x|:=d(e,x)$. 
Finally, we will denote by $v$ the \textit{volume growth rate} of $\Gamma$ with respect to $S$, which is defined by
$$v=\limsup \frac{1}{n}\log \sharp B(e,n).$$
The word distance on relatively hyperbolic groups has purely exponential growth in the following sense.

\begin{lem}\label{purelyexponentialgrowth}\cite[Theorem 1.9]{YangPS}
There exists $c>0$ such that for every $n\geq0$, we have
$$\frac{1}{c}\mathrm{e}^{vn}\leq \sharp S_n\leq c \mathrm{e}^{vn}.$$
\end{lem}

\subsection{Floyd boundary}\label{SectionFloydboundary} 
 
We first recall the definition of   the Floyd boundary and their relation with the Bowditch boundary.
This boundary was introduced by Floyd in \cite{Floyd} and we also refer to \cite{Karlssonfreesubgroups} for more details.

Let $\Gamma$ be a finitely generated group and let $\Gx$ denote its Cayley graph associated with a finite generating set $S$.
Let $f:\mathbb{N}\to \mathbb{R}_{\ge 0}$ be a function satisfying that
$$\sum_{n\geq 0}f(n)<\infty$$ and that there exists $\lambda\in (0,1)$ such that
$$1\geq f({n+1})/f(n)\geq\lambda$$ for all $n \in \mathbb{N}$.
The function $f$ is then called the {\it rescaling function} or the \textit{Floyd function}.
In the following, we will always choose an exponential Floyd function, that is $f(n)=\lambda^n$ for some $\lambda\in (0,1)$.
Fix a basepoint $o\in \Gamma$ and rescale $\Gx$ by declaring the length of an edge $\sigma$ to be $f(d(o,\sigma))$.
The induced length distance on $\Gx$ is called the {\it Floyd distance} with respect to the basepoint $o$ and Floyd function $f$ and is denoted by $\delta_{f,o}(.,.)$.
Whenever $f$ is of the form $f(n)=\lambda^n$, we will write $\delta_{\lambda,o}=\delta_{f,o}$ and if $\lambda$ is fixed, $\delta_o=\delta_{\lambda,o}$.

The Floyd compactification $\overline{\Gamma}_{\mathcal{F}}$  is the Cauchy completion of $\Gx$ endowed with the Floyd distance.
The Floyd boundary is then defined as $\partial_{\mathcal{F}}\Gamma=\overline{\Gamma}_{\mathcal{F}}\setminus \Gx$. Different choices of base-points yield bi-Lipschitz Floyd distances because
\begin{equation}\label{changeofbasepointFloyd}
\forall x,y\in \Gamma,\; \delta_{f,o}(x,y) \le \lambda^{-d(o,o')} \delta_{f,o'}(x,y) 
\end{equation}
so the corresponding Floyd compactifications are bi-Lipschitz. Note  that the topology  may  depend on the choice of the rescaling function and the generating set. 

The following fact proved in   \cite{Karlssonfreesubgroups} plays a crucial role in understanding the Floyd geometry.
\begin{lem}\label{Karlssonlem}\cite[Lemma~1]{Karlssonfreesubgroups}
For any $\delta>0$, there exists a function $\kappa=\kappa(\delta)$ with the following property.
If  $x,y,z\in \Gamma$ are three points so that $\delta_x(y,z)\ge \delta$ then $d(x, [y,z])\le \kappa$.    
\end{lem}

If $\sharp \partial_{\mathcal{F}}\Gamma\ge 3$, Karlsson proved in \cite{Karlssonfreesubgroups}  that $\Gamma$ acts  by homeomorphism on $\partial_{\mathcal{F}}\Gamma$ as a convergence group action. By the general theory of convergence groups,   the elements in $\Gamma$ are subdivided into the classes of elliptic, parabolic and loxodromic elements. The latter two   being   infinite order  elements have exactly one and two fixed points in $\partial_{\mathcal{F}}\Gamma$ accordingly.
Moreover,  in this case the Floyd boundary contains uncountable many points and so the cardinality of $\partial_{\mathcal{F}}\Gamma$ is either 0, 1, 2 or uncountably infinite. By   \cite[Proposition~7]{Karlssonfreesubgroups}, $\sharp \partial_{\mathcal{F}}\Gamma= 2$ exactly when the group $\Gamma$ is virtually infinite cyclic.
Following Karlsson, we say that the Floyd boundary is \textit{trivial} if it is finite. The non-triviality of Floyd boundary does not depend on the choice of generating sets \cite[Lemma~7.1]{Yanggrowthtightness}.
We will only have to deal with groups with non-trivial Floyd boundary.

\medskip
We now assume that $\Gamma$ is non-elementary relatively hyperbolic.
We denote by $\partial_{\mathcal{B}}\Gamma$ its Bowditch boundary.
The following is due to Gerasimov.
\begin{thm}\label{mapGerasimov}\cite[Map Theorem]{Gerasimov}
There exists $\lambda_0\in (0,1)$ such that for every $\lambda\in[\lambda_0,1)$, the identity on $\Gamma$ extends to a continuous and equivariant surjection $\phi$ from the Floyd compactification to the Bowditch compactification of $\Gamma$.
\end{thm}

Actually, Gerasimov only stated the existence of the map $\phi$ for one Floyd function $f_0=\lambda_0^n$, but then
Gerasimov and Potyagailo proved that the same result holds for any Floyd function $f\geq f_0$, see \cite[Corollary~2.8]{GePoJEMS}.
They also proved that the preimage of a conical limit point is reduced to a single point and described the preimage of a parabolic limit point in terms of the action of $\Gamma$ on $\partial_{\mathcal{F}}\Gamma$, see precisely \cite[Theorem~A]{GePoJEMS}.
From now on, the parameter $\lambda$ will always be assumed to be contained in $[\lambda_0,1)$.

The Floyd distance can be transferred to a distance on the Bowditch boundary using the map $\phi$.
The resulting distance is called the \textit{ shortcut distance} and we denote it by $\overline{\delta}_{e,\lambda}$ or $\overline{\delta}_e$ if $\lambda$ is fixed.
It is the largest distance on the Bowditch boundary satisfying that for every $\xi ,\zeta\in \partial_\mathcal{F}\Gamma$,
\begin{equation}\label{shortcutsmallerFloyd}
    \bar{\delta}_{e,\lambda}(\phi(\xi),\phi(\zeta))\leq \delta_{e,\lambda}(\xi,\zeta).
\end{equation}
We refer to \cite[Section~4]{GePoGGD} for more details on its construction.

If $\Gamma$ is hyperbolic, then the Gromov, Bowditch and Floyd boundary all coincide.
Thus, the shortcut distance and the Floyd distance are the same.
Furthermore, by \cite[Proposition~6.1]{PY}, the visual distance and the Floyd distance are bi-Lipschitz.

The next couple of lemmas will be used later on.

\begin{lem}\label{ExtensionLem}
Suppose that $\Gamma$ admits a non-trivial Floyd boundary. Then there exist a finite set $F$ of elements and constants $c\ge 1, \delta>0$ with the following property:

for any two elements $g, h\in \Gamma$, there exists $f\in F$ such that $gfh$ labels a $c$-quasi-geodesic and 
$$
\max\{d(g, [e, gfh]), \; d(gf, [e, gfh])\}\le \epsilon
$$
and $\delta_g(e, gfh)\ge \delta$.
\end{lem}
\begin{proof}
Note that if the Floyd boundary of $\Gamma$ is nontrivial, then $\Gamma$ is not virtually cyclic, and every hyperbolic element is strongly contracting \cite{Yanggrowthtightness}.
Thus, the extension lemma in \cite{YangSCC} applies in this setting.
Namely, let $F$ any set of three independent hyperbolic elements. Set $F^n=\{f^n: f\in F\}$ for given $n\ge 1$. Then for any sufficiently large $n_0$, and for any $g,h\in \Gamma$ there exists $f\in F^{n_0}$ such that $gfh$ labels a $c$-quasi-geodesic for a uniform constant $c$. 

It remains to prove that $\delta_g(e, gf^{2n_0}h)$ has a uniformly lower bound $\delta$ when $n_0$ is large.   Indeed, since  every $ f\in F$ is a hyperbolic element with two distinct fixed points, there exists $\delta=\delta(F)>0$ such that $\delta_e(f^{-n_0}, f^{n_0})\ge \delta$ for any $f\in F$ and $n\ge 1$.
Since $gf^{2n_0}h$ labels a $c-$quasi-geodesic, we see that $d(e, f^{n_0}[e,h])$ and $d(e, f^{-n_0}[e,g^{-1}])$  increase to $\infty$  as $n\to\infty$.
By Lemma \ref{Karlssonlem}, we have  for $n\ge n_0$ $\delta_e(f^n,f^nh)<\delta/4$  and $\delta_e(f^{-n_0},f^{-n_0}g^{-1})<\delta/4$.
Thus, $\delta_e(f^{-n_0}g^{-1}, f^{n_0}h)>\delta/2$, and then $\delta_{gf^{n_0}}(e, gf^{2n_0}h)\ge \delta/2$. Consequently, there exists $\delta'=\delta'(\delta, n_0)$ such that $\delta_{g}(e, gf^{2n_0}h)\ge \delta'$.
\end{proof}


Floyd and Bowditch boundaries are visual: any two distinct points can be connected by a geodesic.  This enables us to define 
the notion of  shadows on both of them, that will be used in our arguments. 
Given $K>0$ and $x\in \Gamma$, let $\Pi_K(x)$ be the set of boundary points $\xi$ for which some geodesic between $e$ and $\xi$ intersects $B(x,K)$.
We call $\Pi_K(x)$ the big shadow at $x$ of width $K$. 
Balls and shadows are related by \cite[Lemmas~4.13, 4.14, 4.15]{PY}.
We prove here a slight generalization of these results.

\begin{lem}\label{diametershadow}
There exists $C$ such that the diameter of the big shadow $\Pi_K(g)$ is bounded by $CK\lambda^{|x|-K}$ for either the Floyd distance on the Floyd boundary or the shortcut distance on the Bowditch boundary.
\end{lem}

\begin{proof}
By~(\ref{shortcutsmallerFloyd}) we only need to give the proof for the Floyd boundary.
Let $\xi,\zeta$ in $\Pi_K(g)$ and let $[e,\xi]$ and $[e,\zeta]$ be two geodesics intersecting $B(x,K)$ at $y$ and $z$ respectively.
Then, following back $[e,\xi]$ from $\xi$ to $y$, then following a path from $y$ to $z$ that stays inside $B(x,K)$ and finally following the geodesic $[e,\zeta]$ from $z$ to $\zeta$ yields a path from $\xi$ to $\zeta$ of Floyd length at most
$$\sum_{k\geq |x|-K}\lambda^k+2K\lambda^{|x|-K}+\sum_{k\geq |x|-K}\lambda^k\leq CK\lambda^{|x|-K}.$$
This concludes the proof.
\end{proof}

\subsection{Transition points and Floyd geometry}\label{sectiontransitionpoints}
In contrast with hyperbolic groups, the Cayley graph of relatively hyperbolic groups is not Gromov hyperbolic anymore, so the thin triangle property and the Morse property do not hold in general. However, a certain kind of ``relative" Morse property persists and is manifested in a notion of transition points introduced in \cite{Hruska} (see also \cite{DrutuSapir, GePoGGD}). 

Recall that $\mathbb P=\{gP: g\in \Gamma, P\in\mathbb P_0\}$ is the collection of all parabolic cosets. 
\begin{defn}
Let $P\in\mathbb P$ be a parabolic coset and $\eta, L>0$ be fixed constants. A point $p$ on a geodesic $\alpha$ is called \textit{$(\eta, L)$-deep} in $P$ if $B(p, 2L)\cap \alpha\subseteq N_\eta(P)$. It is called \textit{$(\eta, L)$-transitional} if it is not \textit{$(\eta, L)$-deep} in any parabolic coset $P\in \mathbb P$.  
\end{defn}

According to the definition, it is clear that an $(\eta, L_1)$-transition point is an $(\eta, L_2)$-transition point for $L_1\le L_2.$ 
The parameters $\eta, L>0$ are usually chosen via the \textit{bounded intersection} property of the collection $\mathbb P$ (see \cite[Lemma~4.7]{DrutuSapir}): for any $\eta>0$ there exists $L=L(\eta)>0$ such that for any two $P\ne P'\in\mathbb P$ we have $\diam{N_\eta(P)\cap N_\eta(P')}\le L.$

\begin{lem}\label{EntryTransLem}\cite[Proposition~5.6]{GePoGGD}, \cite[Theorem~4.1]{DrutuSapir}.
For large enough $\eta$, there exists $L=L(\eta)$ such that any point of a geodesic $\alpha$ can  be $(\eta, L)$-deep in at most one  $P\in \mathbb P$. 
Moreover, if it is $(\eta, L)$-deep in $P$, the  entry and exit points of $\alpha$ at $N_\eta(P)$ are $(\eta, L)$-transitional. 
\end{lem}

The following result which refines Lemma~\ref{Karlssonlem} explains the application of the Floyd geometry in relatively hyperbolic groups. 
\begin{lem}\label{TransLargeLem}\cite[Corollary~5.10]{GePoGGD} 
For every large enough $\eta$, there exist $L=L(\eta)$ and $\delta=\delta(\eta)>0$ such that if $x$ is an $(\eta, L)$-transition point  on a geodesic $\alpha$ with endpoints $\alpha_-$ and $\alpha_+$, then $\delta_{x}(\alpha_-,\alpha_+)\ge \delta$. 
\end{lem}





The following is called the relative thin triangle property for transition points.
It is folklore and was proved at several places, using different terminology, see for instance \cite[Proposition~7.1.1]{GePoGGD}, \cite[Proposition~4.6]{Sistometricrelhyp}.
It can also be derived from \cite[Theorem~1.12]{DrutuSapir}, \cite[Section~8]{Hruska}, \cite[Proposition~3.15]{Osin}.
Usually, it is stated for points $x,y,z$ in the $\Gamma$.
The following version for points in the boundary is proved in \cite[Lemma~2.4]{DG20}.

\begin{lem}\label{ThinTransitional2}
For large enough $\eta$, there exist $L(\eta)$ such that for every $L\geq L(\eta)$, there exists $C=C(\eta,L)$ such that for every triple of points $(x,y,z)$ that are either conical limit points or elements of $\Gamma$, any $(\eta,L)$-transition point on one of the side of the geodesic triangle with vertices $x,y,z$ is within $C$ of an $(\eta,L)$-transition point on one of the two other sides.
\end{lem}
In what follows, we fix $\eta$ and $L(\eta)$ so that any pair $(\eta,L)$ satisfies the above lemmas for $L\geq L(\eta)$.

\subsection{Random walks on relatively hyperbolic groups}
Let $\Gamma$ be a relatively hyperbolic group and let $\mu$ be a finitely supported symmetric admissible probability measure on $\Gamma$.
Let $\rho$ be the spectral radius of the random walk and $R$ its inverse.
We collect here some results that will be used all along the paper.
Recall that $\eta$ and $L(\eta)$ are fixed such that for every $L\geq L(\eta)$, any $(\eta,L)$-transition point satisfies the results in Section~\ref{sectiontransitionpoints}.

A very useful set of inequalities relating the Green functions along geodesics were first proved   by Ancona \cite{Ancona} in hyperbolic groups and used to  identify  the Martin boundary   with the Gromov boundary.
These inequalities were recently extended up to the spectral radius by Gou\"ezel-Lalley in co-compact Fuchisan groups \cite{GouezelLalley} and by Gou\"ezel in general hyperbolic groups \cite{Gouezel}.
They state that there exists $C$, depending only on the hyperbolicity parameters of the group, such that for every $x,y,z\in \Gamma$ such that $y$ is on a geodesic from $x$ to $z$, for every $r\leq R$,
$$\frac{1}{C}G_r(x,y)G_r(y,z)\leq G_r(x,z)\leq C G_r(x,y)G_r(y,z).$$
The only non-trivial part is the upper-bound.
In relatively hyperbolic groups, a relative version of Ancona inequalities in terms of Floyd distance was obtained in \cite{GGPY} to establish a surjective map from the Martin boundary to the Floyd boundary.

Ancona inequalities are one of the main ingredient in \cite{SWX} for studying branching random walks on hyperbolic groups.
In the present paper, we will make very crucial use of the relative Ancona inequalities extended up to the spectral radius in \cite{DG21}.

\begin{prop}[Relative Ancona inequalities]\cite[Theorem~3.6]{DG21}
For every $L\geq L(\eta)$ and $K\geq0$, there exists $C=C(\eta, K)$ such that the following holds.
Let $x,y,z\in \Gamma$ and assume that $y$ is within $K$ of an $(\eta,L)$-transition point on $[x,z]$.
Then for every $r\leq R$, we have
$$\frac{1}{C}G_r(x,y)G_r(y,z)\leq G_r(x,z)\leq C G_r(x,y)G_r(y,z).$$
\end{prop}

Note that there also exists a strong form of relative Ancona inequalities in \cite{DG21}, although we will not need them in this paper.
The following result is one of the step into proving relative Ancona inequalities in \cite{DG21}.
It will be useful in this paper.
If $A\subset \Gamma$ and if $x,y\in \Gamma$, for every $r\leq R$ we denote by $G_r(x,y;A)$ the Green function from $x$ to $y$ restricted to trajectories staying in $A$, except maybe at the end points, i.e.
$$G_r(x,y;A)=\sum_{n\geq 0}\sum_{z_1, \ldots ,z_{n-1}\in A}\mu(x^{-1}z_1)\mu(z_1^{-1}z_2) \cdots \mu(z_{n-1}^{-1}y)r^n.$$

\begin{lem}\label{SupExpDecayLem}\cite[Proposition 3.5]{DG21}
For every $L\geq L(\eta)$, there exist  $\delta>1$ and $K_0>0$   such that the following holds.  
For every $x,y,z\in \Gamma$ such that  $y$ is an $(\eta, L)$-transition point on $[x,z]$ and for every $K\geq K_0$, we have
$$G_R(x,z; B(y,K)^c)\le e^{- e^{\delta K}}.$$
\end{lem}

Finally, we will also use the following result, proved in \cite{DG21}.
If $P\in \mathbb P$ is a parabolic coset, $\eta\geq 0$ and $x\in \Gamma$, we denote by  $$\pi_{N_\eta(P)}(x) :=\{y\in N_\eta(P): d(x,y)=d(x, N_\eta(P))\}$$ the set of its shortest projections on the $\eta$-neighborhood $N_\eta(P)$ of $P$. 
For $x,y\in \Gamma$, we set $$d_{N_\eta(P)}(x,y):=\diam{\pi_{N_\eta(P)}(x)\cup\pi_{N_\eta(P)}(y)}.$$
It follows from \cite[Corollary~8.2]{Hruska} that the shortest projection is coarsely Lipschitz: 
$$d_{N_\eta(P)}(x,y)\le kd(x,y)+k$$ for a fixed $k\ge 1$ depending only on $\eta$. Thus, $\pi_{N_\eta(P)}(x)$ has  bounded diameter.

\begin{lem}\label{lemmaexponentialmoments}
Let $P\in \mathbb P$ be a parabolic coset.
For every $M\geq 0$, there exists $\eta_0$ such that for $\eta\geq \eta_0$, we have
$$G_R(e,x;N_\eta(P)^c)\leq \mathrm{e}^{-Md_{N_\eta(P)}(e,x)}.$$
\end{lem}

\begin{proof}
In \cite[Lemma~4.6]{DG21} this result is stated for the first return kernel to $N_\eta(P)$, i.e. the quantity
$G_R(x_0,x;N_\eta(P)^c)$ where $x_0\in N_\eta(P)$, but the proof can be applied here.
Indeed it is shown without assuming that $x_0\in N_\eta(P)$ that the whole contribution of trajectories from $x_0$ to $x$ staying outside $N_\eta(P)$ is bounded by
$\mathrm{e}^{-l h(\eta)}G_R(x_0,x)$, where $\frac{1}{c}l\leq d_{N_\eta(P)}(x_0,x)\leq c l$ for some $c>0$ 
and $h(\eta)$ is a function of $\eta$ going to infinity as $\eta$ goes to infinity, see the before last equation of the proof of \cite[Lemma~4.6]{DG21}.
In particular, applying this to $x_0=e$, we get that for large enough $\eta$,
$$G_R(e,x;N_\eta(P)^c)\leq \mathrm{e}^{-M d_{N_\eta(P)}(e,x)}.$$
This concludes the proof.
\end{proof}



\section{The growth rate of the Green function}\label{SectiongrowthrateGreenfunction}

\subsection{Preliminary results}
Recall the following definitions from the introduction.
Let $\Gamma$ be a finitely generated group endowed with a finite generating set.
Let $\mu$ be a symmetric
probability measure whose finite support generates $\Gamma$. Denote by $\rho$ the spectral
radius of $\mu$ and by $R$ its inverse.
Set $p_n(x,y)=\mu^{\star n}(x^{-1}y)$ for $n\ge 1$ and $p_0(x,y)=\delta_x(y)$. Then, $R$ is the radius of convergence of the
Green function
$$
G_r(x, y) :=\sum_{n=0}^\infty r^np_n(x,y).
$$

For $1\le r\le R$, we consider   the sum of the $r$-Green function over the $n$-sphere
   $$H_r(n) := \sum_{x\in S_n} G_r(e,x)$$    
and define the \textit{growth rate  of the Green function} as follows 
$$\omega_\Gamma(r):= \limsup_{n\to \infty} \frac{\log H_r(n)}{n}$$
We first record a few simple facts about $H_r(n)$.
\begin{lem}\label{FactsHrLem}
The following statements are true:
\begin{enumerate}
    \item 
      There exists $C > 1$ such that for any $1 \leq  r \leq R$ and $n \geq 1$
    $$C^{-1} H_{r}(n+1)\le H_r(n) \le CH_{r}(n+1).$$
    \item
    There exists a constant $C>0$ such that $\frac{1}{C}\leq H_1(n)\le C$ for any $n\ge 1$. In particular, $\omega_{\Gamma}(1) = 0$. 
    \item
      The function $r \mapsto \omega_{\Gamma}(r)$ is strictly increasing on $[1,R]$ and continuous on $[1,  R)$. Consequently, $\omega_{\Gamma}(r) > 0$ for $r > 1$.
\end{enumerate}

\end{lem}
\begin{proof}
Since the random walk is irreducible and $\Gamma$-invariant, there exist 
$l \ge 1$ and a uniform number $p_0 > 0$ such that $p_l(x, y) > p_0$ for any $x,y\in \Gamma$ with $d(x,y)=1$. Thus, if $x\ne 1\ne y$, we have $p_n(e,x)\ge p_0\cdot  p_{n-l}(e,y)$ for $n\ge l$. This implies $G_r(e, x)\ge p_0 G_r(e,y)$. Thus, we have $H_{r}(n+1)\ge p_0 H_{r}(n)$.
For the other inequality, note that every  $y\in S_{n+1}$ is adjacent   to at most $N=\sharp S_1$ elements  $x\in S_{n}$.  
We then obtain $H_{r}(n+1)\le \frac{N}{p_0} H_{r}(n)$.
This proves the first statement.

We introduce the partial shadow $\widetilde{\Pi}_K(x)$ of width $K$ at $x$ as the set of limit points such that some geodesic $[e,\xi]$ intersects the ball $B(x,K)$ at a transition point. By the Shadow Lemma \cite[Proposition~4.4]{DG20} for harmonic measures, there exists $K>0$ such that $G_1(e,x)\asymp \nu(\widetilde{\Pi}_K(x))$ for any $x\in \Gamma$, which means that the ratio of these two quantities is bounded away from 0 and infinity.
Note that for each $n\ge 1$, any conical limit point can be covered in a uniform number of shadows at $x\in S_n$. We thus obtain that the sum $\sum_{x\in S_n} \nu(\widetilde{\Pi}_K(x))$ coarsely gives the measure of the whole set of conical limit points, so is uniformly bounded from above and below. The second statement follows.

Finally let us prove the third statement.
Let $1\le s<r$.
Since the random walk is finitely supported, there exists $c_1>0$ such that for every $x$, $p_m(e,x)=0$ for every $m\leq c_1|x|$.
Thus, we have
$$G_s(e,x)=\sum_{m\geq c_1 |x|}s^m p_m(e,x)
\leq \left(\frac{s}{r}\right)^{c_1 |x|}\sum_{m \geq c_1 |x|}r^m p_m(e,x)
=\left(\frac{s}{r}\right)^{c_1 |x|}G_r(e,x)$$ and so
\begin{equation}\label{exponentialdifferenceGreen}
G_s(e,x)\leq \left(\frac{s}{r}\right)^{c_1 |x|}G_r(e,x).
\end{equation}
Thus, $H_s(n) \leq \left( \frac{s}{r} \right)^{c_1 n} H_r(n)$ and $\omega_{\Gamma}(r) -  \omega_{\Gamma}(s) \ge c_1 \left( \log r  - \log s\right)$.  Therefore $\omega_{\Gamma}(r)$ is strictly increasing on $[1, R]$.

  For $\delta > 0$ we choose $c_2$ so that $v \le c_2 \left( \log R - \log (R - \delta) \right)$.
  Note that since the underlying random walk is symmetric, for every $x$ and every $m$, we have
  $p_m(e,x)p_m(e,x)\leq p_{2m}(e,e)$ and by \cite[Lemma~1.9]{Woessbook}, $p_{2m}(e,e)\leq R^{-2m}$.
  Thus,
  $p_m(e, x) \leq R^{-m}$ for every $x \in \Gamma$ and $n \ge 0$.
  Consequently, by Lemma~\ref{purelyexponentialgrowth}, we have for $1 \le r \le R - \delta$,
\[
\sum_{|x| = n} \sum_{m > c_2 n} r^m p_m(e, x) \le c\mathrm{e}^{v n} \sum_{m > c_2 n} \left( \frac{r}{R} \right)^m \le c_3 
\]
for some constant $c_3 > 0$. Now for $1 \le s < r \le R - \delta$,
\begin{align*}
      H_s(n) = \sum_{|x| = n} \sum_{m = 0}^{\infty} s^m p_m(e, x)
  &\ge \left( \frac{s}{r} \right)^{c_2 n} \sum_{|x| = n} \sum_{m = 0}^{c_2 n} r^m p_m(e, x)\\
  &\ge \left( \frac{s}{r} \right)^{c_2 n} \left( H_r(n) - c_3 \right). 
\end{align*}
It follows that $\omega_{\Gamma}(r) - \omega_\Gamma(s) \le c_2 \left( \log r - \log s \right)$. Since $\delta > 0$ is arbitrary, we prove that $\omega_{\Gamma}(r)$ is continuous in $1 \le r < R$.
\end{proof}

Our goal in the remainder of this section is to compare $H_r(n)$ with $\mathrm{e}^{n \omega_\Gamma(r)}$.

\subsection{Upper bound} 
We start with the following lemma.

\begin{lem}\label{UpperBDLem}
There exists  constants $C,C'>1$ such that for any $1\le r\le R$ and any  integer $n, m\ge 1$, we have
$$H_r(m) H_r(n) \le C H_{r}(m+n) ,$$
 and $H_r(n)\le C' e^{n \omega_\Gamma(r)}$. 
In particular, the following limit exists
$$\omega_\Gamma(r)= \lim_{n\to \infty} \frac{\log H_r(n)}{n}=\sup_n \frac{\log C^{-1}H_r(n)}{n}.$$
\end{lem}

Set $l=\max\{d(e,f): f\in F\}+4\epsilon$, where $F$ is a finite set given by Lemma \ref{ExtensionLem}.
and let
$A(n, l)=\bigcup\limits_{-l\le i\le l}S_{n+i}$ be the annulus of width $l$ and radius $n$ and
define $$\Phi: S_m\times S_n\to A(n+m, l)$$ by setting $\Phi(g,h)=gfh$, where $f\in F$ is provided by Lemma~\ref{ExtensionLem}.
\begin{lem}\label{BoundedToOne}
The map $\Phi$ is uniformly bounded to one.
\end{lem}
\begin{proof}
Indeed, assume that $\Phi(g,h)=gf_1h$ and $\Phi(x,y)=xf_2y$ for $f_1,f_2\in F$. If $\Phi(g,h)=\Phi(x,y)$, we obtain $d(g,x), d(h^{-1},y^{-1})\le 4\epsilon$ from Lemma \ref{ExtensionLem}. Thus, there are at most $3 (\sharp B(e,4\epsilon))^2$ pairs of  elements $(x,y)$ such that  $\Phi(g,h)=\Phi(x,y)$.
\end{proof}
\begin{proof}[Proof of Lemma \ref{UpperBDLem}]
For any $g\in S_n, h\in S_m$, we have 
$$
G_r(e,g) G_r(e,h) \le c_1 G_r(e,g) G_r(e,f) G_r(e,h)\le c_2 G_r(e, gfh) 
$$
Thus, 
$$
H_r(n) H_r(m) \le c_2\sum_{z\in A(n+m,l)} G_r(e,z).
$$
Note that $c_3^{-1} H_{r}(n+1)\le H_r(n) \le c_3H_{r}(n+1)$ by Lemma \ref{FactsHrLem}.
Thus, we have
$$H_r(n) H_r(m) \le c_4 H_r(n+m).$$
This proves the first part of the lemma.
The second part follows from the Fekete sub-additive lemma.
\end{proof}

\begin{cor}\label{lowersemicontinuity}\label{corolowersemicontinuous}
The function $r\in [1,R]\mapsto \omega_\Gamma(r)$ is increasing and continuous. 
\end{cor}

\begin{proof}
We proved in Lemma \ref{FactsHrLem} that $\omega_{\Gamma}$ is increasing and continuous on $[1, R)$.
By Lemma~\ref{UpperBDLem}, it can be expressed as a supremum of continuous functions, hence it is lower semi-continuous by \cite[Lemma~2.41]{AliprantisBorder}.
We deduce that it is left continuous at $R$.
\end{proof}

\begin{cor}\label{coroboundv/2}
For every $r\leq R$, $\omega_\Gamma(r)\leq v/2$.
\end{cor}

\begin{proof}
By the Cauchy-Schwarz inequality,
$$\left(\sum_{x\in S_n}G_r(e,x)\right)^2\leq |S_n|\sum_{x\in S_n}G_r(e,x)^2.$$
For any $r<R$, by \cite[Proposition~1.9]{GouezelLalley},
$$c(r)=\sum_{x\in \Gamma}G_r(e,x)^2<\infty,$$
so
$$\omega_\Gamma(r)\leq \limsup \frac{1}{2n}\left(\log c(r)+\log |S_n|\right)=\frac{1}{2}v.$$
This proves the desired inequality for $r<R$.
By Corollary~\ref{lowersemicontinuity}, $r\mapsto \omega_{\Gamma}(r)$ is 
continuous, so the inequality also holds at $R$.
\end{proof}

\subsection{Lower bound via parabolic gap} 

For $1\le r\le R$ and $s\ge 0$, we consider the following Poincar\'e series:

\begin{equation}\label{defPoincareseries}
\Theta_{r, s}(\Gamma) := \sum_{h\in \Gamma} G_r(e, h) e^{-sd(e, h)}.
\end{equation}
 We can rearrange the terms in $\Theta_{r, s}(\Gamma)$ as follows

$$\Theta_{r, s}(\Gamma) = \sum_{n\ge 0} H_r(n) e^{-sn}$$
so that  $H_r(n)$ appears in the place of $\sharp S_n$ in the usual Poincare series. Thus, for each $r$ fixed, $\omega_\Gamma(r)$ is the exponential radius of convergence of the series $\Theta_{r, s}(\Gamma)$ in $s$.

Similarly, for any subgroup $P\subset \Gamma$, we can consider the associated Poincar\'e series
$$\Theta_{r, s}(P)=\sum_{n\geq 0}\sum_{x\in P\cap S_n}G_r(e,x)\mathrm{e}^{-sd(e,x)}$$ and its growth rate
$$ \omega_P(r)=\limsup \frac{1}{n}\log \sum_{x\in P\cap S_n}G_r(e,x),$$
which is the exponential radius of convergence of the series $s\mapsto \Theta_{r,s}(P)$.

\begin{defn}
We say that $\Gamma$ has \textit{parabolic gap}   for the Green functions if for every parabolic subgroup $P\in \mathbb{P}_0$, $\omega_P(r)<\omega_\Gamma(r)$ for every $1< r\le R$.
\end{defn}

\begin{lem}\label{LowerBDLem}
Suppose that $\Gamma$ has a parabolic gap for the Green function.
For every $1<r\leq R$, there exists $C=C(r)>1$ such that 
$$H_{r} (n+m) \le C H_r(n) H_r(m)
$$
and there exists a constant $C'=C'(r)>1$ such that for any $1<r\le R_\mu$, we have 
$$
\frac{1}{C'} e^{n \omega_\Gamma(r)}\le H_r(n)\le C' e^{n \omega_\Gamma(r)}
$$
for every $n\ge 1$.
\end{lem}

\begin{proof}
Let $\mathbb P_0$ be the finite set of maximal parabolic subgroup up to conjugacy.
Set $$K_r(n) := \max_{P\in\mathbb P_0}\sum_{x\in S_n\cap P} G_r(e,x).$$

\textbf{Step 1.} First, the following holds  
\begin{equation}\label{threeunion}
H_{r}(n+m) \le \sum_{0\le k \le n} \sum_{0\le j \le m}  c_0 H_{r}(k) \cdot H_{r}(j) \cdot   K_{r}(n+m-k-j)  
\end{equation}
for any $n, m\ge 0$. 

Indeed, for given $x\in S_{n+m}$, consider a  geodesic $\gamma=[e,x]$ and the point $y\in \gamma$ such that $d(e,y)=n$. If $y$ is a transition point, then by the relative Ancona inequalities, we have 
$$
G_r(e,x) \le C G_r(e, y) G_r(e, y^{-1}x).
$$
This corresponds to the case $k=j=0$.

If $y$ is deep in some $P$-coset $X$, let $u,v$ be the entry and exit points of $\gamma$ in $N_\eta(X)$ (possibly $u=e$ or $v=x$). Then $u,v$ are transition points so again the relative Ancona inequalities show   
$$
G_r(e,x) \le C G_r(e, u) G_r(e, u^{-1}v)  G_r(e, v^{-1}x).
$$

Summing up $G_r(e,x)$ over all $x\in S_n$ according to $k=d(u,y)$ and $l=d(y,v)$, we obtain (\ref{threeunion}).

\textbf{Step 2.} 
By assumption,  $\omega_P(r) < \omega_\Gamma(r)$ for every $P\in \mathbb P_0$. Then for any given $\omega_\Gamma(r)  > \omega> \omega_P(r)$, there exists $c_1=c_1(r)>0$ such that $K_r(i)\le c_1 e^{i \cdot \omega }$ for any $i\ge 1$.
For $\omega>0$, we define $$a^\omega(n)= e^{-\omega n} \cdot H_{r}(n).$$ 
Then a re-arrangement of (\ref{threeunion}) gives rise to the form as follows:  
\begin{equation}\label{SubMulEQ}
a^\omega(n+m) \le c_1 \left( \sum_{1\le k\le n} a^\omega(k) \right) \cdot \left(\sum_{1\le j \le m} a^\omega(j) \right),
\end{equation}
for any $n, m \ge 0$.
We conclude as in \cite[Theorem 5.3]{YangSCC}.
\end{proof}

Theorem~\ref{mainthmparabolicgap} is now a consequence of Lemma~\ref{FactsHrLem} for the lower bound at $r=1$, Lemma~\ref{LowerBDLem} for the upper bound and Lemma~\ref{UpperBDLem} for the lower bound at $r>1$.



\subsection{Criteria for Green parabolic gap}
A possible way to get this parabolic gap is the following divergence criterion based on \cite[Lemma~2.23]{YangSCC}.

Let $A, B$ be two subsets of $G$.
Denote by $\mathbb W(A, B)$ the set of all words over the alphabet set $A\sqcup B$ with letters alternating in $A$ and $B$. 

\begin{lem}\label{degrowth}
Assume that there exists an injective  map $\iota:  \mathbb W(A)\to  \mathbb W(A, B)$ such that the evaluation map $\Phi: \mathbb W(A, B) \to G$ is injective on the subset $\iota(\mathbb W(A))$ as well. Set $X:=\Phi(\iota(\mathbb W(A,B))$  and assume that   $B$ is finite. Then   $\Theta_{r,s}(A)$ converges  at $s=\omega_{X}(r)$. In particular, if $\Theta_{r,s}(A)$ diverges at $s=\omega_{A}(r)$ 
then $\omega_{X}(r)
> \omega_{A}(r)$.
\end{lem}

\begin{proof}
Since $\Phi: \mathbb W(A, B) \to G$ is injective, each element in the image $X$ has a unique  alternating product form over $A\sqcup B$. Set  $C_1=\max_{b\in B} \{d(e, b)\}<\infty$, $C_2=\min_{b\in B} \Big\{\frac{G_r(1, b)}{G_r(1,1)}\Big\}<\infty$ where  $B$ is a finite set by assumption. For a word $W=a_1b_1a_2\cdots a_nb_n \in \mathbb W(A, B)$, we have $$
d(e, a_1b_1\cdots a_nb_n) \le \sum_{1\le i\le n} \big (d(e, a_i)+C_1\big),
$$and 
$$
G_r(e, a_1b_1\cdots a_nb_n) \ge C_2^n \prod_{1\le i\le n} \big (G_r(e, a_i) \big).
$$
As a consequence, we estimate the Poincar\'e series of $X$ as follows:
$$
\begin{array}{rl}
&\sum\limits_{g \in X} G_r(e,g)  e^{-s\cdot d(e, g)} \\
 
\ge &   \sum\limits_{n =1}^{\infty} \left(\sum\limits_{a \in A} G_r(e,a) e^{-s\cdot d (e, a)}\right)^n\cdot   (C_2e^{-sC_1}) ^n.
\end{array}
$$

By contradiction, assume that $\sum\limits_{a \in A} G_r(e,a) e^{-\omega_X(r) \cdot d (e, a)} = \infty$. Then there exists some $s >
\omega_X(r)$ such that $\sum\limits_{a \in A} G_r(e,a)  e^{-s\cdot d (e,
a)}\cdot C_2e^{-sC_1}
> 1$. By the above estimates, this implies the  series $\Theta_X(r,s)$ diverges at $s$, so $s\le \omega_X(r)$, which is a contradiction.
\end{proof}

\begin{cor}\label{parabolicdivergentgrowthgap}
Let $P\in \mathbb{P}_0$. 
If the series $$\Theta_{r,s}(P)=\sum_{p\in P} G_r(e,p) e^{-s d(e,p)} $$ diverges at $s=\omega_P(r)$ then $\omega_P(r)<\omega_\Gamma(r)$.
\end{cor}
\begin{proof} 
For any hyperbolic element $h$ and sufficiently large  integer $n\ge 1$, the subgroup generated by $P$ and $h^n$ is a free product $P* \langle h^n\rangle$. 
Thus, the evaluation map $\Phi: \mathbb W(A, B) \to G$ is injective on $\mathbb W(A)$ seen as a subset of $\mathbb W(A,B)$, where $A=P$ and $B=\{h^n,h^{-n}\}$.
\end{proof}


Here is a second criterion for the parabolic gap.

\begin{prop}\label{secondcriterionparabolicgap}
Let $P\in \mathbb P_0$.
For every $1<r\leq R$, if $\omega_P(r)\leq 0$, then $\omega_P(r)<\omega_\Gamma(r)$.
\end{prop}

\begin{proof}
By Lemma~\ref{FactsHrLem}, for every $1<r$, $\omega_\Gamma(r)>0$.
\end{proof}

\subsection{Examples}\label{sectionexamples}
We now give various examples of different possible situations.
In Example~A, we give a criterion for having $\omega_P(r)\leq 0$, which automatically implies that $\omega_P(r)<\omega_\Gamma(r)$.
In Example~B, we construct an example where $\omega_P(r)>0$, but we still have $\omega_P(r)<\omega_\Gamma(r)$.
Example~C is devoted to construct an example where $\omega_P(r)=\omega_\Gamma(r)$, assuming that $P$ satisfies some properties.

We first recall some terminology from \cite{DG21}.
Let $r\leq R$ and let $P$ be a parabolic subgroup.
Denote by $p^{r,P}$ the first return transition kernel to $P$ associated with $r\mu$, i.e.\ for every $x,y\in P$,
\begin{equation}\label{deffirstreturnkernel}
p^{r,P}(x,y)=\sum_{n\geq 1}\hspace{.1cm}\sum_{z_1,...,z_n\notin P}r^n\mu(x^{-1}z_1)\mu(z_1^{-1}z_2)...\mu(z_n^{-1}y).
\end{equation}
Also denote by $G_t^{r,P}$ its associated Green function at $t$, i.e.
$$G_t^{r,P}(x,y)=\sum_{n\geq 0}t^np_n^{r,P}(x,y),$$
where $p^{r,P}_n$ is the $n$th convolution power of $p^{r,P}$.
Finally, denote by $R_P(r)$ the inverse of the spectral radius of $p^{r,P}$.
Then by \cite[Lemma~4.4]{DG21}, $G_1^{r,P}(e,x)=G_r(e,x)$ which is finite, so in particular $R_P(r)\geq 1$.

\begin{defn}\label{defnspectraldegeneracy}
We say that the random walk is \textit{spectrally degenerate} along $P$ if $R_P(R)=1$.
It is called \textit{spectrally non-degenerate} if it is not spectrally degenerate along any parabolic subgroup.
\end{defn}

\subsubsection{Example~A}
\begin{prop}
Let $P$ be a parabolic subgroup.
Assume that $P$ is amenable.
Then for every $1<r< R$, $\omega_P(r)< 0$.
In particular, for every $1<r< R$, $\omega_P(r)<\omega_\Gamma(r)$.
Moreover, if the random walk is not spectrally degenerate along $P$, then $\omega_P(R)\leq 0<\omega_\Gamma(R)$.
\end{prop}

\begin{proof}
Fix $r<R$ and choose $s$ such that $r<s<R$.
Denote by $p^{s,P}$ the first return transition kernel to $P$ associated with $s\mu$ and by $G^{s,P}$ its associated Green function as above.
By \cite[Lemma~4.15]{DG21}, the spectral radius of $p^{s,P}$ is strictly less than 1, because $s<R$.
Since the underlying random walk driven by $\mu$ is symmetric, $p^{s,P}$ is a symmetric $P$-invariant transition kernel on $P$.
By amenability, this transition kernel is necessarily sub-Markov, see \cite[Corollary~12.5]{Woessbook}.
Let $t<1$ be such that $\tilde{p}^s=\frac{1}{t}p^{s,P}$ is Markov.
Then, letting $\widetilde{G}^s$ be the Green function associated with $\tilde{p}^s$, for any $x\in P$ we have $G_1^{s,P}(e,x)=\widetilde{G}_t^s(e,x)$.
We now prove that
\begin{equation}\label{equationGreenparabolicfinite}
\sum_{x\in P}\widetilde{G}_t^s(e,x)<\infty.
\end{equation}
Consider the convolution operator
$$\widetilde{P}_s:f\mapsto \left ( x\mapsto \sum_{y}\tilde{p}^s(x,y)f(y)\right )$$
acting on the space of $\ell_1$ functions, i.e.\ summable functions on $P$.
The $\ell_1$-norm of $\widetilde{P}_s$ is 1, so its $\ell_1$-spectral radius is bounded by 1.
In particular, $1/t$ is bigger than the $\ell_1$-spectral radius, so by definition, $I-t\widetilde{P}_s$ is invertible in the space of summable functions.
Moreover, the inverse is of the form 
$$\widetilde{Q}_s=\sum_{n\geq 0} t^n\widetilde{P}_s^n.$$
Consider the function $f$ defined by $f(x)=1$ if $x=e$ and 0 otherwise.
Then, $\widetilde{Q}_sf(x)=\widetilde{G}_t^s(e,x^{-1})$.
Since the function $f$ is summable, $\widetilde{Q}_sf$ also is summable, and so the Green function $\widetilde{G}^s$ at $t$ is summable.
This proves~(\ref{equationGreenparabolicfinite}).

Since for $x\in P$, $G_s(e,x)=G_1^{s,P}(e,x)=\widetilde{G}_t^s(e,x)$,
we deduce that $\sum_{x\in P}G_s(e,x)<\infty$ and so
$\omega_P(s)\leq 0$.
Moreover, by~(\ref{exponentialdifferenceGreen}), we have
$$G_r(e,x)\leq \left(\frac{r}{s}\right)^{c|x|}G_s(e,x)$$ for every $x\in \Gamma$.
Summing over $P\cap S_n$, we see that $\omega_P(r)<\omega_P(s)$ and so $\omega_P(r)<0$.

Finally, if the random walk is not spectrally degenerate along $P$, then by definition the spectral radius of $p^{R,P}$ is strictly less than 1.
The same proof shows that $\omega_P(R)\leq 0$.
\end{proof}
 
 
 When the parabolic subgroup $P$ has sub-exponential growth, we do not need spectral non-degeneracy to get that $\omega_P(R)\leq 0$.
 \begin{prop}\label{subexpgrowthparabolicgap}
 Let $P$ be a maximal parabolic subgroup.
Assume that $P$ has sub-exponential growth.
Then for every $1\leq r\leq R$, $\omega_P(r)\leq 0$.
In particular, for every $1<r\leq R$, $\omega_P(r)<\omega_\Gamma(r)$.
 \end{prop}
 
 \begin{proof}
 Let $s>0$.
There exists $C\geq 0$ such that for every $x\in \Gamma$, $G_R(e,x)\leq C$.
 In particular, for every $r\leq R$,
 $$\Theta_{r,s}(P)\leq C\sum_{n\geq 0}\sharp (P\cap S_n)\mathrm{e}^{-sn}<+\infty.$$
 Thus, $\Theta_{r,s}(P)$ is finite for every positive $s$ and so $\omega_P(r)\leq 0$.
 \end{proof}
 
 \subsubsection{Adapted random walks on free products}
Before giving other examples, let us briefly recall some terminology and basic properties of random walks on free products.
Consider a free product $\Gamma=\Gamma_0*\Gamma_1$.
Let $\mu_0$ be an admissible probability measure on $\Gamma_0$, $\mu_1$ an admissible probability measure on $\Gamma_1$ and define $$\mu_\alpha=\alpha \mu_1+(1-\alpha)\mu_0.$$
Then $\mu_\alpha$ is an admissible probability measure on $\Gamma$.
Moreover, if both $\mu_0$ and $\mu_1$ are symmetric, respectively finitely supported, then so is $\mu_\alpha$.

Such a probability measure is called adapted to the free product structure and it can only move inside one of the free factors $\Gamma_0$ or $\Gamma_1$ at each step.
Adapted probability measures on free products have been considered by many authors, see \cite{Cartwright1}, \cite{Cartwright2}, \cite{CandelleroGilch}, \cite{CandelleroGilchMuller}, \cite{WoessMartin} and \cite{WoessLLT} for instance.
For convenience we will assume that the random walk driven by $\mu_i$ on $\Gamma_i$ is transient at the spectral radius, i.e.\ the Green function is finite at its radius of convergence.
This is not very restrictive, since by the Varapoulos Theorem, only groups with quadratic growth, i.e.\ groups that are virtually $\mathbb{Z}$ or $\mathbb Z^2$, can carry an admissible random walk which is not transient at the spectral radius.
Actually, Varapoulos \cite{Varapoulos} proved that only groups with quadratic growth can carry a non-transient (at $r=1$) random walk, but a standard argument due to Guivarc'h \cite{Guivarch} allows one to reduce non-transience at $r=R$ to non-transience at $r=1$ by the use of a suited $h$-process, see \cite{Woessbook} for more details and in particular \cite[Theorem~7.8]{Woessbook} for a complete proof.

Denote by $G^{\mu_i}$ the Green function on $\Gamma_i$ associated with $\mu_i$, $i=0,1$ and by $R_{\mu_i}$ the radius of convergence of $G^{\mu_i}$, i.e.\ the inverse of the spectral radius of $\mu_i$.
Also, as in the previous example, denote by $p^{s,\Gamma_0}$ the first return kernel to $\Gamma_0$ associated with $s\mu_\alpha$ and by $G_t^{s,\Gamma_0}$ the associated Green function.
We first relate $G^{\mu_0}$ and $G^{s,\Gamma_0}$ which are two Green functions associated with different transition kernels on the same group $\Gamma_0$.
Because $\mu_\alpha$ is adapted to the free product structure, the first return kernel $p^{s,\Gamma_0}$ can be written as
$$p^{s,\Gamma_0}(e,x)=(1-\alpha)s\mu_0+w_{s,\alpha}\delta_{e,x},$$
where $w_{s,\alpha}$ is the weight of the first return to $e$, starting with a step in $\Gamma_1$.
Thus, \cite[Lemma~9.2]{Woessbook} shows that for any $x,y\in \Gamma_0$,
\begin{equation}\label{superWoess}
G_t^{s,\Gamma_0}(x,y)=\frac{1}{1-w_{s,\alpha}t}G_\frac{(1-\alpha)st}{1-w_{s,\alpha}t}^{\mu_0} (x,y).
\end{equation}
Define
$$\zeta_0(s)=\frac{(1-\alpha)s}{1-w_{s,\alpha}}.$$
As in the previous example, by \cite[Lemma~4.4]{DG21}, $G_1^{s,\Gamma_0}(e,x)=G_s(e,x)$, where as usual $G_s$ denotes the Green function associated with the initial random walk driven by $\mu_\alpha$ at $s$.
So in particular, applying~(\ref{superWoess}) at $t=1$,
\begin{equation}\label{superWoess2}
G_s(x,y)=G_1^{s,\Gamma_0}(x,y)=\frac{1}{1-w_{s,\alpha}}G_{\zeta_0(s)}^{\mu_0}(x,y).
\end{equation}

We also set
$$\zeta_1(s)=\frac{\alpha s}{1-w'_{s,\alpha}},$$
where $w'_{s,\alpha}$ is the weight of the first return to $e$, starting with a step in $\Gamma_0$. 
Then, by symmetry, we also have for any $x,y\in \Gamma_1$,
\begin{equation}\label{superWoess2'}
G_s(x,y)=\frac{1}{1-w'_{s,\alpha}}G_{\zeta_1(s)}^{\mu_1}(x,y).
\end{equation}

Let $R_\alpha$ be the inverse of the spectral radius of $\mu_\alpha$ and let $R_{\mu_0}$ be the inverse of the spectral radius of $\mu_0$ on $\Gamma_0$.
\begin{lem}\label{lemalphato0}
As $\alpha$ tends to 0, $R_\alpha$ and $\zeta_0(R_\alpha)$ both converge to $R_{\mu_0}$.
Moreover, $w_{r,\alpha}$ converges to 0 and the convergence is uniform in $r\leq R_\alpha$.
\end{lem}

\begin{proof}
We first show that $R_\alpha$ is a continuous function of $\alpha$.
Let $P_\alpha$ be the convolution operator
$$P_\alpha:f\mapsto \left (x\mapsto \sum_{y}\mu_\alpha(x^{-1}y)f(y)\right).$$

\begin{claim}
In the $\ell^2$ operator norm, $P_\alpha$ is a continuous function of $\alpha$.
\end{claim}

\begin{proof}[Proof of the claim]
We first show that $P_\alpha$ is continuous in the $\ell_1$ and $\ell_{\infty}$ operator norms.
Fix $\alpha_0$.
Note that $\mu_\alpha$ has a finite support included in a fixed finite set $\Sigma$, so $\mu_\alpha(x)$ uniformly converges to $\mu_{\alpha_0}(x)$ as $\alpha$ tends to $\alpha_0$.
Let $f$ be an $\ell_\infty$ function.
Then, for every $x$,
$$\left |\bigg (P_\alpha-P_{\alpha_0}\bigg)f(x)\right |\leq \sum_{y\in \Gamma}\left |\mu_{\alpha}(x^{-1}y)-\mu_{\alpha_0}(x^{-1}y)\right|\|f\|_{\infty}.$$
The only possible $y$ such that we do not have $\mu_{\alpha_0}(x^{-1}y)=0$ and $\mu_\alpha(x^{-1}y)=0$ are in $x\Sigma$.
This yields
$$\sum_{y\in \Gamma}\left |\mu_{\alpha}(x^{-1}y)-\mu_{\alpha_0}(x^{-1}y)\right|=\sum_{y'\in \Sigma}\left |\mu_{\alpha}(y')-\mu_{\alpha_0}(y')\right |.$$
Therefore, for every $\epsilon>0$, if $\alpha$ is close enough to $\alpha_0$,
$$\left |\bigg (P_\alpha-P_{\alpha_0}\bigg)f(x)\right |\leq\epsilon \|f\|_\infty.$$
This proves continuity for the $\ell_\infty$ operator norm.
Now, let $f$ be an $\ell_1$ function. Then,
$$\left \|\bigg(P_\alpha-P_{\alpha_0}\bigg)f\right \|_1\leq \sum_{x,y\in \Gamma}\big|\mu_\alpha(y^{-1}x)-\mu_{\alpha_0}(y^{-1}x)\big||f(y)|.$$
Fix $y$.
Then, as above,
$$\sum_{x\in \Gamma}\big|\mu_\alpha(y^{-1}x)-\mu_{\alpha_0}(y^{-1}x)\big|=\sum_{x'\in \Sigma}\big|\mu_\alpha(x')-\mu_{\alpha_0}(x')\big|.$$
Let $\epsilon>0$.
Then if $\alpha$ is close enough to $\alpha_0$ this last sum is bounded by $\epsilon$, independently of $y$.
Consequently,
$$\left \|\bigg(P_\alpha-P_{\alpha_0}\bigg)f\right \|_1\leq \epsilon  \sum_{y\in \Gamma}|f(y)|=\epsilon \|f\|_1.$$
This proves continuity for the $\ell_1$ operator norm.
The Riesz-Thorin interpolation theorem \cite[Theorem~(6.27)]{Folland} shows that $P_\alpha$ converges to $P_{\alpha_0}$ for the $\ell_p$ operator norm, for every $1\leq p\leq +\infty$.
This proves the claim, taking $p=2$.
\end{proof}

Since $R_\alpha$ is the inverse of the spectral radius of $P_\alpha$, continuity of $R_\alpha$ follows.
To conclude it thus suffices to prove that $w_{r,\alpha}$ uniformly converges to 0.
By \cite[Proposition~9.18]{Woessbook},
$$w_{r,\alpha}=\sum_{n\geq 1}\mathbf P_\alpha\left (X_n=e,X_k\neq e,1\leq k<n, \text{ first step chosen using }\alpha\mu_1\right )r^n$$
and by the Markov property,
$$w_{r,\alpha}=r\alpha\mu_1(e)+r\sum_{x\neq e\in \Gamma_1}\mathbf P_\alpha(X_1=x)\sum_{n\geq 1}\mathbf P_\alpha\left (X_n=x^{-1},X_k\neq x^{-1},k<n\right )r^n.$$
Set
\begin{equation}\label{defF}
F_r(e,x)=\sum_{n\geq 1}\mathbf P_\alpha\left (X_n=x,X_k\neq x,k<n\right )r^n
\end{equation}
and for $i=0,1$, $x\in \Gamma_i$,
\begin{equation}\label{defFi}
F_r^i(e,x)=\sum_{n\geq 1}\mathbf P_{\mu_i}\left (X_n=x,X_k\neq x,k<n\right )r^n.
\end{equation}
Then, by \cite[Lemma~1.13~(b)]{Woessbook}, for $x\in \Gamma$,
$$G_r(e,x)=F_r(e,x)G_r(e,e)$$
and for $x\in \Gamma_i$,
$$G_r^i(e,x)=F_r^i(e,x)G_r^i(e,e).$$
Thus, by~(\ref{superWoess2}), we have
$$F_r(e,x)=F_{\zeta_i(r)}^i(e,x)$$
and so we recover \cite[Proposition~9.18]{Woessbook}.
In particular,
$$w_{r,\alpha}\leq \alpha R_\alpha \left (\mu_1(e)+\sum_{x\in \Gamma_1}\mu_1(x)F_{R_{\mu_1}}^1(e,x^{-1})\right ).$$
Finally, $F_{R_{\mu_1}}^1(e,x^{-1})$ is finite since we assume that the random walk on $\Gamma_1$ is transient at the spectral radius.
Thus, there exists a constant $C$ such that
$$w_{r,\alpha}\leq \alpha C$$
and so $w_{r,\alpha}$ uniformly converges to 0.
\end{proof}

\begin{lem}\label{lemalphato1}
There exists $\alpha_0$ such that for $\alpha\in [0,\alpha_0]$, the quantity $w'_{r,\alpha}$ stays bounded away from 1.
In particular, as $\alpha$ tends to 0, $\zeta_1(r)$ converges to 0 and
the convergence is uniform in $1\leq r\leq R_\alpha$.
\end{lem}

\begin{proof}
By \cite[Proposition~9.18]{Woessbook},
$$w'_{r,\alpha}\leq U(r)=\sum_{n\geq 1}\mathbf P_\alpha (X_n=e,X_k\neq e,1\leq k<n)r^n.$$
Also, by \cite[Lemma~1.13]{Woessbook},
$$G_r(e,e)=\frac{1}{1-U(r)}.$$
Note that $G_r(e,e)$ depends on $\alpha$, but according to Lemma~\ref{lemalphato0} and~(\ref{superWoess2}), $G_r(e,e)$ is uniformly bounded for $\alpha$ in a fixed neighborhood of 0.
Consequently, $U(r)$ stays bounded away from 1.
\end{proof}

\subsubsection{Example~B}

Here is now an example where $\omega_P(r)>0$ and $\omega_P(r)<\omega_\Gamma(r)$.
Consider the group
$\Gamma=\mathbb F_2*\mathbb{Z}^d$, where $\mathbb F_2$ is the free group with two generators.
In other words, with the notations above, we set $\Gamma_0=\mathbb F_2$ and $\Gamma_1=\mathbb{Z}^d$.
Then $\Gamma$ is hyperbolic relative to $\Gamma_0$ and $\Gamma_1$.

We choose $d\geq 3$ so that every finitely supported admissible random walk on $\Gamma_1$ is transient at the spectral radius.
Choose an adapted measure $\mu_\alpha=\alpha\mu_1+(1-\alpha)\mu_0$ as above.
Then, $\mu_0$ is a probability measure on the non-amenable group $\mathbb F_2$ whose finite support generates $\mathbb F_2$,
hence, $R_{\mu_0}>1$.
According to Lemma~\ref{lemalphato0}, we can thus fix $\alpha$ so that $\zeta_0(r)>1$ for every $r$ in a neighborhood of $R_\alpha$, say $r_0< r\leq R_\alpha$.
Now that $\alpha$ is fixed, we omit it in the notations.
By~(\ref{superWoess2}),
$$\sum_{x\in \mathbb F_2\cap S_n}G_r(e,x)=\frac{1}{1-w_r}\sum_{x\in \mathbb F_2\cap S_n}G_{\zeta_0(r)}^{\mu_0}(e,x).$$
Since $\zeta_0(r)>1$, \cite[Note~1.7]{GouezelLalley} shows that
$\sum_{x\in \mathbb F_2\cap S_n}G_{\zeta_0(r)}^{\mu_0}(e,x)$ diverges as $n$ tends to infinity.
In particular, we see that $\omega_{\mathbb F_2}(r)\geq 0$.
On the other hand, by~(\ref{exponentialdifferenceGreen}), for $s>r$, $\omega_{\mathbb F_2}(s)>\omega_{\mathbb F_2}(r)$, so for large enough $r$, $\omega_{\mathbb F_2}(r)>0$.

We also deduce from~(\ref{superWoess2}) that $\omega_{\mathbb F_2}(r)=\omega_{\mu_0}(\zeta_0(r))$, where $\omega_{\mu_0}$ is the growth rate of the Green function associated with $\mu_0$ on $\mathbb F_2$.
Since $\mathbb F_2$ is hyperbolic, by \cite[Theorem~3.1]{SWX}, the Green function has purely exponential growth, i.e.
$$\sum_{x\in \mathbb F_2\cap S_n}G_r^{\mu_0}(e,x)\asymp \mathrm{e}^{n\omega_{\mu_0}(r)}.$$
Consequently, the Poincaré series
$$\Theta_{r,s}(\mathbb F_2)=\sum_{x\in \mathbb F_2}G_r(e,x)\mathrm{e}^{-sd(e,x)}$$
diverges at $s=\omega_{\mathbb F_2}(r)$.
By Corollary~\ref{parabolicdivergentgrowthgap}, $\omega_{\mathbb F_2}(r)<\omega_\Gamma(r)$.

\subsubsection{Example~C}
Finally, here is a last example.
We assume that there exists a finitely generated group $\Gamma_0$ endowed with an admissible finitely supported probability measure $\mu_0$ such that for some $r_0< R_{\mu_0}$, the Poincaré series
$$\Theta_{r_0,s}(\Gamma_0)=\sum_{x\in \Gamma_0}G_{r_0}^{\mu_0}(e,x)\mathrm{e}^{-sd(e,x)}$$
converges at $s=\omega_{\mu_0}(r_0)>0$.

We consider the free product $\Gamma=\Gamma_0*\Gamma_1$ where $\Gamma_1=\mathbb{Z}^d$, $d\geq 3$.
As above, $\Gamma$ is hyperbolic relative to $\Gamma_0$ and $\Gamma_1$.
We consider the adapted measure $\mu_\alpha=\alpha \mu_1+(1-\alpha)\mu_0$.
We will prove that for some $r$, $\omega_{\Gamma_0}(r)=\omega_\Gamma(r)$.
By Lemma~\ref{lemalphato0}, $\zeta_0(R_\alpha)$ converges to $R_{\mu_0}$ as $\alpha$ converges to 0.
Thus, for small enough $\alpha$, there exists $r_\alpha$ such that $\zeta_0(r_\alpha)=r_0$.
Also, by Lemma~\ref{lemalphato1}, $\zeta_1(r_\alpha)$ converges to 0 and $w'_{r,\alpha}$ stays bounded away from 1 as $\alpha$ tends to 0.
By~(\ref{superWoess2'}), for every $\epsilon>0$, there exists $\alpha$ such that for every $x\neq e\in \Gamma_1$,
\begin{equation}\label{smallGs}
G_r(e,x)\leq \epsilon.
\end{equation}

Every element $x\in \Gamma$ can be written as $x=a_1b_1...a_kb_k$, where $a_i\in \Gamma_0$, $b_i\in \Gamma_1$ and $a_i\neq e$ except maybe $a_1$ and $b_i\neq e$ except maybe $b_k$.
Moreover, since the random walk is adapted to the free product structure, it has to pass through $a_1b_1...a_k$ before reaching $a_1b_1...a_kb_k$.
Consequently,
$$\frac{G_r(e,x)}{G_r(e,e)}=\frac{G_r(e,a_1b_1...a_k)}{G_r(e,e)}\frac{G_r(e,b_k)}{G_r(e,e)},$$
see also \cite[(3.3)]{WoessMartin}.
Note that this is an exact version of the relative Ancona inequalities in the specific case of adapted random walks on free products.
We thus get

\begin{equation}\label{boundPoincare}
\begin{split}
&\Theta_{r,s}(\Gamma)=\sum_{x\in \Gamma}G_r(e,x)\mathrm{e}^{-sd(e,x)}\\
&\leq G_r(e,e)\sum_{k\geq 0}\left (\sum_{a\in \Gamma_0\setminus\{e\}}\frac{G_r(e,a)}{G_r(e,e)}\mathrm{e}^{-sd(e,a)}\right )^k\left (\sum_{b\in \Gamma_1\setminus\{e\}}\frac{G_r(e,b)}{G_r(e,e)}\mathrm{e}^{-sd(e,b)}\right )^k.
\end{split}
\end{equation}

By~(\ref{superWoess2}), $\omega_{\Gamma_0}(r_\alpha)=\omega_{\mu_0}(r_0)$, so $\omega_\Gamma(r_\alpha)\geq \omega_{\mu_0}(r_0)>0$.
Therefore,
\begin{align*}
\sum_{x\in \Gamma_0}\frac{G_{r_\alpha}(e,x)}{G_{r_\alpha}(e,e)}\mathrm{e}^{-\omega_\Gamma(r_\alpha)d(e,x)}\leq &\sum_{x\in \Gamma_0}\frac{G_{r_0}^{\mu_0}(e,x)}{G_{r_0}^{\mu_0}(e,e)}\mathrm{e}^{-\omega_{\mu_0}(r_0)d(e,x)}\\
&=\frac{1}{G_{r_0}^{\mu_0}(e,e)}\Theta_{r_0,\omega_{\mu_0}(r_0)}(\Gamma_0)<+\infty.
\end{align*}
If $x\in \Gamma_1$, then as explained above, we have
$$\frac{G_r(e,x)}{G_r(e,e)}=F_r(e,x)=F_{\zeta_1(r)}^1(e,x)$$ where $F$ and $F^1$ are defined by~(\ref{defF}) and~(\ref{defFi}).
Since the random walk on $\Gamma_1=\mathbb Z^d$ is transient at the spectral radius, there exists $C_1$ such that
$$F_{\zeta_1(r)}^1(e,x)\leq F_{R_{\mu_1}}^1(e,x)\leq C_1.$$
Also, since $\Gamma_1=\mathbb{Z}^d$ has polynomial growth, there exists a constant $C_2$ such that
$$\sum_{x\in \Gamma_1}\frac{G_r(e,x)}{G_r(e,e)}\mathrm{e}^{-\omega_{\Gamma}(r_\alpha)d(e,x)}\leq C_1 \sum_{x\in \Gamma_1}\mathrm{e}^{-\omega_{\mu_0}(r_0)d(e,x)}\leq C_2.$$
We now choose $\epsilon>0$ such that
$$\epsilon C_2 \frac{\Theta_{r_0,\omega_{\mu_0}(r_0)}(\Gamma_0)}{G_{r_0}^{\mu_0}(e,e)}=C_3<1$$
and we fix $\alpha$ such that
by~(\ref{smallGs}),
$$\sum_{b\in \Gamma_1\setminus\{e\}}\frac{G_r(e,b)}{G_r(e,e)}\mathrm{e}^{-sd(e,b)}\leq \epsilon \sum_{x\in \Gamma_1}\mathrm{e}^{-\omega_{\Gamma}(r_\alpha)d(e,x)}\leq \epsilon C_2.$$
Consequently, by~(\ref{boundPoincare}),
$$\Theta_{r_\alpha,\omega_\Gamma(r_\alpha)}(\Gamma)\leq \sum_{k\geq 0}C_3^k.$$
This proves that the Poincaré series $\Theta_{r_\alpha,\omega_\Gamma(r_\alpha)}(\Gamma)$ is convergent.
According to Lemma~\ref{LowerBDLem}, we deduce that $\omega_{P}(r_\alpha)=\omega_\Gamma(r_\alpha)$ for some parabolic group $P$.
Since $\Gamma_1=\mathbb Z^d$ has polynomial growth, by Proposition~\ref{subexpgrowthparabolicgap}, we necessarily have $P=\Gamma_0$.

\begin{rem}
We assumed that the Poincaré series of $\Gamma_0$ was convergent at its critical exponent for some $r_0<R_{\mu_0}$.
However, this was only for convenience.
If this Poincaré series is convergent for $r=R_{\mu_0}$, then we can make the same construction and choose instead $r_\alpha=R_\alpha$.
We only need to know that $\zeta_0(R_\alpha)=R_{\mu_0}$ for small enough $\alpha$.
The fact that $\zeta_0(R_\alpha)=R_{\mu_0}$ is equivalent to the fact that $\mu_\alpha$ is spectrally degenerate along $\Gamma_0$, so we just need to ensure that for small enough $\alpha$, $\mu_\alpha$ is spectrally degenerate along $\Gamma_0$.
Now, if the Poincaré series is convergent for $r=R_{\mu_0}$, then $\sum_{x\in S_n\cap \Gamma_0}G^{\mu_0}_{R_{\mu_0}}(e,x)\leq C\mathrm{e}^{-\omega_{\mu_0}(R_{\mu_0})}$ and so
$$\sum_{x\in \Gamma_0}G^{\mu_0}_{R_{\mu_0}}(e,x)G^{\mu_0}_{R_{\mu_0}}(x,e)\leq C\Theta_{r_0,\omega_{\mu_0}}(r_0)<\infty.$$
By \cite[Proposition~1.9]{GouezelLalley}, this implies that the first derivative of the Green function $t\mapsto G^{\mu_0}_t(e,e)$ is finite at $R_{\mu_0}$.
Using the work of \cite[Section~7]{CandelleroGilch}, we can then construct such a spectrally degenerate probability $\mu_\alpha$ along $\Gamma_0$ for small enough $\alpha$.
\end{rem}

\begin{rem}\label{CGMGap}
In \cite[Lemma~4.7]{CandelleroGilchMuller}, the authors prove that $\omega_P(r)<\omega_\Gamma(r)$ always holds.
However, their proof is incorrect.
Indeed, they use the following inequality
$$\sum_{x\in S_{n+m}\cap P}G_r(e,x)\leq C\sum_{x\in S_n\cap P}G_r(e,x)\sum_{x\in S_m\cap P}G_r(e,x).$$
Proving such an inequality would require that for any $x\in S_{n+m}$ and any $y\in S_n$ on a geodesic from $e$ to $x$, we have
$$G_r(e,x)\leq C G_r(e,y)G_r(y,x).$$
This in turn would require Ancona inequalities for the group $P$.
It is an open question whether the fact that Ancona inequalities hold for any geodesic implies that the group is hyperbolic, but it is easy to prove that they do not hold in a parabolic subgroup $P$ if $P$ is virtually abelian, see for instance \cite[Section~5.2]{DG20}.
\end{rem}

Let us give some final remarks to conclude this discussion.
The last example raises the following question.
\begin{quest}\label{question}
Does there exist a finitely generated group $\Gamma_0$ endowed with an admissible finitely supported probability measure $\mu_0$ such that for some $r\leq R_{\mu_0}$, the Poincaré series $\Theta_{r,s}(\Gamma_0)$ is convergent at the radius of convergence $s=\omega_{\Gamma_0}(r)>0$ ?
\end{quest}

If there is a positive answer to this question, then as we saw, there exists a relatively hyperbolic group $\Gamma$ endowed with an admissible finitely supported probability measure $\mu$ and a parabolic subgroup $P$ for which $\omega_P(r)=\omega_\Gamma(r)$ at some $r$.
Moreover, if the measure $\mu_0$ is symmetric, then we can choose the measure $\mu$ to be symmetric as well.

On the contrary, if this question has a negative answer, then Corollary~\ref{parabolicdivergentgrowthgap} suggests that we always have $\omega_P(r)<\omega_\Gamma(r)$.
However, this corollary requires the Poincar\'e series
$$\Theta_{r,s}(P)=\sum_{p\in P} G_r(e,p) e^{-s d(e,p)} $$
to be divergent, where $G_s(e,p)$ is the Green function associated with the measure $\mu$ on $\Gamma$.
Using \cite[Lemma~4.4]{DG21} as above, we can rewrite this Poincar\'e series as
$$\Theta_{r,s}(P)=\sum_{p\in P} G_1^{r,P}(e,p) e^{-s d(e,p)},$$
where $G_1^{r,P}$ is the Green function at 1 associated with the first return kernel $p^{r,P}$ defined in~(\ref{deffirstreturnkernel}).
Unfortunately, this first return kernel is not in general finitely supported, so even if Question~\ref{question} has a negative answer, we cannot deduce that this Poincar\'e series is divergent, hence we cannot deduce that $\omega_P(r)<\omega_\Gamma(r)$.

\section{The growth rate of the trace of the branching random walk}\label{SectiongrowthrateBRW}
\label{s:vgrbrw}

Recall that $\mathcal{P}_n$ is the set of points in $S_n$ that are eventually visited by some particle of the branching random walk and that $M_n=\sharp \mathcal{P}_n$.
In this section, we compare the growth rate of the Green function $\omega_\Gamma(r)$ with the growth rate of the branching random walk $\log \limsup M_n^{1/n}$.
Our goal is to prove the following proposition.
\begin{prop}\label{P:vgrBRW}
For $r\in [1, \rho^{-1}]$, $\omega_{\Gamma}(r) = \limsup_{n \to \infty} \frac{1}{n} \log M_n$ almost surely.
\end{prop}

Theorem~\ref{Thmvolumegrowthconst} is then a consequence of Proposition~\ref{P:vgrBRW}, Corollary~\ref{corolowersemicontinuous} and Corollary~\ref{coroboundv/2}.

\subsection{Upper bound}
We first prove the following.

\begin{prop}\label{upperboundvgrBRW}
Almost surely, we have
$$\limsup_{n \to \infty} \frac{1}{n} \log M_n \leq \omega_\Gamma(r).$$
\end{prop}

The proof of \cite{SWX}, which relies on the Borel-Cantelli lemma uses the purely exponential growth of  the Green functions over spheres.
However, only a small adaptation is needed to apply it here.
We rewrite it for convenience.

\begin{proof}
For $x\in \Gamma$, we denote by $Z_x$ the number of particles of the branching random walk that ever visit $x$.
The many-to-one formula states that
$$\mathbf E[Z_x]=G_r(e,x).$$
Then,
$M_n=\sum_{x\in S_n}\mathbb{1}_{Z_x\geq 1}$, so by the Markov inequality,
$$\mathbf E[M_n]=\sum_{x\in S_n}\mathbf P (Z_x\geq 1)\leq \sum_{x\in S_n}\mathbf E[Z_x]=H_r(n).$$

Let $\epsilon>0$.
By the Markov inequality,
$$\mathbf{P}\left (M_n^{1/n}\geq \mathrm{e}^{\omega_\Gamma(r)}+\epsilon\right)\leq \frac{\mathbf{E}[M_n]}{(\mathrm{e}^{\omega_\Gamma(r)}+\epsilon)^n}\leq \frac{H_r(n)}{(\mathrm{e}^{\omega_\Gamma(r)}+\epsilon)^n}.$$
By definition,
$\omega_\Gamma(r)=\log \limsup H_r(n)^{1/n}$,
so there are at most finitely many $n$ such that
$$H_r(n)^{1/n}\geq \mathrm{e}^{\omega_\Gamma(r)}+\epsilon/2.$$
Therefore,
$$\sum_n \frac{H_r(n)}{(\mathrm{e}^{\omega_\Gamma(r)}+\epsilon)^n}<\infty.$$
The statement of the lemma is thus a consequence of the Borel-Cantelli lemma.
\end{proof}

\subsection{Lower bound}
Before proving the lower bound, we first recall some geometric lemmas about relatively hyperbolic groups.
Let $\mathbb P=\{gP: g\in \Gamma, P\in\mathbb P_0\}$ be the collection of all parabolic cosets.
 Recall that $\eta$ and $L(\eta)$ are fixed so that for $L\geq L(\eta)$, any $(\eta,L)$-transition point satisfies the results of Section~\ref{sectiontransitionpoints}.

\begin{defn}\label{deftransitional}
For $L\geq L(\eta)$, an $(\eta,L)$-transition point on a geodesic is called an $L$-transition point.
A geodesic $\alpha$ is called \textit{$L$-transitional} if every point on $\alpha$ is an $L$-transition point.
\end{defn}

\begin{lem}\label{RelThinTriangleLem2}
Let $L\geq L(\eta)$.
There exists $K$ such that the following holds.
For every $x,y,z\in \Gamma$,
if both $[x,z]$ and $[y,z]$ are $L$-transitional, then there exists a point $w$ within $K$ of an $L$-transition point on $[x,z]$, an $L$-transition point on $[y,z]$ and an $L$-transition point on $[x,y]$.
\end{lem}

\begin{proof}
Let $x,y,z$ satisfy the statement of the lemma.
Applying Lemma~\ref{ThinTransitional2}, consider the last point $w_0$ on $[z,x]$ which is within $C$ of $[z,y]$ and let $w$ be the next point on $[z,x]$.
Since $w$ also is $L$-transitional, by definition of $w_0$, $w$ is within $C$ of an $L$-transition point on $[x,y]$.
Moreover, $w$ is within $C+1$ of an $L$-transitional point on $[y,z]$.
\end{proof}

\begin{lem}\label{existenceofcenter}
There exists a constant $C$ such that the following holds.
Let $x_0,x_1,x_2$ be three points in $\Gamma$.
Then, there exist $w_0,w_1,w_2$ such that for $i \mod 3$,
$$d(x_i,x_{i+1}) \geq d(x_i,w_i)+d(w_i,w_{i+1})+d(w_{i+1},x_{i+1})-C.$$
Moreover, $w_i$ and $w_{i+1}$ are within $C$ of an $L$-transition point on $[x_i,x_{i+1}]$.
Finally, if $[x_i,x_{i+1}]$ is $L$-transitional, then $d(w_i,w_{i+1})\leq C'$, where $C'$ only depends on $L$. 
\end{lem}

It will be convenient to rely on similar results proved in \cite{DussauleLLT1}.
However, the terminology is a bit different and \cite{DussauleLLT1} uses the notion of relative geodesics.
Let us briefly introduce this notion.
Let $\mathbb P_0$ be the chosen set of representatives of conjugacy classes of parabolic subgroups and let $S$ be a finite generating set of $\Gamma$.
Following Osin \cite{Osin}, the relative graph is the Cayley graph of $\Gamma$ endowed with the (possibly infinite) generating set $S\bigcup \bigcup_{P\in \mathbb P_0}P$.
It is quasi-isometric to the coned-off graph introduced by Farb \cite{Farb} who gave one of the first definitions of relatively hyperbolic groups.
The relative graph is hyperbolic.
A relative geodesic is a geodesic in the relative graph.
By \cite[Proposition~8.13]{Hruska}, if $x,y\in \Gamma$, then any point on a relative geodesic from $x$ to $y$ is within a uniformly bounded distance of a transition point on a geodesic from $x$ to $y$ (in the Cayley graph of $\Gamma$).

\begin{proof}[Proof of Lemma~\ref{existenceofcenter}]
Consider the projection of $x_0$ on a relative geodesic from $x_1$ to $x_2$ in the relative Cayley graph.
If several projections exist, choose the closest possible to $x_1$.
Denote this projection by $w_1$ and
let $w_2$ be the point on this relative geodesic just after $w_1$.

By \cite[Lemma~4.16]{DussauleLLT1}, any relative geodesic from $x_0$ to $x_1$ passes at a point $v$ within a bounded distance of $w_1$.
We prove by contradiction that $v$ is within a bounded distance of the projection of $x_2$ on a relative geodesic from $x_0$ to $x_1$ the closest to $x_1$.
Denote by $v'$ such a projection.
Then applying again \cite[Lemma~4.16]{DussauleLLT1}, the relative geodesic from $x_2$ to $x_1$ we chose would pass at a point $w_1'$ within a bounded distance of $v'$.
If $d(v,v')$ is large, then $d(w_1,w_1')$ is also large.
Now, if $w_1$ is before $w_1'$ on the relative geodesic, this contradicts the definition of $v'$ and if $w_1'$ is before $w_1$, this contradicts the definition of $w_1$.

Finally, denote by $w_0$ the point just before $v$ on the relative geodesic from $x_0$ to $x_1$.
Then, applying \cite[Lemma~4.16]{DussauleLLT1} one last time, a relative geodesic from $x_i$ to $x_{i+1}$ passes within a bounded distance of $w_i$ and $w_{i+1}$.
Since points on a relative geodesics are within a bounded distance of transition points (see \cite[Proposition~8.13]{Hruska}), this proves the two first properties of the points $w_i$.

Notice that the points $w_i,w_{i+1}$ are chosen within a bounded distance of successive points on a relative geodesic from $x_i$ to $x_{i+1}$.
If the geodesic $[x_i,x_{i+1}]$ is $L$-transitional, then the corresponding relative geodesic has bounded jumps in parabolic subgroups, hence the distance between $w_i$ and $w_{i+1}$ is bounded.
This concludes the last part of the lemma.
\end{proof}

We define
\[
  S_{n, L} = \left\{ x \in S_n \colon [e, x] \text{ is } L \text{-transitional} \right\}.  
\]
Set $M_{n, L} = \sharp \mathcal{P}_n \cap S_{n, L}$. 
We first consider the lower bound for the quantity $\lim_{n \to \infty} \frac{1}{n} \log M_{n, L}$. 
To this end, mimicking the strategy of \cite{SWX}, we need first and second moments estimates for $M_{n, L}$.  Set
\[
  H_{r, L}(n) = \sum_{x \in S_{n, L}} G_r(e, x), 
\]
and 
\[
  \omega_{\Gamma, L}(r) = \limsup_{n \to \infty} \frac{1}{n} \log H_{r, L}(n). 
\]

\begin{prop}\label{puregrowthtransitional}
  \label{GreenL}
  For every $L \geq L(\eta)$ and $r \in [1, \rho^{-1}]$, there is a constant $c_L > 0$ such that
  \[
    c_L^{-1} \mathrm{e}^{\omega_{\Gamma, L}(r)} \leq H_{r, L}(n) \leq c_L \mathrm{e}^{\omega_{\Gamma, L}(r)}, \quad n \geq 0. 
  \]
  
\end{prop}

\begin{proof}
  For $x \in S_{n + m, L}$, the point $y = [e, x] \cap S_n$ is in $S_{n, L}$ and $y^{-1}x \in S_{m, L}$.  By the relatively Ancona inequalities, there is a constant $c_1 > 0$ such that
\[
  G_r(e, x) \leq c_1 G_r(e, y) G_r(e, y^{-1} x).
\]
Thus we have that
\begin{equation}
  \label{e:GreenLU}
  H_{r, L}(n + m) \leq c_1 \sum_{y \in S_{n, L}} G_r(e, y) \sum_{z \in S_{m, L}} G_r(e, z) = c_1 H_{r, L}(n) H_{r, L}(m). 
\end{equation}

Let $F$ be the finite set given by Lemma~\ref{ExtensionLem}.  Set
\[
  l = \max \left\{ d(e, f) \colon f \in F \right\} + 4 \epsilon. 
\]
For $x \in S_{n, L}$ and $y \in S_{m, L}$ there is $f \in F$ such that
\[
  d(x, [e, x f y]) \leq \epsilon, \quad d(x f, [e, x f y]) \leq \epsilon. 
\]
Also, there are positive constants $c_2$ and $c_3$ such that
\[
  G_r(e, x) G_r(e, y) \leq c_2 G_r(e, x) G_r(x, x f) G_r(x f, x f y) \leq c_3 G_r(e, x f y).  
\]
Note that $x f y \in \bigcup_{-l \leq i \leq l} S_{n + i, L}$.  Thus
\begin{equation}
  \label{e:GreenLD}
  H_{r, L}(n) H_{r, L}(m) \leq c_4 H_{r, L}(n+m)
\end{equation}
for some $c_4 > 0$.  This proposition follows by the Fekete Subadditive Lemma and~\eqref{e:GreenLU}, \eqref{e:GreenLD}. 
\end{proof}

\begin{lem}\label{omegaL}
  For $r \in [1, \rho^{-1}]$, $\lim_{L \to \infty} \omega_{\Gamma, L}(r) = \omega_{\Gamma}(r)$. 
\end{lem}

\begin{proof}
  For $n_0 \in \mathbb{N}$ and $x_i \in S_{n_0}$, $1 \leq i \leq m$, we can choose $f_i \in F$ so that the conditions in Lemma~\ref{ExtensionLem} hold with $g$, $h$ replaced by $x_1f_1x_2\cdots x_{i-1} f_{i-1} x_i$, $x_{i+1}$.  In particular, 
  \[
    x = x_1 f_1 x_2 \cdots f_{m-1} x_m \in \bigcup_{k = (n_0 - l) m}^{(n_0 + l) m} S_{k, \, L}
  \]
  for $L$ sufficiently large,  where $l = \max \left\{ d(e, f) \colon f \in F \right\} + 4 \epsilon$.  Since $F$ is a fixed finite set, we have that 
  \[
    G_r(e, x) \geq c_1^m G_r(e, x_1) \cdots G_r(e, x_m) 
  \]
  with $c_1 = \min_{f \in F} \frac{G_r(e, \, f)}{G_r(e, \, e)}$.  By Lemma~\ref{BoundedToOne}, there is $c_2 > 0$ such that each $x$ has at most $c_2^m$ possible representations in the form of $x = x_1 f_1 x_2 \cdots f_{m-1} x_m$.  Therefore 
  \[
    \sum_{k=(n_0 - l) m}^{(n_0 + l) m} H_{r, L}(r) \geq c_3^m \sum_{x_1, \ldots, x_m \in S_{n_0}} G_r(e, x_1) \cdots G_r(e, x_m) = c_3^m \left[ H_r(n_0) \right]^m  
  \]
  with $c_3 = c_1 c_2^{-1}$.    This and Proposition~\ref{GreenL} imply that
  \[
    2 c_L l m \mathrm{e}^{(n_0 + l) m \omega_{\Gamma, L}(r)} \geq c_3^m \left[ H_r(n_0) \right]^m.   
  \]
  Letting first $m \to \infty$ and then $L \to \infty$, we have for every $n_0 \in \mathbb{N}$, 
  \[
    \liminf_{L \to \infty} \omega_{\Gamma, L}(r) \geq \frac{1}{n_0 + l} \left( \log c_3 + \log H_r(n_0) \right), 
  \]
  which completes the proof of this lemma. 
\end{proof}

Now we are ready to estimate the second moment of $M_{n, L}$, which will help us find a lower bound. 
Let $x$ and $y$ be in $S_{n, L}$.  By Lemma~\ref{RelThinTriangleLem2}, there exists $w = w(x, y) \in \Gamma$ such that $w$ is within a bounded distance of transition points of $[e, x]$ and $[e, y]$ respectively.
\begin{lem}\label{Correlationxy}
  For $1 \leq r < \rho^{-1}$ and $x$, $y \in S_{n, L}$, there exists a positive constant $c > 0$ such that
  \[
    \sum_{z \in \Gamma} G_r(e, z) G_r(z, x) G_r(z, y) \leq c G_r(e, w) G_r(w, x) G_r(w, y).  
  \]
  
\end{lem}

\begin{proof}
  Let $\kappa =C+K$ where $C$ and $K$ are given by Lemma~\ref{ThinTransitional2} and Lemma~\ref{RelThinTriangleLem2}.  Define
  \[
    \Omega_1 = \left\{ z \in \Gamma \colon d(w, u) \leq \kappa \text{ for some transition point } u \in [e, z] \right\}, 
  \]
  and
  \[
    \Omega_2 = \left\{ z \in \Gamma \colon d(w, u_i) \leq \kappa \text{ for transition points } u_1 \text{ on } [z, x] \text{ and } u_2 \text{ on } [z, y] \right\}.  
  \]
By Lemma~\ref{ThinTransitional2} and Lemma~\ref{RelThinTriangleLem2}, $\Gamma = \Omega_1 \cup \Omega_2$.

Assume $z \in \Omega_1$. Applying Lemma~\ref{existenceofcenter} to $(x,y,z)$, we get the existence of transition points $v,v'$ on $[x,y]$ such that $v$ is within a bounded distance of a transition point on $[z,x]$ and $v'$ is within a bounded distance of a transition point on $[z,y]$.
Moreover, since $[e,x]$ and $[e,y]$ are $L$-transitional, $d(v,v')$ is bounded.
Combining all this, we see that $v$ is within a bounded distance of transition points on $[x,y]$, $[x,z]$ and $[y,z]$.

Now, applying Lemma~\ref{RelThinTriangleLem2} to $(w,x,y)$, we have that $v$ is within a bounded distance of either a transition point on $[w,x]$ or on $[w,y]$. We assume without loss of generality that the latter holds.

  \begin{center}
\begin{tikzpicture}[scale=2.2]
\draw (-1.8,.5)--(-2,-.4);
\draw (2.2,.5)--(2,-.4);
\draw (-2,-.4)--(2,-.4);
\draw (-.2,-1.3)--(-.2,-.4);
\draw[dotted] (-.2,-.4)--(-.2,0);
\draw (-.2,0)--(.4,-.2);

\draw (-.2,0)--(.2,.4);

\draw[dotted] (.4,-.2)--(.4,-.4);
\draw (.4,-.4)--(.4,-1);
\draw (.4,-.2)--(.2,.4);
\draw (.2,.4)--(.2,1);
\draw[dotted] (.2,1)--(.2,1.4);
\draw (-2.3,1.9)--(-2.5,1);
\draw (.7,1.9)--(.5,1);
\draw (-2.5,1)--(.5,1);
\draw (.2,1.4)--(-1.6,1.7);
\draw (-1.6,1.7)--(-.5,1.3);
\draw (-.5,1.3)--(.2,1.4);
\draw (-1.6,1.7)--(-1.6,2.5);
\draw (-.5,1.3)--(-.5,2.3);
\draw (-.5,-1.3) node{$e$};
\draw (.1,.05) node{$w$};
\draw (.6,-1.1) node{$x$};
\draw (-1.7,2.5) node{$z$};
\draw (.3,1.4) node{$v$};
\draw (-.45,1.2) node{$v'$};
\draw (-.4,2.3) node{$y$};
\end{tikzpicture}
\end{center}

\begin{claim}
For every given $K$, there exists $K'$ such that if $w$ is within $K$ of a transition point on $[v,y]$, then $d(v,w)\leq K'$.
\end{claim}

\begin{proof}[Proof of the claim]
Since $v$ is within a bounded distance of a transition point on $[w,y]$ by assumption, we have
$$d(w,y)\geq d(w,v)+d(v,y)-C.$$
Thus, if $d(w,[v,y])\leq K$, then
$$d(v,y)\geq d(d(w,v)+d(w,y)-2K$$
and so
$$2d(v,w)\leq C+2K,$$
which proves the claim.
\end{proof}

\begin{claim}
For every given $K$, there exists $K'$ such that if $v$ is within $K$ of a transition point on $[w,x]$, then $d(v,w)\leq K'$.
\end{claim}

\begin{proof}[Proof of the claim]
By the previous claim, we can assume that $w$ is far from a transition point on $[v,y]$.
But then, since $w$ is within a bounded distance of a transition point on $[x,y]$, applying Lemma~\ref{RelThinTriangleLem2} to $(v,x,y)$ shows that $w$ is within a bounded distance of a transition point on $[v,x]$.
Therefore,
$$d(v,x)\geq d(v,w)+d(w,x)-C.$$
Thus, if $d(v,[w,x])\leq K$, then
$$d(w,x)\geq d(v,w)+d(v,x)-2K$$
and so
$$2d(v,w)\leq c+2K,$$
which proves the claim.
\end{proof}

These two claims show that either $d(v,w)$ is bounded or
$w$ is far from a transition point on $[v,y]$ and $v$ is far from a transition point on $[w,x]$.
Applying Lemma~\ref{RelThinTriangleLem2} to $(v,x,y)$ and then to $(z,w,x)$, in every case we get that
$w$ is within a bounded distance of a transition point on $[v,x]$ and $v$ is within a bounded distance of a transition point on $[w,z]$.

By the relatively Ancona inequalities, we thus have
  \begin{align*}
    G_r(e, z) &\leq c_1 G_r(e, w) G_r(w, v) G_r(v, z), \\
    G_r(z, x) &\leq c_1 G_r(z, v) G_r(v, w) G_r(w, x), \\
    G_r(z, y) &\leq c_1 G_r(z, v) G_r(v, y). 
  \end{align*}
Consequently,
  \begin{align*}
&\sum_{z \in \Omega_1} G_r(e, z) G_r(z, x) G_r(z, y) \\
    &\hspace{.2cm}\leq
    c_1^3 G_r(e, w) G_r(w, x) \sum_{z \in \Omega_1} G_r(v, y) G_r(w, v)G_r(v,w) G_r(v, z)G_r(z,v)^2. 
  \end{align*}
    
Note that $G_r(w, v) G_r(v, y) \leq c_2 G_r(w, y)$.  Therefore,
  \begin{align*}
    \sum_{z \in \Omega_1} G_r(e, z) G_r(z, x) G_r(z, y)
    \leq&
    c_3 G_r(e, w) G_r(w, x) G_r(w, y) \\
    &\hspace{.6cm}\sum_{v \in [w, y]} G_r(w, v) \sum_{z \in \Gamma} G_r(v,z)G_r(z, v)^2\\
    &
    \leq c_4 G_r(e, w) G_r(w, x) G_r(w, y). 
  \end{align*}
Here we used the facts that for $r < \rho^{-1}$, $G_r(w, v)$ is decaying exponentially in $d(w, v)$, which is a direct consequence of~(\ref{exponentialdifferenceGreen}) and that
$$\sum_{z \in \Gamma} G_r(v, z)G_r(z,v)^2\leq C\sum_{z\in \Gamma}G_r(v,z)G_r(z,v) < \infty$$
by \cite[Proposition~1.9]{GouezelLalley}.

\medskip
Now we consider the case $z \in \Omega_2$.  Using Lemma~\ref{existenceofcenter} again, there exists $v \in \Gamma$ such that
$v$ is within a bounded distance of a transition point on $[e,z]$, a transition point on $[z,w]$ and a transition point on $[e,w]$. 

  \begin{center}
\begin{tikzpicture}[scale=2.2]
\draw (-1.8,.5)--(-2,-.4);
\draw (2.2,.5)--(2,-.4);
\draw (-2,-.4)--(2,-.4);
\draw (-1.2,-1.3)--(-1.2,-.4);
\draw[dotted] (-1.2,-.4)--(-1.2,0);
\draw (-1.2,0)--(.8,-.1);
\draw (-1.2,0)--(-1,.4);
\draw[dotted] (.8,-.1)--(.8,-.4);
\draw (.8,-.4)--(.8,-1);
\draw (.8,-.1)--(-1,.4);
\draw (-1,.4)--(-1,1);
\draw[dotted] (-1,1)--(-1,1.4);
\draw (-2.3,1.9)--(-2.5,1);
\draw (.7,1.9)--(.5,1);
\draw (-2.5,1)--(.5,1);
\draw (-1,1.4)--(-1.2,1.7);
\draw (-1.2,1.7)--(-1.9,1.3);
\draw (-1.9,1.3)--(-1,1.4);
\draw (-1.2,1.7)--(-1.2,2.5);
\draw (-1.9,1.3)--(-1.9,2.3);
\draw (-1.5,-1.3) node{$e$};
\draw (-.9,.2) node{$v$};
\draw (1,-1) node{$z$};
\draw (-1.1,2.5) node{$y$};
\draw (-2.1,2.3) node{$x$};
\draw (-1.3,1.5) node{$w$};
\end{tikzpicture}
\end{center}
  
By the same argument as in the case $z \in \Omega_1$, we have that
  \[
    \sum_{z \in \Omega_2} G_r(e, z) G_r(z, x) G_r(z, y)
    \leq
    c_5 G_r(e, w) G_r(w, x) G_r(w, y)
  \]
This completes the proof of the lemma.  
\end{proof}

This lemma will help us estimate $\mathbf E \left [M_{n,L}^2\right ]$.
We first recall the following result from \cite{SWX} whose proof has nothing to do with hyperbolicity and holds for any finitely generated group $\Gamma$.
As above, for every $x\in \Gamma$, we denote by $Z_x$ the number of particles that ever visit $x$.

\begin{lem}\label{4.3SWX}
Assume that $\nu$ has finite second moment.
Then, there exists $C$ such that for every $x,y\in \Gamma$,
$$\mathbf E[Z_xZ_y]\leq  C \sum_{z\in \Gamma}G_r(e,z)G_r(z,x)G_r(z,y).$$
\end{lem}

We deduce the following result.

\begin{prop}\label{Coro4.4L}
Assume that $\nu$ has finite second moment and that $r<\rho^{-1}$.
Then there exists $C_L$ such that
$$\mathbf{E}[M_{n,L}]\geq C_L \mathrm{e}^{n\omega_{\Gamma,L}(r)}.$$
\end{prop}

\begin{proof}
As in the proof of \cite[Lemma~4.4]{SWX}, we deduce from Lemma~\ref{4.3SWX} that
$$\mathbf P(Z_x\geq 1)\geq c G_r(e,x).$$
Since $\mathbf{E} [M_{n,L}]=\sum_{x\in S_{n,L}}\mathbf P(Z_x\geq 1)$, the result follows from Proposition~\ref{puregrowthtransitional}.
\end{proof}

\begin{cor}\label{GreenLmoment}
Assume that $\nu$ has finite second moment.  For $r \in (1, \rho^{-1})$ and $L$ sufficiently large, there exists a positive constant $c$ such that
$$\mathbf{E} \left[ M_{n, L}^2 \right] \leq c \left( \mathbf{E}[M_{n, L}] \right)^2.$$ 
\end{cor}

\begin{proof}
By Lemma~\ref{omegaL} and the fact that $\omega_{\Gamma}(r) > 1$, we have $\omega_{\Gamma, L}(r) > 1$ for sufficiently large $L$.  Applying Proposition~\ref{GreenL} and Lemma~\ref{Correlationxy},
  \begin{align*}
    \mathbf{E} \left[ M_{n, L}^2 \right]
    \leq&
          c_1 \sum_{x, y \in S_{n, L}} \sum_{z \in \Gamma} G_r(e, z) G_r(z, x) G_r(z, y) \\
    \leq&
          c_2 \sum_{k=0}^n \sum_{w \in S_{k, L}} \sum_{x, y \in S_{n, L}} G_r(e, w) G_r(w, x) G_r(w, y) \\
    \leq&
          c_3 \sum_{k=0}^n \mathrm{e}^{(2n - k) \omega_{\Gamma, L}(r)} \\
    \leq&
          c_4 \left( \mathbf{E}[M_{n, L}] \right)^2.   
  \end{align*}
This yields the desired bound.
\end{proof}

We can now prove the lower bound and finish the proof of Proposition~\ref{P:vgrBRW}.
\begin{proof}[Proof of Proposition~\ref{P:vgrBRW}]
We first fix $r<\rho^{-1}$ and assume that $\nu$ has finite second moment.
By Proposition~\ref{Coro4.4L},
$$\mathbf{P}\left (M_{n,L}\geq \frac{c_1}{2}\mathrm{e}^{\omega_{\Gamma,L}(r)n}\right)\geq \mathbf{P}\left (M_{n,L}\geq \frac{1}{2}\mathbf{E}[M_{n,L}]\right)$$
and so, by the Paley-Zygmund inequality and Corollary~\ref{GreenLmoment}, for some $c_2>0$,
$$\mathbf{P}\left (M_{n,L}\geq \frac{c_1}{2}\mathrm{e}^{\omega_{\Gamma,L}(r)n}\right)\geq \frac{\big(\mathbf{E}[M_{n,L}]\big)^2}{\mathbf{E}\left[M_n^2\right]}\geq c_2.$$
Thus, with positive probability,
the events $\left\{M_{n,L}^{1/n}\geq \left(\frac{c_1}{2}\right)^{1/n}\mathrm{e}^{\omega_{\Gamma,L}(r)}\right\}$
occur for infinitely many $n$ and so, with positive probability, $\limsup \frac{1}{n}\log M_{n,L}\geq \omega_{\Gamma,L}(r)$.
By definition, $M_n\geq M_{n,L}$,
hence for every large enough $L$, with positive probability (a priori depending on $L$), we have
$\limsup \frac{1}{n}\log M_n\geq \omega_{\Gamma,L}(r)$.
By \cite[Lemma~4.7]{SWX}, $\limsup M_n^{1/n}$ is almost surely a constant.
Thus, for every $L$, almost surely we have
$\limsup \frac{1}{n}\log M_n\geq \omega_{\Gamma,L}(r)$.
Letting $L$ tend to infinity along a sequence, it follows from Lemma~\ref{omegaL} that
$\limsup \frac{1}{n}\log M_n\geq \omega_{\Gamma}(r)$.
Thus by Proposition~\ref{upperboundvgrBRW},
$$\limsup \frac{1}{n} \log M_n=\omega_\Gamma(r).$$

We conclude as in \cite{SWX}.
For every $\epsilon>0$, we can construct a probability measure $\nu'$ with mean $r-\epsilon$ and with finite second moment so that $\nu$ stochastically dominates $\nu'$.
Denoting by $M_n'$ the number of vertices in $S_n$ ever visited by a branching random walk driven by $\mu$ and $\nu'$, we see that $M_n$ stochastically dominates $M_n'$ and so
$$\mathbf{P}\left(\limsup \frac{1}{n}\log M_n\geq \omega_\Gamma(r-\epsilon)\right)\geq \mathbf{P}\left(\limsup \frac{1}{n}\log M_n'\geq \omega_\Gamma(r-\epsilon)\right)=1.$$
Since $\omega_P(r)$ is continuous by Corollary~\ref{corolowersemicontinuous}, we deduce that
$$\limsup \frac{1}{n} \log M_n\geq \omega_\Gamma(r)$$ almost surely.
\end{proof}

\section{A lower bound for the Hausdorff dimension of the limit set}\label{SectionlowerboundHausdorff}
\subsection{The limit set in the Floyd and Bowditch boundaries}
For $r\leq \rho^{-1}$, we let $\Lambda_\mathcal{F}(r)$ and $\Lambda_\mathcal{B}(r)$ be the limit sets of the branching random walk whose offspring distribution has mean $r$ in the Floyd and Bowditch boundary respectively.
\begin{prop}\label{lowerboundHdim}
Let $r\leq \rho^{-1}$.
Almost surely,
$$\Hdim(\Lambda_\mathcal{F}(r),\delta_e)\geq \frac{-1}{\log \lambda}\omega_{\Gamma}(r)$$
and
$$\Hdim(\Lambda_\mathcal{B}(r),\overline{\delta}_e)\geq \frac{-1}{\log \lambda}\omega_{\Gamma}(r).$$
\end{prop}

By~(\ref{shortcutsmallerFloyd}), it is enough to prove the lower bound for the shortcut distance on the Bowditch boundary.
For simplicity, we write $\Lambda=\Lambda_\mathcal{B}(r)$ and $X=\Gamma\cup \partial_{\mathcal{B}}\Gamma$.

Let $h<\frac{-1}{\log \lambda}\omega_\Gamma(r)$.
We will prove that with positive probability, there exists a positive finite measure $\chi$ on the limit set $\Lambda$ such that
$$\int_{\Lambda}\int_{\Lambda}\overline{\delta}_e(x,y)^{-h}d\chi(x)d\chi(y)<+\infty.$$
Recall that $\mathcal{P}_{n,L}$ is the set of $L$-transitional points in $S_n$ that are ever visited by the branching random walk and $M_{n,L}$ is the cardinality of $\mathcal{P}_{n,L}$.
Using Lemma~\ref{omegaL}, we fix $L$ such that
\begin{equation}\label{choiceofL}
h<\frac{-1}{\log \lambda}\omega_{\Gamma,L}(r).
\end{equation}

Let $A_n$ be the event
$$\left \{M_{n,L}\geq \frac{1}{2}\mathbf E[M_{n,L}]\right \}.$$
By the Paley-Zygmund inequality and Corollary~\ref{GreenLmoment}, there exists $p>0$ such that
$\mathbf P(A_n)\geq p$.
Also, for any $C\geq 0$,
$\mathbf P(M_{n,L}\geq C\mathbf E[M_{n,L}])\leq 1/C$.
Therefore, for large enough $C$, the event $B_n$ defined by
$$\left \{\frac{1}{2}\mathbf E[M_{n,L}]\leq M_{n,L}\leq C\mathbf E[M_{n,L}]\right \}$$
satisfies $\mathbf P(B_n)\geq p/2$.
We define for every $n$ a random measure $\chi_n$ by
\begin{equation}\label{defchin}
  \chi_n = \mathbf 1_{B_n}\frac{1}{\mathbf{E} \left[ M_{n, L} \right]} \sum_{x \in \mathcal{P}_{n, L}}\mathbf D(x)
\end{equation}
where $\mathbf D(x)$ is the Dirac measure at $x$.

Our goal is to apply a compactness theorem to the sequence $\chi_n$ and find a limit random measure $\chi$.
In \cite{Crauel}, the author proves a compactness criterion for random probability measures.
Here the measure $\chi_n$ is not almost surely a probability measure, however some of the results of \cite{Crauel} still hold in our context.
References and proofs are postponed to the Appendix.

Note that
$$\mathbf E[\chi_n(X)]\leq C,$$
so $\chi_n$ is a random finite measure in the sense of Definition~\ref{defrandomfinitemeasure}.
We define the measure $\pi_\Omega(\chi_n)$ on $\Omega$ by setting, for every event $A$,
$$\pi_\Omega(\chi_n)(A)=\mathbf E\left [\mathbf 1_A\chi_n(X)\right]=\mathbf E\left [\mathbf 1_A \mathbf 1_{B_n}\frac{M_{n,L}}{\mathbf E[M_{n,L}]}\right ].$$
Then, 
$$\pi_\Omega(\chi_n)(A)\leq C\mathbf P(A).$$
Since $X$ is compact, any subset of $\mathcal{M}_\Omega(X)$ is tight in the sense of Definition~\ref{deftightrandom}.
Moreover, any point of $\Gamma$ is isolated in $X$.
By Corollary~\ref{coroProkhorovisolated}, the closure of $\left \{\chi_n\right \}$ is compact for the weak topology on random finite measures.
This is the smallest topology such that for every random bounded continuous function, the map $\mu\mapsto \mu(f)$ is continuous, where a random bounded continuous function is a function
$f:(x,\omega)\mapsto f(x,\omega)$ such that for every $\omega$, $f(\cdot,\omega)$ is bounded continuous, for every $x$, $f(x,\cdot)$ is measurable and the map $\omega \mapsto \|f(\cdot,\omega)\|_\infty$ is $\mathbf P$-essentially bounded.

Thus, there exists a sub-net $(\chi_\alpha)_{\alpha\in A}$ such that $\chi_{\alpha}$ converges to some random finite measure $\chi$.
This means that there exists a directed set $A$ and a monotone final function $h:A\to \mathbb N$ such that for every $\alpha\in A$,
$\chi_\alpha=\chi_{h(\alpha)}$ and such that $\chi_\alpha$ eventually lies in every neighborhood of $\chi$.
We refer to \cite[Definition~2.11, Definition~2.15, Theorem~2.31]{AliprantisBorder} for more details on nets and a characterization of compactness in terms of convergent sub-nets.
This implies that for every random bounded continuous function $f$, we have
\begin{equation}\label{equationconvergencechi}
\chi_{\alpha}(f)\underset{\alpha\to \infty}{\longrightarrow}\chi(f).
\end{equation}

Since
$$\frac{p}{4}\leq \mathbf E[\chi_{\alpha}(X)]\leq C,$$
the same holds for $\chi$, applying~(\ref{equationconvergencechi}) to the function $\mathbf 1_\Omega\mathbf 1_X$.
Therefore, with positive probability, $\chi$ is not the null-measure.

Note that $\chi_{\alpha}(\mathcal{P}\cup \Lambda)=\chi_{\alpha}(X)$ and that $\mathcal{P}\cup \Lambda$ is a random closed set in the sense of Definition~\ref{defrandomclosedset}.
Thus, by Proposition~\ref{randomPortmanteau}, which is analogous to the classical Portmanteau theorem, we have
$$\chi(\mathcal{P}\cup \Lambda)\geq \limsup_{\alpha}\chi_{\alpha}(\mathcal{P}\cup \Lambda)=\limsup_\alpha\chi_{\alpha}(X)=\chi(X).$$
Also, the topology on $\Gamma \cup \partial_{\mathcal{B}}\Gamma$ extends the discrete topology of $\Gamma$.
Thus, any compact $K\subset \Gamma$ is closed and open, so the function $\mathbf 1_K$ is continuous.
Applying convergence to this function and noting that for large enough $n$,we have $\chi_{n}(K)=0$, hence for large enough $\alpha$, $\chi_\alpha(K)=0$ we get that
$\chi(K)=0$.
Since $\Gamma$ can be written as a countable union of compact sets, we get that
$$\mathbf E[\chi(\mathcal{P})]\leq \mathbf E[\chi(\Gamma)]=0.$$
Thus, $\chi$ almost surely gives full measure to $\Lambda$.

We can now finish the proof of Proposition~\ref{lowerboundHdim}.
\begin{proof}[Proof of Proposition~\ref{lowerboundHdim}]
First, as in the proof of Proposition~\ref{P:vgrBRW}, we can construct a branching random walk $(\Gamma,\mu,\nu')$ such that $\nu'$ has mean $r-\epsilon$, has finite second momment and is stochastically dominated by $\nu$.
As a consequence, we can assume that $r<\rho^{-1}$ and that $\nu$ has finite second moment.

We slightly modify the distance $\overline{\delta}_e$ and set for $x,y\in \Gamma \cup \partial_{\mathcal{B}}\Gamma$
$$\hat{\delta}_e(x,y)=\begin{cases}
\overline{\delta}_e(x,y) &\text{ if }x\neq y \\
\lambda^{d(e,x)} &\text{ if }x=y
\end{cases}$$
where by definition $\lambda^{\infty}=0$.
Note that for $x,y\in \partial_{\mathcal{B}}\Gamma$, $\hat{\delta}_e(x,y)=\overline{\delta}_e(x,y)$.

\begin{claim}
The function $(x,y)\in (\Gamma\cup \partial_\mathcal{B}\Gamma)\times (\Gamma\cup \partial_\mathcal{B}\Gamma)\mapsto \hat{\delta}_e(x,y)$ is continuous.
\end{claim}
\begin{proof}[Proof of the claim]
Let $x_n,y_n$ converge to $x,y$.
If $x\neq y$, then $x_n\neq y_n$, so $\hat{\delta}_e(x_n,y_n)=\overline{\delta}_e(x_n,y_n)$ for large enough $n$, which converges to $\overline{\delta}_e(x,y)=\hat{\delta}_e(x,y)$.

Now if $x=y$, there are two cases.
First, if $x\in \Gamma$, then $x_n=y_n=x$ and so $\hat{\delta}_e(x_n,y_n)=\hat{\delta}_e(x,y)$ for large enough $n$.

Second, assume that $x=y\in \partial_\mathcal{B}\Gamma$.
We have to prove that $\hat{\delta}_e(x_n,y_n)$ converges to 0 as $n$ tends to infinity.
Up to extracting sub-sequences, we can assume that either $x_n=y_n$ for every $n$ or that $x_n\neq y_n$ for every $n$.
In the former sub-case, $\hat{\delta}_e(x_n,y_n)=\lambda^{d(e,x_n)}$ which tends to 0 since $x_n$ tends to infinity.
In the latter, $\hat{\delta}_e(x_n,y_n)=\overline{\delta}_e(x_n,y_n)$, which concludes the proof.
\end{proof}
  Let
\begin{equation}\label{defWn}
\begin{split}
  W_n &= \int\int \hat{\delta}_e(x, y)^{-h} d \chi_n(x) d \chi_n(y)\\
  &= \mathbf 1_{B_n}\frac{1}{\mathbf{E} \left[ M_{n, L} \right]^2} \sum_{x, y \in \mathcal{P}_{n, L}} \hat{\delta}_e(x, y)^{-h}. 
  \end{split}
\end{equation}

\begin{claim}
  \label{boundWn}
The expectation $\mathbf{E}[W_n]$ is uniformly bounded.
\end{claim}

\begin{proof}[Proof of the claim]
  By Lemma~\ref{RelThinTriangleLem2}, for every $x$, $y \in \mathcal{P}_{n, L}$ there exists a point $w_n$ which is within a bounded distance of transition point on $[e, x]$ and on $[e, y]$.  We denote by $d_n$ the supremum of $d(e,w_n)$ for such a point $w_n$.  Note that $[e, x]$ and $[e, y]$ are $L$-transitional.  A small adaptation of the proof of \cite[Proposition~5.13]{PY} yields that
  \[
    \hat{\delta}_e(x, \, y) \geq c_L \lambda^{d_n}.  
  \]
  There exists $C$ such that $d_n\leq n+C$.  Furthermore, if $d_n = k$, then
  $$2 n - k - C \leq d(x, y) \leq 2n - k + C.$$
  By Lemma~\ref{4.3SWX},
  \begin{align*}
    & \mathbf{E} \left[ \sharp \left\{ (x, y) \in \mathcal{P}_{n, L} \colon d_n = k \right\} \right] \\
    & \hspace{.3cm} \leq
          \mathbf{E} \left[ \sum_{\substack{x, y \in S_{n, L} \\ 2n - k - C \leq d(x, y) \leq 2n - k + C}} \mathbb{1}_{\left\{ x, y \text{ are visited by }\mathrm{BRW}(\Gamma,\nu,\mu) \right\}} \right] \\
    & \hspace{.3cm} \leq 
           C \sum_{z \in \Gamma} \sum_{\substack{x, y \in S_{n, L} \\ 2n - k - C \leq d(x, y) \leq 2n - k + C}} G_r(e, z) G_r(z, x) G_r(z, y). 
  \end{align*}
Recall that $r<\rho^{-1}$.
Applying Proposition~\ref{GreenL} and Lemma~\ref{Correlationxy}, we get
\begin{equation}\label{boundAnWn3}
  \mathbf{E} \left[ \sharp \left\{ (x, y) \in \mathcal{P}_{n, L} \colon d_n = k \right\} \right] 
  \leq
  C_L \mathrm{e}^{\omega_{\Gamma,L}(r)(2n-k)}.
\end{equation}
Combining~\eqref{defWn}, (\ref{boundAnWn3}) and Proposition~\ref{GreenL}, we get
$$\mathbf{E}[W_n]\leq C_L\mathrm{e}^{-2n\omega_{\Gamma,L}(r)}\sum_{k=0}^{n+C}\lambda^{-hk}\mathrm{e}^{\omega_{\Gamma,L}(r)(2n-k)}\leq C'_L\sum_{k=0}^{n+C}(\lambda^{-h}\mathrm{e}^{-\omega_{\Gamma,L}(r)})^k.$$
By our choice of $L$~(\ref{choiceofL}), this last quantity is uniformly bounded.
\end{proof}

For any $\kappa\geq 0$, the function
$\kappa\wedge \hat{\delta}_e(x,y)^{-h}$
is bounded continuous on $X\times X$.
Thus,
$$\mathbf E\left [\int \int \kappa \wedge \hat{\delta}_e(x,y)^{-h} d\chi_{\alpha}(x)d\chi_{\alpha}(y)\right ]\underset{\alpha\to \infty}{\longrightarrow}\mathbf E\left [\int \int \kappa \wedge \hat{\delta}_e(x,y)^{-h} d\chi(x)d\chi(y)\right ].$$
By what precedes, we have that for every $\alpha$,
$$\mathbf E\left [\int \int \hat{\delta}_e(x,y)^{-h} d\chi_{\alpha}(x)d\chi_{\alpha}(y)\right ]\leq C$$ for some uniform $C$.
By the Fatou Lemma applied to the measure $\pi_X(\chi)$ on $X$ defined for every Borelian subset $B$ of $X$ by
$$\pi_X(B)=\mathbf E[\chi(B)],$$ we have
$$\mathbf E\left [\int \int \hat{\delta}_e(x,y)^{-h} d\chi(x)d\chi(y)\right ]\leq \liminf_{\kappa \to \infty}\mathbf E\left [\int \int \kappa \wedge \hat{\delta}_e(x,y)^{-h} d\chi(x)d\chi(y)\right ].$$
For every $\kappa$, for every $\alpha$,
$$\mathbf E\left [\int \int \kappa \wedge \hat{\delta}_e(x,y)^{-h} d\chi_{\alpha}(x)d\chi_{\alpha}(y)\right ]\leq \mathbf E\left [\int \int \hat{\delta}_e(x,y)^{-h} d\chi_{\alpha}(x)d\chi_{\alpha}(y)\right ]\leq C$$
and letting $\alpha$ go to infinity, we get
$$\mathbf E\left [\int \int \kappa \wedge \hat{\delta}_e(x,y)^{-h} d\chi(x)d\chi(y)\right ]\leq C.$$
This bound being uniform in $\kappa$, we finally get that
$$\mathbf E\left [\int \int \hat{\delta}_e(x,y)^{-h} d\chi(x)d\chi(y)\right ]\leq C.$$
Consequently, $\mathbf P$-almost surely,
$$\int \int \hat{\delta}_e(x,y)^{-h} d\chi(x)d\chi(y)<+\infty.$$

Recall that $\hat{\delta}_e=\overline{\delta}_e$ on $\partial_\mathcal{B}\Gamma$.
Thus, with positive probability, there exists a finite measure $\chi$ on $\Lambda$ which is not the null-measure and such that
$$\int \int \overline{\delta}_e(x,y)^{-h} d\chi(x)d\chi(y)<+\infty.$$
By Frostman Lemma for metrizable spaces (see \cite[Theorem~2.6]{Shah}), this shows that with positive probability, $h<\Hdim(\Lambda,\overline{\delta}_e)$.
The same argument as in \cite[Lemma~4.7]{SWX} shows that $ \Hdim(\Lambda,\overline{\delta}_e)$ is almost surely a constant.
This concludes the proof.
\end{proof}

\subsection{The limit set in the ends boundary}\label{sectionends}
We prove here Theorem~\ref{coroaccessiblegroups}.
We briefly introduce infinitely ended groups and refer to \cite[Section~4]{DussauleYang} and references therein for more details on those groups and on the link between random walks and the end boundary.
Let $(V,E)$ be a locally finite graph and let $F$ be a finite set of $V$.
We denote by $\mathcal{C}(F)$ an infinite connected component of the complement of $F$ in $(V,E)$.
An end $\xi$ of $(V,E)$ is a collection of infinite connected components $\mathcal{C}(F)$, where $F$ is finite, such that for any two such $F,F'$, we have that the intersection of $\mathcal{C}(F)$ and $\mathcal{C}(F')$ is infinite.
We will also say for simplicity that $\xi$ lies in the connected component $\mathcal{C}(F)$ if $\mathcal{C}(F)$ is part of the collection defining $\xi$.
We denote by $\partial_{\mathcal{E}}(V,E)$ the set of ends.
We can endow the end compactification $V\cup \partial_{\mathcal{E}}(V,E)$ with a topology which extends the discrete topology on $V$ such that $V\cup \partial_{\mathcal{E}}(V,E)$ is compact and $V$ is dense in its ends compactification.
If $\Gamma$ is a finitely generated group, we define its end boundary as the set of ends of a Cayley graph with respect to a finite generating system.
Its topology does not depend on the choice of the finite generating system.
We denote by $\partial_{\mathcal{E}}\Gamma$ the end boundary of $\Gamma$.

Let $0<\lambda<1$.
We define the visual distance $\tilde{\delta}_e$ of parameter $\lambda$ on $\Gamma \cup \partial_{\mathcal{E}}\Gamma$ by setting
$\tilde{\delta}_e(x,y)=\lambda^n$, where $n$ is the minimal integer such that $x$ and $y$ lie in two distinct connected components of the complement of $B(e,n)$.
It is well known that the end boundary is covered by the Floyd boundary, see for instance \cite[Proposition~11.1]{GGPY} or \cite{Karlssonfreesubgroups}.
Moreover, we can be more precise and by \cite[Lemma~4.3]{DussauleYang}, the identity of $\Gamma$ extends to an 1-Lipschitz continuous and equivariant map $\psi$ from the Floyd compactification to the end compactification, so in particular
$$\delta_e(x,y)\geq \tilde{\delta}_e(\psi(x),\psi(y)).$$

If $\Gamma$ is a group with infinitely many ends, we denote by $\Lambda_{\mathcal{E}}(r)$ the limit set of a branching random walk $(\Gamma,\nu,\mu)$ with $\mathbf E[\nu]=r$.
We prove the following.
\begin{prop}\label{lowerboundends}
Let $r\leq \rho^{-1}$.
Almost surely,
$$\Hdim(\Lambda_\mathcal{E}(r),\tilde{\delta}_e)\geq \frac{-1}{\log \lambda}\omega_{\Gamma}(r).$$
\end{prop}

By a celebrated result of Stallings \cite{Stallings}, a group with infinitely many ends $\Gamma$ splits as an HNN extension $A_{*C}$ or an amalgamated product $A*_{C}B$, where $C$ is a finite group.
The action on the corresponding Bass-Serre tree satisfies the conditions of \cite[Definition~2]{Bowditch} and so $\Gamma$ is relatively hyperbolic.
However, there is no clear relation between the shortcut distance on the Bowditch boundary and the visual distance on the end boundary.
Indeed, the ends boundary is in general larger than the Bowditch boundary and even if they coincide, the shortcut distance is the largest distance on the Bowditch boundary satisfying~(\ref{shortcutsmallerFloyd}), so it is bounded from below by the visual distance.
Thus, we cannot deduce Proposition~\ref{lowerboundends} from Proposition~\ref{lowerboundHdim}.

\begin{proof}[Proof of Proposition~\ref{lowerboundends}]
We follow the same strategy as for the Bowditch boundary.
Since $\Gamma$ has infinitely many ends, it is relatively hyperbolic.
It is either an HNN extension $A*_{C}$ or an amalgamated product $A*_{C}B$, where $C$ is finite.
In the former case, the parabolic subgroups are the conjugates of $A$ and one can choose $\mathbb P_0=  \{A\}$, in the latter case, they are the conjugates of $A$ and $B$ and one can choose $\mathbb P_0=\{A,B\}$.
In both cases, every element of $\Gamma$ can be written uniquely in a normal form, see \cite[(9.2),(9.4)]{Woessends}.
For simplicity, we only give details of the proof when $\Gamma=A*_{C}B$.
The case of an HNN extension is treated similarly.
In this situation, the normal form is described as follows.
We choose a set of representatives $\mathfrak{A}$ of $A/C$ and a set of representatives $\mathfrak{B}$ of $B/C$.
Then, any element $x$ of $\Gamma$ can be uniquely written as
\begin{equation}\label{normalform}
x=a_1b_1...a_nb_nc,
\end{equation}
where $a_i\in \mathfrak{A}$, $b_i\in \mathfrak{B}$, $c\in C$.
Moreover, any path from $e$ to $x$ in the Cayley graph of $\Gamma$ has to pass within a bounded distance of every prefix of $x$ in the normal form~(\ref{normalform}).
This follows from the fact that $A*_CB$ is quasi-isometric to the space $X$ obtained by taking copies of $A$ and $B$ for each coset $gA$ and $hB$, $g,h\in \Gamma$ and connecting
\begin{itemize}
    \item $ga\in gA$ to $ga\in gaB$ by adding $\sharp C$ edges between the coset $gaC$ in $gA$ and the coset $gaC$ in $gaB$.
    \item $gb\in gB$ to $gb\in gbA$ by adding $\sharp C$ edges between the coset $gbC$ in $gB$ and the coset $gbC$ in $gbA$.
\end{itemize}
The construction of the space $X$ is performed in \cite{ScottWall}.
It is similar to the construction of a tree of spaces modeling the free product $A*B$ obtained by adding one single edge between every element $ga\in gA$ and $ga\in gaB$ and one single edge between every element $gb\in gB$ and $gb\in gbA$.
This space is also used in \cite{PapasogluWhyte} to prove that $A*B$ is quasi-isometric to $A*_CB$.

In particular, if a geodesic $[e,x]$ is $L$-transitional in the sense of Definition~\ref{deftransitional}, then $|e,x]$ cannot travel long in parabolic subgroups and thus every word $a_i$ and $b_i$ in the normal form of $x$~(\ref{normalform}) satisfies $|a_i|\leq D_L$ and $|b_i|\leq D_L$, where $D_L$ only depends on $L$.

Let $h<\frac{-1}{\log \lambda}\omega_\Gamma(r)$.
Using Lemma~\ref{omegaL}, we fix $L$ such that $h<\frac{-1}{\log \lambda}\omega_{\Gamma,L}(r)$.
The sequence of random finite measures $\chi_n$ on $\Gamma$ defined by~(\ref{defchin}) converges, up to a sub-net, to a random finite measure $\chi$ on the end boundary that gives full measure to $\Lambda_{\mathcal{E}}(r)$ and which is not the null measure with positive probability.
We slightly modify the distance $\tilde{\delta}_e$ by setting $\hat{\delta}_e(x,y)=\tilde{\delta}_e(x,y)$ if $x\neq y$ and $\hat{\delta}_e(x,y)=\lambda^{|x|}$ if $x=y$ and we define
\begin{align*}
  W_n &= \int\int \hat{\delta}_e(x, y)^{-h} d \chi_n(x) d \chi_n(y)\\
  &= \mathbf 1_{B_n}\frac{1}{\mathbf{E} \left[ M_{n, L} \right]^2} \sum_{x, y \in \mathcal{P}_{n, L}} \hat{\delta}_e(x, y)^{-h}. 
\end{align*}
For $x,y\in \mathcal{P}_{n,L}$, we set $d_n$ to be the maximal length of a common prefix of $x$ and $y$ in their normal form.
\begin{claim}
There exists $c_L$ only depending on $L$ such that $\hat{\delta}_e(x,y)\geq c_L\lambda^{d_n}$.
\end{claim}
\begin{proof}[Proof of the Claim]
Let $w_n$ be a common prefix of $x$ and $y$ of length $d_n$ and write
$$x=w_nx_1...x_mc,\ y=w_ny_1...y_lc'$$
where $x_i,y_i$ are either in $\mathfrak{A}$ or $\mathfrak{B}$ and $c,c'\in C$.
Then, any path from $x$ to $y$ has to pass within a bounded distance of $w_nx_1$.
Since $|x_1|\leq D_L$ and $|w_n|=d_n$, we see that $x$ and $y$ lie in distinct components of $B(e,d_n+D_L+C)$.
Thus, $\hat{\delta}_e(x,y)\geq \lambda^{d_n+D_L+C}$.
\end{proof}
Using this claim, we prove as above that $\mathbf E[W_n]$ is uniformly bounded, which allows us to prove that
$$\int \int \tilde{\delta}_e(x,y)^{-h}d\chi(x)d\chi(y)$$
is almost surely finite.
We then use the Frostman Lemma to conclude.
\end{proof}
\begin{rem}
Without involving \cite{PapasogluWhyte}, an alternative proof uses the bottleneck property  introduced in \cite{DussauleYang}.
Indeed, with $r, F$ be given in \cite[Lemma~5.4]{DussauleYang} replacing Lemma~\ref{ExtensionLem}, any two elements $g, h$ can be concatenated via some $f\in F$ such that any path from $e$ to $gfh$ intersects $B(g,r)$.
Such points are referred to as bottleneck points on $[e, gfh]$.
If a path contains a sequence of bottleneck points with consecutive distance at most $L$, then it is said to have the $L$-bottleneck property.
A triangle with sides having the $L$-bottleneck property satisfies the conclusion of Lemma~\ref{RelThinTriangleLem2} and Lemma~\ref{existenceofcenter}.
We define similarly  $S_{n,L}$ to be the set of points $x \in S_n$ so that $[e, x]$ has the $L$-bottleneck property.
Then  Proposition~\ref{P:vgrBRW} follows verbatim the same argument  with transition point replaced with bottleneck points.
The remaining modification goes as explained above.
\end{rem}

A particular class of groups with infinitely many ends are free products of the form $\Gamma=\Gamma_0*\Gamma_1$ where at least one the free factors $\Gamma_i$ is not $\mathbb Z/2\mathbb Z$.
For such groups, the authors of \cite{CandelleroGilchMuller} prove that the Hausdorff dimension of the limit set in the end boundary endowed with a visual distance is exactly $\frac{-1}{\log \lambda}\omega_{\Gamma}(r)$.
Their proof for the upper bound applies to any group with infinitely many ends.
Indeed, it consists in saying that for any end $\xi\in \Lambda_{\mathcal{E}}(r)$, for every $n$, the branching random walk has to visit some point $x\in \bigcup_{0\leq l\leq l_0}S_{n+l}$ such that $\xi$ and $x$ are in the same infinite connected component of the complement of $B(e,n)$.
The constant $l_0$ only depends on the support of the measure $\mu$.
Thus, $\Lambda_{\mathcal{E}}(r)$ can be covered with $\mathcal{P}_n$ sets of diameter bounded by $C\lambda^n$.

For the lower bound, they first show that
$\Hdim (\Lambda_{\mathcal{E}}(r),\tilde{\delta}_e)$ is bounded from below by some number $z^*$ and then prove that $z^*=\omega_\Gamma(r)$, see \cite[Lemma~4.7]{CandelleroGilchMuller}.
However, in order to prove that $z^*=\omega_\Gamma(r)$, they use the same invalid argument as the one described in Section~\ref{sectionexamples}, namely that the quantity $H_r(n)$ is sub-multiplicative.
Proposition~\ref{lowerboundends} fills their gap and
combining it with the proof for the upper bound described above, we get Theorem~\ref{coroaccessiblegroups}.

\begin{rem}\label{CGMCor}
The proof of the upper bound described above also works for the parabolic cosets.
Namely, for every $P\in \mathbb P$, the Hausdorff dimension of $\Lambda_{\mathcal{E}}(r)\cap \partial_{\mathcal{E}}P$ is bounded from above by $\omega_{P}(r)$.
Thus, assuming further that the Green function has a parabolic gap, we recover \cite[Corollary~3.7]{CandelleroGilchMuller}, i.e.\
$$\Hdim(\Lambda_{\mathcal{E}}(r)\cap \partial_{\mathcal{E}}P)<\Hdim (\Lambda_{\mathcal{E}}(r)).$$
\end{rem}

\section{An upper bound for the Hausdorff dimension of the limit set}\label{SectionupperboundHausdorff}
Let $\Gamma$ be a  relatively hyperbolic group and let $\Lambda_\mathcal{F}(r)$ be the limit set of the branching random walk in the Floyd boundary of $\Gamma$, endowed with the Floyd distance $\delta_e$.
Recall that by Theorem~\ref{mapGerasimov}, there exists a map $\phi$ from the Floyd boundary to the Bowditch boundary such that the preimage of a conical limit point is reduced to a single point.
We can thus see the set of conical limit points $\partial_{\mathcal{B}}^{con}\Gamma$ in the Bowditch boundary as a subset of the Floyd boundary.
We set
$$\Lambda^{con}_{\mathcal{F}}(r)=\Lambda_{\mathcal{F}}(r)\cap \phi^{-1}(\partial_{\mathcal{B}}^{con}\Gamma).$$
We prove here the following proposition. 

\begin{prop}\label{upperboundHausdorff}
Let $r\leq \rho^{-1}$.
Almost surely, 
$$\Hdim(\Lambda^{con}_{\mathcal{F}}(r),\delta_e)\leq \frac{-1}{\log \lambda}\omega_{\Gamma}(r).$$
\end{prop}

Under additional assumption on the volume growth of parabolic subgroup, we have the upper bound on the full limit set in Floyd boundary.
\begin{cor}\label{corowholeFloyd}
Let $r\leq \rho^{-1}$. Assume that $v_S(P)\le \omega_{\Gamma}(r)$ for every parabolic subgroup $P$. Then
almost surely,  
$$\Hdim(\Lambda_{\mathcal{F}}(r),\delta_e)\leq \frac{-1}{\log \lambda}\omega_{\Gamma}(r).$$
\end{cor}
\begin{proof}
The Bowditch boundary consists of conical points and countably many parabolic points and the preimage of each parabolic point is exactly the limit set of a parabolic subgroup (see \cite[Theorem A]{GePoJEMS}).
So the limit set $\Lambda_{\mathcal{F}}(r)$   is contained in the  union of $\Lambda^{con}_{\mathcal{F}}(r)$ with countably many limit sets of parabolic cosets $P\in\mathbb P$.
As $P$ is quasi-convex  in  the Cayley graph of $\Gamma$, any geodesic from $e$ to the limit points of $P$ is contained in a fixed neighborhood of $P$, see for instance \cite[Lemma~4.3]{DrutuSapir}.
Using a suited covering by shadows based at $S_n\cap P$, we can see that $\Hdim(\Lambda_{\mathcal{F}} P)\le \frac{-1}{\log \lambda}v_S(P)$ (e.g. \cite[Lemma 4.1]{PY}).
The conclusion now follows from Proposition \ref{upperboundHausdorff}.   
\end{proof}

To prepare the proof for Proposition \ref{upperboundHausdorff}, we first need a few geometric lemmas.
Denote as usual by $\mathbb P=\{gP: g\in \Gamma, P\in\mathbb P_0\}$ the collection of all parabolic cosets, and by $\ell(\gamma)$ the length of a path $\gamma$.
Fix a geodesic $[x,z]$ and $\epsilon_1,\epsilon_2\in [0,1]$. An \textit{$[\epsilon_1,\epsilon_2]$-percentage} of $[x,z]$ consists of points $w\in [x,z]$ such that $\epsilon_1 \le d(x,w)/d(x,z)\le \epsilon_2$.

\begin{lem}\label{DeepContainedLem}
Let $\gamma$ be a geodesic segment such that $[\epsilon, 1-\epsilon]$-percentage of $\gamma$ contains no $(\eta, L)$-transition point. Then there exists a unique $P\in\mathbb P$  such that the entry and exit points of $\gamma$ in $N_\eta(P)$ have distance at most $\epsilon \ell(\gamma)$ to the corresponding endpoints of $\gamma$.  
\end{lem}
\begin{proof}
By assumption, the middle point  $m\in \gamma$ is $(\eta, L)$-deep in a unique $P\in \mathbb P$. Let $u, v$ be the corresponding entry and exit point of $\gamma$ in $N_\eta(P)$. By Lemma \ref{EntryTransLem}, $u, v$ are $(\eta, L)$-transition points, so $d(x,u)\leq \epsilon d(x,y)$ and $d(v, y)\le \epsilon d(x,y)$, which concludes the proof.   
\end{proof}


Fix  $C>0$ and $x\in S_n$. The \textit{$C$-{\color{blue}}partial cone} $\Omega(x,C)$ consists  of points $z\in G$ such that $[e,z]$ contains an $(\eta, L)$-transition point $C$-close to $x$. 


Let $B([x,z])$ be the ball centered at the middle point of $[x,z]$  of radius ${d(x,z)/2}$. Define $U(x)$ to be the union of the partial cone $\Omega(x, C)$ and the balls $B([x,z])$ for all geodesics $[x,z]$ between $x$ and $z\in \Omega(x,C)$. That is,
$$
U(x):= \Omega(x,C) \cup \left( \bigcup\left\{B([x,z]): \forall [x,z], \forall z\in  \Omega(x,C)\right\} \right). 
$$
Note that any ball of centered at $w\in[x,z]$ of radius $\min\{d(w,x), d(w,z)\}$ is contained in $B([x,z])$.

 In what follows, let $C>0$ be given by Lemma \ref{ThinTransitional2}. 
\begin{lem}\label{AvoidBallLem}
Let $\alpha$ be a path starting from $e$ and first entering  at a point $z\in U(x)$. Let $ w\in [x,z]$ be a transition point.    Set $S:=\min\{d(w,x), d( w, z)\}$. Then $B(w, S-2C)$ is contained in $U(x)$ so $\alpha$ has distance at least $S-2C$ to   the point $w$.     
\end{lem}
\begin{proof}
We first consider the case $z\in \Omega(x,C)$. By definition of $U(x)$,  any ball centered at $w\in [x,z]$ of radius $S=\min\{d(x,w), d(z,w)\}$ is contained in $ U(x)$. The statement follows immediately.  

Assume now that $z$  lies in a ball $B([x,\hat z])$ where $[x,\hat z]$ is a geodesic between $x$ and  some $\hat z$ in $\Omega(x,C)$. Let  $\hat w\in [x,\hat z]$    be the middle point.

Consider the triangle with vertices $x,\hat w,z$. As $w$ is a transition point on $[x,z]$, Lemma \ref{ThinTransitional2}  shows that $d(w, w')\le C$ for some $w'\in [\hat w,x]\cup[\hat w, z]$. Thus, $B(w, \kappa-2C)\subset B(w', \kappa-C)$. As $w'$ is on the radius $[\hat w,x]$ or $[\hat w, z]$ of the ball $B([x,\hat z])$,   we have the ball $B(w', S-C)$ is contained in $B(\hat w, \eta)\subseteq U(x)$. 
\end{proof}

\begin{lem}\label{ConvConicalLem}
Let $\xi\in \partial_{\mathcal B} \Gamma$ be a conical point and consider a sequence of points $z_n\to \xi$.
Let $x\in [e,\xi]$ be an $(\eta, L)$-transition point. Then   for for all  but finitely many $z_n$,   there exists an $(\eta,L)$-transition point $x_n$ on $[e, z_n]$ such that $d(x_n, x)\le C$.  In particular,  $z_n\in \Omega(x, C)$ for large enough $n$.

\end{lem}
\begin{proof}
If $z_n\to \xi$ then $\delta_x(z_n, \xi) < \lambda^C \delta$ for all large enough $n$, where $\delta=\delta(\eta,L)$ is given by Lemma~\ref{TransLargeLem}.
Applying Lemma~\ref{ThinTransitional2} to the triangle with vertices $e$, $\xi$, $z_n$, there is an $(\eta, L)$-transition point $x_n$ on $[z_n, \xi]$ or $[e, z_n]$ such that $d(x, x_n) \leq C$. 
It suffices to prove that $x_n \in [e, z_n]$. 
Let $y$ be any transition point on $[z_n, \xi]$. 
Then by Lemma~\ref{TransLargeLem}, $\delta_y(z_n, \xi) > \delta$.  It follows from~(\ref{changeofbasepointFloyd}) that $\lambda^C \delta > \delta_x(z_n, \xi) \geq \lambda^{d(x, y)} \delta_y(z_n, \xi)>\lambda^{d(x,y)}\delta$ and hence $d(x, y) > C$. The conclusion follows.
\end{proof}

Fix $\epsilon\in (0, 1/2)$.
Let $U_\epsilon(x)$ be the set of points $z\in U(x)$ such that $[x,z]$ contains a transition point $w$ being at distance at least $\epsilon d(x,z)$ to one of the endpoints:  $$ \max\{d( w, x), d(w,z)\}\ge \epsilon d(x,z).$$  

By  Lemma \ref{EntryTransLem}, the set $U(x)\setminus U_\epsilon(x)$
consists of points $z\in U(x)$ such that the $[\epsilon, 1-\epsilon]$-percentage of $[x,z]$ is contained in the $\eta$-neighborhood of a unique peripheral coset $P\in\mathbb P$. Explicitly, there exists a subsegment of $[x,z]$ with length at least $(1-2\epsilon)d(x,z)$ contained in $N_\eta(P)$.    

\begin{lem}\label{AvoidTransBall}
Let $\alpha$ be a path starting from $e$ and first entering $U(x)$ at a point $z\in U_\epsilon(x)$. 
Assume that $\epsilon d(x,z)>10C$. Then there exists a  transition point $y$   on $[e, z]$ such that $\alpha$ lies outside the ball around $y$ with radius $\epsilon d(x,z)-3C$.
\end{lem}
 
\begin{proof}
By definition of $z\in U_\epsilon(x)$, $[x,z]$ contains a transition point $w$ such that
$$\epsilon d(x,z)\le \max\{d(w,x),  d( w, z)\}.$$
Setting $S=\epsilon d(x,z)$,   Lemma \ref{AvoidBallLem} implies that $B(w,\kappa-2C)\subset U(x)$, so $\alpha$ does not intersect $B(w,S-2C)$. To conclude the
proof, it remains to find a $(\eta,L)$-transition point $y\in [e,z]$ such that $d(w,y)\le C$. In particular,  $\alpha$ does not intersect $B(y,S-3C)$, completing the proof.

Indeed, as in the proof of Lemma \ref{AvoidBallLem}, $z$ lies on the ball $B([x,\hat z])$ centered at  the middle point  $\hat w$  of $[x,\hat z]$ for some $\hat z\in \Omega(x,C)$. 
By assumption,   $w$ is  an $(\eta,L)$-transition point on $[x,z]$.   Lemma \ref{ThinTransitional2} applied for   the  triangle with vertices $x,\hat w,z$ shows that $d(w,w')\le C$ for some $w'\in [\hat w,x]\cup[\hat w,z]$. A schematic figure is shown below. 

Similarly  for   the triangle with vertices $e,x,z$,    Lemma \ref{ThinTransitional2}  implies that either $[e,z]$ or $[e,x]$ contains an $(\eta,L)$-transition point $y$ such that $d(w,y)\le C$ and then $d(y,w')\le 2C$.   As $\hat z\in \Omega(x,C)$, $d(x,[e,\hat z])\le C$ holds, so the triangle inequality shows
$$
d(e,x)+d(x,\hat z) -2C \le d(e,\hat z). 
$$
In a different term, this implies that the  path $p:=[e,x][x,\hat z]$ is   \textit{$(2C)$-taut}: $\ell(p)\le d(p_-,p_+)+2C$.
It follows from triangle inequality  that any subpath of a $(2C)$-taut path $p$ is $(2C)$-taut.

We first claim that either $w'\notin [x,\hat w]$ or $y\notin [e,x]$. Otherwise, we have  $w'\in [x,\hat w]$ and $y\in [e,x]$. Since $[y,x][x,w']$ is a $(2C)$-taut subpath of $p$,  we obtain that $d(y,x)+d(x,w')\le d(y,w')+2C\le 4C$, so $d(x,w)\le d(x,w')+d(w',w)\le 5C$. This is a contradiction because  $d(x,w)\ge \epsilon d(x,z)> 5C$.
\begin{figure}[htb] 
 
\includegraphics[width=0.8\linewidth]{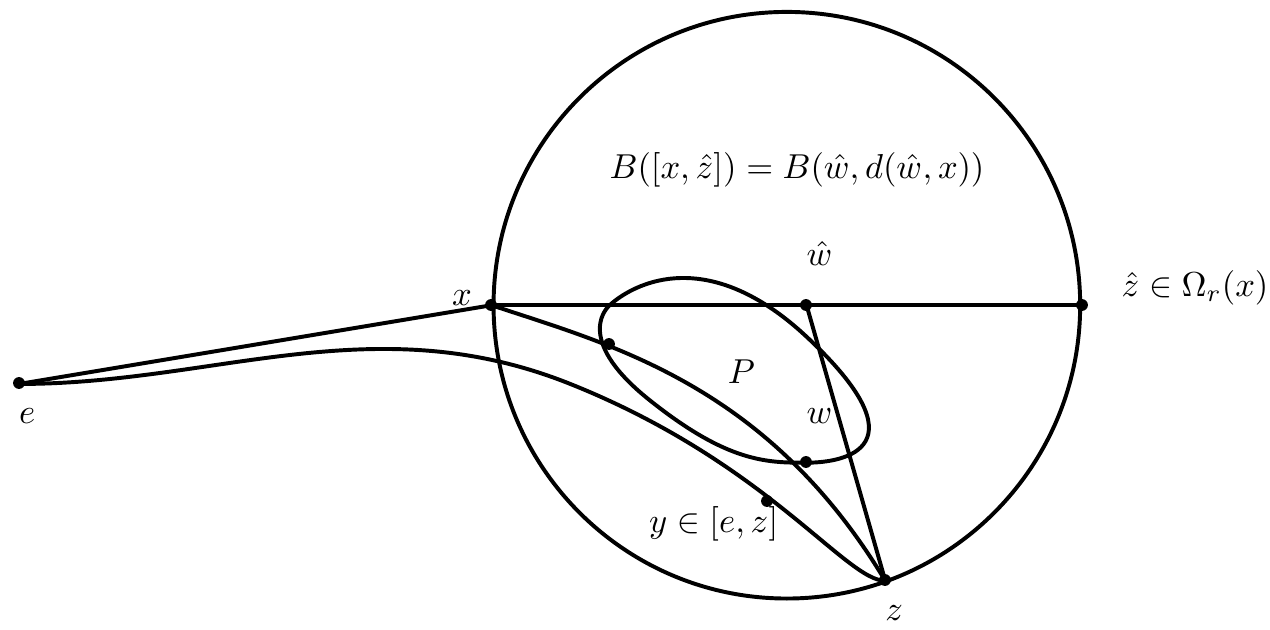} 
\label{figure1}
\end{figure}

Let us assume now  $y\in [e,x]$ to derive a contradiction.
The above claim shows   $w'\notin [x,\hat w]$ and then $w'\in [\hat w,z]$.  Using the $(2C)$-taut path   $[y,x][x,\hat w]$, we have $$d(y, \hat w) +2C \ge d(y,x)+d(x, \hat w).$$ which combined with          $d(x,\hat w)=d( z,\hat w)\ge d(w', \hat w)$ gives $$d(y, \hat w) +2C \ge d(y,x)+d(w', \hat w).$$
On the other hand,   $|d(y,\hat w)-d(w',\hat w)|\le d(y,w') \le 2C$.
These together show that $d(y, x)\le 4C$ and then $d(w',x)\le d(w',y)+d(y,x)\le 6C$.
Thus, $|d(x,\hat w)-d( w',\hat w)|\le 6C$.
Since $d(x,\hat w)=d( z,\hat w)$ and $w'$ is on $[\hat{w},z]$, we have $d(w',z)\le 6C$, so   $d(x,z)\le 12C$. This contradicts the assumption $\epsilon d(x,z)>10C$, hence we proved that $y\in [e,z]$ is impossible, so $y$ is  the desired transition point on $[e,z]$. 
\end{proof}

For any $m\ge 1$, let $U_{\epsilon}(x,m)$ be the set of elements  $z\in U_{\epsilon}(x)$ such that $d(x,z)\ge m$.

\begin{lem}\label{TraceAvoidTransBall}
For any $\epsilon\in (0,1/2)$, there exists $\kappa>0$ with the following property.
Almost  surely,   there exists $n_0>0$ such that  for all $n>n_0$ and all $x\in S_n$: if $\mathrm{BRW}(\Gamma,\nu,\mu)$ first enters $U(x)$ at a point $z$, then   $z\in U(x) \setminus U_{\epsilon}(x,\kappa \log |x|)$.  

\end{lem}
 
\begin{proof}
Let us freeze all particles of the branching random walk 
when they enter $U(x)$ at the first time.  Denote by $\mathcal{Z}(x, m)$ the collection of frozen particles $z \in U_{\varepsilon}(x, m)$.  Then for $z \in \mathcal{Z}(x, m)$ we have 
\begin{enumerate}
    \item 
     $d(x,z)\ge m$,  
    \item 
     $\max\{d(y,x), d(y,z)\}>\epsilon d(x,z)$ where $y$ is an $(\eta,L)$-transition point on $[e,z]$ given by Lemma~\ref{AvoidTransBall}.  
\end{enumerate}    
Note that  the genealogy path from $e$ to $z$ does not intersect $B(y,\epsilon d(x,z))$.  

Let $\delta=\delta(\eta, L)$ be given by Lemma \ref{SupExpDecayLem}. Then     the expected number of  particles frozen
at $z\in U_\epsilon(x)$   is upper bounded by
$$
G_{r}(e,z; [U_\epsilon(x)]^c) \le G_{r}(e,z; [B(y,\epsilon d(x,z))]^c) \le e^{-   e^{\delta \epsilon d(x,z)}}. 
$$  

By Lemma~\ref{purelyexponentialgrowth} there exists $c>0$ such that $\sharp S_n\le c e^{vn}$ for any $n\ge 1$. Thus,  there exist $\epsilon_1=\epsilon_1(\epsilon,\delta,v)$ and $m_0>0$ such that for any $m>m_0$, we have
$$
\mathbf E[\sharp \mathcal Z({x,m})]\le \sum_{z\in U_\epsilon(x,m)} G_{r}\bigl (e,z; [B(y,\epsilon d(x,z)]^c)\bigr) \le \sum_{k=m}^{\infty} ce^{kv} e^{-   e^{\delta\epsilon k}} \le e^{-e^{\epsilon_1 m}}.
$$
 

Choose $\kappa$ so that $ \kappa \epsilon_1 >1$ and let $m=\kappa \log n$.  Consider the event
$$A_n = \left\{ \mathcal{Z}(x, m) \ge 1 \, \text{for some } x \in S_n \right\},$$ i.e. \ $A_n$ is the   event such that if  the branching random walk visits $U_x$ for some $x\in S_n$, then the first frozen particle is in $U_{\epsilon}(x,m)$.  Then  
$$
\mathbf P(A_n) \le \sum_{x\in S_n}\mathbf P(\sharp  \mathcal Z({x,m})\ge 1) \le\sum_{x\in S_n} \mathbf E[\sharp \mathcal Z({x,m})] \le c e^{vn-e^{\epsilon_1 \kappa\log n}} \le c e^{vn-n^{\epsilon_1\kappa}}. 
$$ 
Therefore,  $\sum_{n=1}^{\infty} \mathbf P(A_n)  <\infty$ and so
the conclusion follows from the Borel-Cantelli Lemma.      
\end{proof}

Similarly, we prove the following.
\begin{lem}\label{TraceAvoidTransBall2}
For every $K\geq 0$ and $C\geq 0$, there exists $\kappa>0$ such that the following holds.
Almost surely, for all but finitely many $x$, if the branching random walk ever visits a point $y$ with $d(e,y)\leq Kd(e,x)$ and such that $x$ is within $C$ of a transition point on $[e,y]$, then it first enters $B(x,\kappa \log |x|)$.
\end{lem}

\begin{proof}
Given $x\in S_n$, we consider the set $U(x)$ of $y\in\Gamma$ so that $|y|\le K|x|$ and $x$ is within $C$ of a transition point on $[e,y]$. We freeze particles when they first reach a point $y\in U(x)$, without entering $B(x,\kappa \log |x|)$.
We denote by $\mathcal{Z}_n$ the set of frozen particles for all $x\in S_n$.
By Lemma \ref{SupExpDecayLem}, for $n$ large enough,    the expected number of  particles frozen at $y\in U(x)$   is upper bounded by
$$
G_{r}(e,y; [U(x)]^c) \le G_{r}(e,y; [B(x,\kappa \log |x|)]^c) \le e^{-   n^{\delta \kappa }}. 
$$  

By Lemma~\ref{purelyexponentialgrowth}, there exist $c>0$ such that $\sharp S_n\le c e^{vn}$ and $\sharp U(x)\le c e^{vKn}$ for any $n\ge 1$ and $x\in \Gamma$. Thus,   we have
$$
\mathbf E[\sharp \mathcal Z_n]\le \sum_{x\in S_n}  \sum_{y\in U(x)} G_{r}(e,y; [U(x)]^c) \le  c e^{v(K+1)n} e^{-   n^{\delta\kappa}}.   
$$
If $\kappa$ is chosen large enough, then $\sum_{n=1}^\infty \mathbf E[\sharp \mathcal Z_n]<\infty$, so the proof follows from the  Borel-Cantelli Lemma.
\end{proof}
Recall that if $P\in \mathbb P$ is a parabolic coset, $\eta\geq 0$ and $x\in \Gamma$, we denote by $$\pi_{N_\eta(P)}(x) :=\{y\in N_\eta(P): d(x,y)=d(x, N_\eta(P))\}$$ the set of its shortest projections on the $\eta$-neighborhood $N_\eta(P)$ of $P$. 
Also, for $x,y\in \Gamma$, we denote by $$d_{N_\eta(P)}(x,y):=\diam{\pi_{N_\eta(P)}(x)\cup\pi_{N_\eta(P)}(y)}$$
and \cite[Corollary~8.2]{Hruska} the shortest projection is coarsely Lipschitz,
$$d_{N_\eta(P)}(x,y)\le kd(x,y)+k$$ for a fixed $k\ge 1$ depending only on $\eta$. Thus, $\pi_{N_\eta(P)}(x)$ has  bounded diameter.

\begin{lem}\label{ProjectionLem}\cite[Lemma~1.15]{Sistoprojections}
For every large enough $\eta$, there exists $C=C(\eta)>0$ such that
for every $x\in \Gamma$, $P\in \mathbb P$ and for every geodesic $\gamma$ starting at $x$ and entering $N_\eta(P)$, we have
$$
\pi_{N_\eta(P)}(x) \subset B(y,C) 
$$
where $y$ is the entrance point of $\gamma$ in $N_\eta(P)$.
\end{lem}


\begin{lem}\label{exponentialmoments}
There exists $\eta_0\geq 0$ such that for $\eta\geq \eta_0$, the following holds.
Almost surely, for all but finitely many parabolic cosets $P\in \mathbb P$,  if the branching random walk ever visits a point $z$ satisfying 
both that $d_{N_\eta(P)}(e,z)\geq d(e,N_\eta(P))$ and $d(z,N_\eta(P))\leq  d_{N_\eta(P)}(e,z)$, then it first needs to enter $N_\eta(P)$ at a point $w$ such that $d_{N_\eta(P)}(e,w)\leq d(e,N_\eta(P))$.

In particular, if the branching random walk ever enters $N_\eta(P)$, then the first entrance point must be within $d(e,N_\eta(P))$ of the projection of $e$ on $N_\eta(P)$.
\end{lem}


\begin{proof}
Given $P\in\mathbb P$, we freeze particles when they first visit some point $z\in \Gamma$ with
$$d_{N_\eta(P)}(e,z)\geq d(e,N_\eta(P))$$ and
$$d(z,N_\eta(P))\leq  d_{N_\eta(P)}(e,z),$$ without having entered $\{w\in N_\eta(P): d_{N_\eta(P)}(e,w)\leq d(e,N_\eta(P))\}$. Denote  by $\mathbb P_k$  the collection  of   parabolic cosets $P$  with $d(e,N_\eta(P))=k$.
We denote by $\mathcal{Z}_k$ the set of such frozen particles for those $P\in \mathbb P_k$.
Then,
$$\mathbf E[\sharp \mathcal{Z}_k]\leq \sum_{P\in \mathbb P_k}\sum_{\underset{d_{N_\eta(P)}(x,e)\geq k}{x\in N_\eta(P)} }\sum_{\underset{d(x,y)\leq k}{y \colon \pi_{N_\eta(P)}(y)=x}}G_R(e,y;N_\eta(P)^c).$$
As $x\in \pi_{N_\eta(P)}(z)$,  there exists a trajectory for the $\mu$-random walk from $z$ to $x$ that stays outside $N_\eta(P)$ of linear length in $d(x,z)$, since the support of $\mu$ is finite.
In particular,
$$G_R(z,x;N_\eta(P)^c)\geq \mathrm{e}^{-\alpha d(x,z)}\ge \mathrm{e}^{-\alpha k}$$ for some positive $\alpha$.
Note that by Lemma~\ref{purelyexponentialgrowth}, for a fixed $x$, the set of such elements $z$ being contained in a ball of radius $k$ grows as an exponential function in $k$.  Thus,
\begin{equation}\label{proofexponentialmoments}
    \mathbf E[\sharp \mathcal{Z}_k]\leq \sum_{P\in\mathbb P_k}\sum_{\underset{d(x,\pi_{N_\eta(P)}(e))\geq k}{x\in N_\eta(P)} }G_R(e,x;N_\eta(P)^c)\mathrm{e}^{\beta k},
\end{equation}
where $\beta$ depends both on $\alpha$ and on the growth rate $v$ of the word distance.
By  Lemma~\ref{lemmaexponentialmoments}, for every $M\geq 0$, there exists $\eta_0$ such that for all $\eta\geq \eta_0$,
$$G_R(e,x;N_\eta(P)^c)\leq C \mathrm{e}^{-Md(x,\pi_{\eta,P}(e))}.$$
Choosing $M>v+\beta$, where $v$ is the growth rate of the word distance, we have by Lemma~\ref{purelyexponentialgrowth} and by~(\ref{proofexponentialmoments})
that for $\eta\geq \eta_0$,
$$\mathbf E[\sharp \mathcal{Z}_k]\leq C\mathrm{e}^{-k(M-v-\beta)}.$$
By the choice of $M$, the sum $\sum_{k\geq 0}\mathbf E[\sharp \mathcal{Z}_k]$is finite.
The result again follows from the Borel-Cantelli lemma.
\end{proof}


\begin{lem}\label{LogTrackingLem}
There exists $\kappa$ such that almost surely,   for any conical limit  point $\xi$ in $\Lambda$
and {for all but finitely many transition point $x$ on $[e,\xi]$, $\mathcal{P}$ intersects $B(x,\kappa \log |x|)$.}
\end{lem}
\begin{proof}
Let  $\xi\in \Lambda$  be a conical point. Then $[e,\xi]$ contains infinitely many $(\eta, L)$-transition points (see  \cite[Lemma 2.20]{YangPS}). Consider any  $(\eta, L)$-transition point $x\in [e,\xi]$ so that $|x|>n_0$. According to  Lemma \ref{ConvConicalLem}, for $\Omega(x,C)\subset U(x)$, the branching random walk must enter $  U(x)$.  

Set $\epsilon\in (0,1/2)$ so that $K=\epsilon^{-1}\ge 4$.  Let $\kappa$ be given by Lemmas \ref{TraceAvoidTransBall} and \ref{TraceAvoidTransBall2}.  Up to enlarging $\eta$, we may assume that it is big enough to apply Lemma~\ref{exponentialmoments}.

Let $z\in U_\epsilon(x)\setminus U_{\epsilon}(x,\kappa \log |x|)$ be the first entrance point by Lemma \ref{TraceAvoidTransBall}, so one of  the following statements is true:
\begin{enumerate}
    \item 
      $d(z, x)\le  \kappa \log |x|$,  
    \item
    the $[\epsilon, 1-\epsilon]$-percentage  of $[x,z]$ does not contain any $(\eta,L)$-transition point.  
\end{enumerate}

If the case (1) happens, then  we are done. We now assume     $d(x,z)>\kappa \log |x|$. 

By Lemma \ref{DeepContainedLem}, there exist a unique   coset $P\in \mathbb P$ such that if $y_1,y_2$ are the entrance and exit points of  $[x,z]$ in $N_\eta(P)$, then
$$\max\{d(x,y_1), d(y_2,z)\}\le \epsilon d(x,z),$$
so
$$d(y_1,y_2)\geq (1-2\epsilon)d(x,z)\ge (1-2\epsilon) \kappa \log n_0.$$
By definition,   $z$ is contained in a ball $B([x,\hat z])$ centered at $\hat w$ for some $\hat z\in \Omega(x,C)$.

By definition of $\Omega(x,C)$, there exists a transition point $\hat{x}$ on $[e,\hat{z}]$ such that $d(x,\hat x)\le C$.
By Lemma~\ref{ThinTransitional2} for the triangle with vertices $e,z,\hat{z}$, we see that $\hat{x}$ is within $C$ and so $x$ is within $2C$ of a transition point on $[e,z]$.
According to Lemma~\ref{TraceAvoidTransBall2}, if the branching random walk does not enter $B(x, \kappa\log |x|)$, then we have $ |z|> K |x|=\epsilon^{-1}|x|$.
Noting as above that $x$ is within $2C$ of a transition point on $[e,z]$,     Lemma~\ref{ProjectionLem} implies that $\pi_{N_\eta(P)}(e)$ is within a bounded   distance   of the entry point $y_1$ of $[x,z]$ into $N_\eta(P)$,  which implies that
$d(y_1,y_2) \leq d_{N_\eta(P)}(e,z)+C'$ for some constant $C'$ depending on $C$.
Moreover, if $n_0$ is large enough, then $d(y_1,y_2)\geq C'/2$, so we get
$d(y_1,y_2)\leq 2d_{N_\eta(P)}(e,z)$, hence
$$d(z,N_\eta(P))\leq \frac{2\epsilon}{1-2\epsilon}d_{N_\eta(P)}(e,z).$$
Furthermore, $d(e,N_\eta(P))\leq |x|+d(x,y_1)$.
Since $|x|\leq \epsilon|z|$ and $d(x,y_1)\leq \epsilon d(x,z)$, we get
$$d(e,N_\eta(P))\leq \left (\frac{\epsilon}{1-\epsilon}+\epsilon\right )d(x,z)\leq \left (\frac{\epsilon}{1-\epsilon}+\epsilon\right )(1-2\epsilon)2d_{N_\eta(P)}(e,z).$$

Thus, if $\epsilon$ is small enough, the conditions of Lemma~\ref{exponentialmoments} holds, hence the branching random walk first enters $N_\eta(P)$ at a point $w$ such that $d(\pi_{N_\eta(P)}(e),w)\leq d(e,N_\eta(P))$.
%
Since $x$ is within $2C$ of a transition point on $[e,w]$, $d(e,N_\eta(P))\geq |x|+d(x,N_\eta(P))-2C$
and so
$$d(x,w)\leq d(x, N_\eta(P))+d_{N_\eta(P)}(e,w)\leq 2 d(e, {N_\eta(P)})   +2C\leq 2d(e,w)+2C.$$Thus, for $n_0$ large enough, we have $d(e,w)\ge d(e,x)-2C\ge  C$ so $d(x,w)\le 4d(e,w)$. Applying   Lemma~\ref{TraceAvoidTransBall2} with $K\ge 4$ again, we see that the branching random walk necessarily enters $B(x,\kappa \log |x|)$.
This concludes the proof.
\end{proof}  

We can now end the proof of the upper-bound. 
 
 \begin{proof}[Proof of Proposition~\ref{upperboundHausdorff}]
Let $r\leq \rho^{-1}$ and fix $h$ such that
 $$h>\frac{\omega_\Gamma(r)}{-\log \lambda}.$$
 Let $0<\epsilon<1/2$.
 Let $\xi\in \Lambda^{con}_{\mathcal{F}}(r)$ and let $x$ be a transition point on $[e,\xi]$.
 By Lemma~\ref{LogTrackingLem}, almost surely, there exists $n_0$ such that if $|x|\geq n_0$, we can find $z\in \mathcal{P}$ such that $z\in B(x,\kappa \log |x|)$.
 In particular, $|x|\leq |z|+\kappa \log |x|$ and if $|x|$ is large enough, then $|z|\geq (1-\epsilon)|x|$, hence
 $$x\in B\left (z,\kappa \log \frac{|z|}{1-\epsilon}\right ).$$
 Consequently, for every $m$,
 $$\Lambda^{con}_{\mathcal{F}}(r)\subset \bigcup_{n\geq m}\bigcup_{z\in \mathcal{P}_n}\Pi\left (z,\kappa \log \frac{|z|}{1-\epsilon}\right ).$$
By Lemma~\ref{diametershadow}, the diameter of $\Pi\left (z,\kappa \log \frac{|z|}{1-\epsilon}\right )$ is bounded by $C\lambda^{|z|}|z|^\alpha \log |z|$.
Thus,
$$\sum_{n\geq m}\sum_{z\in \mathcal{P}_n}\mathbf{diam}\left (\Pi\left (z,\kappa \log \frac{|z|}{1-\epsilon}\right )\right )^h\leq C \sum_{n\geq m}M_n\lambda^{hn}n^{h\alpha} (\log n)^h.$$
Since $\limsup \frac{1}{n}\log M_n\leq \omega_{\Gamma}(r)$, by the choice of $h$, this last quantity converges to 0 as $m$ tends to infinity.
This concludes the proof.
 \end{proof}

We deduce the following result.
We denote by $\Lambda_\mathcal{B}(r)$ the limit set of the branching random walk inside the Bowditch boundary, endowed with the shortcut distance $\overline{\delta}_e$.
\begin{cor}\label{coroupperboundHausdorff}
Let $r\leq \rho^{-1}$.
Almost surely,
$$\Hdim(\Lambda_\mathcal{B}(r),\overline{\delta}_e)\leq \frac{-1}{\log \lambda}\omega_{\Gamma}(r).$$
\end{cor}

\begin{proof}
Combining Proposition~\ref{upperboundHausdorff} and~(\ref{shortcutsmallerFloyd}), we get that
$$\Hdim(\Lambda_\mathcal{B}(r)\cap \partial^{con}_{\mathcal{B}}\Gamma,\overline{\delta}_e)\leq \frac{-1}{\log \lambda}\omega_{\Gamma}(r).$$
The complement of the set of conical limit points in the Bowditch boundary is the set of parabolic limit points, which is countable.
This yields the desired upper-bound.
\end{proof}
Theorem~\ref{ThmHdimconst} is a consequence of Proposition~\ref{lowerboundHdim} and Corollary~\ref{coroupperboundHausdorff}.
Recall that if $\Gamma$ is hyperbolic, then it is also relatively hyperbolic and its Bowditch, Floyd and Gromov boundaries coincide.
Moreover, the shortcut distance and the visual distance are bi-Lipschitz by \cite[Proposition~6.1]{PY}.
Thus, Corollary~\ref{corohausdorffdimensionhyperbolic} follows from Theorem~\ref{ThmHdimconst}.

\appendix
\section{Convergence of random finite measures}

\subsection{Random finite measures}
Let $(X,\mathcal{B})$ be a Polish space endowed with its Borelian $\sigma$-algebra.
Let $(\Omega,\mathcal{F},\mathbf P)$ be a probability space.
\begin{defn}\label{defrandomfinitemeasure}
A random finite measure on $X$ is a map
$$\mu : (\omega,B)\in \Omega\times \mathcal{B}\mapsto \mu_\omega(B)\in \mathbb R$$
such that
\begin{enumerate}[(a)]
    \item for every $B\in \mathcal{B}$, the map $\omega\mapsto \mu_\omega(B)$ is measurable,
    \item for $\mathbf P$-almost every $\omega$, $\omega\mapsto \mu_\omega$ is a finite Borelian measure on $X$,
    \item the expectation $\mathbf E[\mu_\omega(X)]$ is finite,
\end{enumerate}
We identify $\mu$ with the family of maps $\mu_\omega:\mathcal{B}\to \mathbb R$ and we write $\mu=(\mu_\omega)_\omega$.
\end{defn}

We denote by $\mathcal{M}_\Omega(X)$ the set of random finite measure on $X$ and by $\mathcal{M}(X)$ the set of finite measures on $X$.
Given a random finite measure $\mu=(\mu_\omega)_\omega$, we define the measure $\pi_X(\mu)$ on $X$ by
$$\pi_X(\mu)(B)=\mathbf E\left [\mu_\omega(B) \right]$$
for every Borelian set $B\in \mathcal B$.
We also define the measure $\pi_\Omega(\mu)$ on $\Omega$ by
$$\pi_\Omega(\mu)(A)=\mathbf E[\mathbf 1_A\mu_\omega(X)].$$
We call $\pi_X(\mu)$, respectively $\pi_\Omega(\mu)$, the $X$-marginal, respectively the $\Omega$-marginal of $\mu$.

If $\mu=(\mu_\omega)_\omega$ is a random finite measure, then one can define a finite measure $\tilde{\mu}$ on $X\times \Omega$  by setting for every measurable set $A$ of $X\times \Omega$
$$\mu(A)=\mathbf E\left [\int \mathbf 1_{A}(x,\omega)d\mu_\omega(x)\right ].$$
Then, $\pi_X(\mu)$ and $\pi_\Omega(\mu)$ are the push-forward measures of $\mu$ by the canonical projections $\pi_X:X\times \Omega \to X$ and $\pi_\Omega:X\times \Omega\to \Omega$.

\begin{defn}
A random finite measure $\mu$ is called $\mathbf P$-regular if $\mathbf P$ and $\pi_\Omega(\mu)$ are absolutely continuous with respect to each other.
\end{defn}

\begin{lem}
Let $\mu=(\mu_\omega)_\omega$ be a random finite measure.
Then $\pi_\Omega(\mu)$ is absolutely continuous with respect to $\mathbf P$.
Moreover, $\mu$ is $\mathbf P$-regular if and only if for $\mathbf P$-almost every $\omega$, $\mu_\omega$ is not the null measure.
\end{lem}

\begin{proof}
If $A$ is  such that $\mathbf P(A)=0$, then $\mathbf P$-almost surely, $\mathbf 1_A \mu_\omega(X)=0$, so $\pi_\Omega(\mu)(A)=\mathbf E[\mathbf 1_A\mu_\omega(X)]=0$.
This concludes the first part of the lemma.

For the second part, assume that the event $A=\{\omega,\mu_\omega(X)=0\}$ has positive probability.
Then, $\mathbf E[\mathbf 1_A \mu_\omega(X)]=0$.
So $\pi_\Omega(\mu)(A)=0$, but $\mathbf P(A)>0$.
Conversely, assume that $\mathbf{P}$-almost surely, $\mu_\omega(X)>0$ and let $A$ be such that
$\pi_\Omega(\mu)(A)=0$, hence $\mathbf P$-almost surely, $\mathbf 1_A \mu_\omega(X)=0$.
Then, we necessarily have $\mathbf 1_A=0$ $\mathbf P$-almost surely, i.e. $\mathbf P(A)=0$. 
\end{proof}

\begin{rem}
In fact, by definition of $\pi_\Omega(\mu)$, we have that the Radon-Nikodym derivative of $\pi_\Omega(\mu)$ with respect to $\mathbf P$ is given by
$$\frac{d\pi_\Omega(\mu)}{d\mathbf P}(\omega)=\mu_\omega(X).$$
Thus, $\mathbf P$ is absolutely continuous with respect to $\pi_\Omega(\mu)$ if and only if this Radon-Nikodym derivative is almost surely positive and then,
$$\frac{d\mathbf P}{d\pi_\Omega(\mu)}(\omega)=\frac{1}{\mu_\omega(X)}$$
$\mathbf P$-almost surely.
\end{rem}

If $\mu=(\mu_\omega)_\omega$ is a random probability measure, i.e.\ $\mathbf P$-almost surely, $\mu_\omega$ is a probability measure on $X$, then the $\Omega$-marginal of $\mu$ is $\mathbf P$.
Thus, $\mathbf P$-regularity is automatic in this context.
However, it is easy to construct an example where $\mathbf P$-regularity fails, since one only needs that $\mathbf P(\mu_\omega(X)=0)>0$.
Indeed, let $\mu$ be any random finite measure and let $A$ be an event such that $\mathbf P(A)\leq 1/2$, then $\mathbf 1_A\mu$ is also a random finite measure and
$\mathbf P(\mathbf 1_A\mu(X)=0)\geq 1/2$.
Restricting our attention to $\mathbf P$-regular measures will be important in the following, mainly because of the following result.

\begin{prop}\label{disintegrationrandom}
Every finite measure $\mu$ on $X\times \Omega$ such that $\pi_\Omega(\mu)$ and $\mathbf P$ are absolutely continuous with respect to each other is defined by a random finite measure.
\end{prop}

\begin{proof}
Let $\mu$ be a finite measure on $X\times \Omega$ and denote by $\|\mu\|=\mu(X\times \Omega)$ its mass.
Then, $\nu=\mu/\|\mu\|$ is a probability measure on $X\times \Omega$ with marginal $\mathbf Q=\pi_\Omega(\mu)/\|\mu\|$.
By the disintegration theorem \cite[Proposition~3.6]{Crauel}, there exists a map
$$(B,\omega)\in \mathcal{B}\times \Omega\mapsto \nu_\omega(B)$$
such that
$\omega\mapsto \nu_\omega$ is $\mathbf Q$-almost surely a probability measure on $X$ and for every Borelian $B\in \mathcal{B}$, $\omega \mapsto \nu_\omega(B)$ is measurable.

By assumption, the probability measure $\mathbf Q$ is absolutely continuous with respect to $\mathbf P$.
Let $f=d\mathbf Q/d\mathbf P$ be the corresponding Radon-Nikodym derivative.
Define then $\mu_\omega=\|\mu\|f(\omega)\nu_\omega$.
Then, $\mu_\omega$ is $\mathbf Q$-almost surely, hence $\mathbf P$-almost surely, a finite measure on $X$, for every Borelian set $B$, the map $\omega\mapsto \mu_\omega(B)$ is measurable and for every non-negative measurable function $\phi$ on $X\times \Omega$, we have
\begin{align*}
\mu(\phi)&=\|\mu\|\nu(\phi)\\&=\|\mu\|\int \int \phi(x,\omega) d\nu_\omega(x) d\mathbf Q(\omega)\\&=\int \int \phi(x,\omega)\|\mu\|d\nu_\omega(x) f(\omega)d\mathbf P(\omega)\\
&=\mathbf E \left [\int \phi(x,\omega)d\mu_\omega(x)\right ].
\end{align*}
Thus, $\mu$ is defined by the random finite measure $(\mu_\omega)_\omega$.
\end{proof}

In general, the disintegration theorem allows one to decompose a finite measure on $X\times \Omega$ along its $\Omega$-marginal, so without assuming $\mathbf P$-regularity, it cannot be seen as a random finite measure in the sense of Definition~\ref{defrandomfinitemeasure}.

We now introduce some definitions based on \cite{Crauel}.

\begin{defn}\label{defrandomclosedset}
A random closed set is a map $\omega \in \Omega \mapsto C(\omega)\in 2^X$ taking values in closed subsets of $X$ and such that the map $\omega\mapsto d(x,C(\omega))$ is measurable for every $x\in X$.
A random open set is a map $\omega \mapsto U(\omega)$ such that the complement map $\omega\mapsto U^c(\omega)$ is a random closed set.
\end{defn}

Let $C=\omega\mapsto C(\omega)$ be a random closed set, $U=\omega\mapsto U(\omega)$ a random open set and $\mu=(\mu_\omega)_\omega$ a random finite measure.
We set
$$\mu(C)=\mathbf E\left [\mu_\omega(C(\omega))\right]$$
and
$$\mu(U)=\mathbf E\left [\mu_\omega(U(\omega))\right].$$

\begin{defn}\label{defrandomcontinuousfunction}
A random bounded continuous function is a map
$f:X\times \Omega\to \mathbb R$ such that
\begin{enumerate}[(a)]
    \item for every $x\in X$, $\omega\mapsto f(x,\omega)$ is measurable,
    \item for every $\omega\in \Omega$, $x\mapsto f(x,\omega)$ is continuous and bounded,
    \item There exists $C\geq 0$ such that for $\mathbf P$-almost every $\omega$, $\|f(\cdot,\omega)\|_\infty\leq C$.
\end{enumerate}
\end{defn}

\begin{rem}
The third condition can be reformulated as $\|f(\cdot,\omega)\|_\infty$ is in $L_\infty(\Omega,\mathbf P)$.
In \cite{Crauel}, the author introduces several spaces of random functions, replacing the third condition by $\|f(\cdot,\omega)\|_\infty\in L_p(\Omega,\mathbf P)$.
For $p=1$ the corresponding space of function is called the space of random continuous functions there.
\end{rem}

We denote by $C_{\Omega,b}(X)$ the space of random bounded continuous functions and endow $C_{\Omega,b}(X)$ with the $L_\infty\times L_\infty$-norm $\|\cdot\|_\infty$,
defined by
$$\|f\|_\infty=\inf \left\{C\geq 0, \mathbf P\big (\omega,\|f(\cdot,\omega)\|_\infty>C\big)=0\right\},$$
for every $f\in C_{\Omega,b}(X)$.
If $f$ is a random bounded continuous function and $\mu$ is a random finite measure, then the integral
$$\mu(f)=\mathbf E\left [\int f(x,\omega)d\mu_\omega(x)\right]$$
is well defined.

Recall that a Lipschitz function on $X$ is a function $f:X\to \mathbb R$ such that
$$\|f\|_L=\sup_{x,y\in X}\frac{|f(x)-f(y)|}{d(x,y)}$$
is finite.
We then set
$$\|f\|_{BL}=\sup \{\|f\|_\infty,\|f\|_L\}$$
and say that $f$ is bounded Lipschitz if $\|f\|_{BL}$ is finite.
We denote by $BL(X)$ the set of bounded Lipschitz functions on $X$.

\begin{defn}\label{defrandomlipschitzfunction}
A random bounded Lipschitz function is a random bounded continuous function $f$ such that there exists $C\geq0$ such that for $\mathbf P$-almost every $\omega$, the map 
$x\mapsto f(x,\omega)$ is bounded Lipschitz and $\|f(\cdot,\omega)\|_{BL}\leq C$.
\end{defn}

\begin{rem}
In \cite{Crauel}, random bounded Lipschitz functions are called random Lipschitz functions.
We changed the terminology to insist on the fact that random functions we are considering here are $\mathbf P$-essentially bounded.
\end{rem}
We denote by $BL_\Omega(X)$ the set of random bounded Lipschitz functions on $X$.

\begin{lem}\label{measuredeterminedbylipschitz}
If two random finite measures $(\mu_\omega)_\omega$ and $(\nu_\omega)_\omega$ coincide on random bounded Lipschitz functions, i.e. for every $f\in BL_\Omega(X)$, $\mu(f)=\nu(f)$, then for $\mathbf P$-almost every $\omega$, $\mu_\omega=\nu_\omega$.
\end{lem}

\begin{proof}
For every closed set $C\in X$, the sequence of functions
$$f_n:x\mapsto 1-(1\wedge nd(x,C))$$ is non-increasing and converges to $\mathbf 1_C$.
Moreover, for every $n$, $f_n$ is bounded Lipschitz.
Therefore, for every event $A$, the maps $(x,\omega)\mapsto \mathbf{1}_A(\omega)f_n(x)$ are random bounded Lipschitz functions.
Thus by monotone convergence, it suffices to prove that if for all closed set $C$, for all event $A$, $\mu(C\times A)=\nu(C\times A)$, then $(\mu_\omega)_\omega$ and $(\nu_\omega)_\omega$ coincide $\mathbf P$-almost surely.
Since closed sets are closed under finite intersection and the measures we are considering are finite, this follows from the monotone class theorem \cite[Theorem~3.4]{Bill86}.
\end{proof}

\subsection{The weak topology on random finite measures}

Recall that the weak topology on $\mathcal{M}(X)$ is the smallest topology such that for all bounded continuous function $f:X\to \mathbb R$, the map $\mu\mapsto \mu(f)$ is continuous.

\begin{defn}\label{defweaktopology}
The weak topology on $\mathcal{M}_\Omega(X)$ is the topology generated by the maps $\mu \mapsto \mu(f)$, for every $f\in C_{\Omega,b}(X)$, i.e. it is the smallest topology on $\mathcal{M}_\Omega(X)$ such that for every $f\in C_{\Omega,b}(X)$, the map $\mu\in \mathcal{M}_\Omega(X)\mapsto \mu(f)\in \mathbb R$ is continuous. 
\end{defn}

\begin{rem}
In \cite{Crauel}, the author defines the narrow topology on the space of random probability measures as the topology generated by the maps $\mu\mapsto \mu(f)$ for every random continuous function $f$.
    Recall that a random continuous function as defined there is a function $f$ such that for all $\omega$, $f(\cdot,\omega)$ is bounded continuous and $\|f(\cdot,\omega)\|_\infty$ is in $L_1(\Omega,\mathbf P)$.
\begin{itemize}
    \item First, we preferred to use the terminology weak topology which is more common, although both exist in literature.
    \item Second, by \cite[Lemma~3.16]{Crauel}, for random probability measures, the induced weak topology on the set of measures is the same when choosing either random bounded continuous functions or random continuous functions.
    However, in our context, the proof of this lemma does not apply and it seems that choosing different spaces of functions can yield different notions of weak topologies.
\end{itemize}
\end{rem}

\begin{lem}\label{lipschitzgenerateweaktopology}
The weak topology is generated by the maps $\mu \mapsto \mu(f)$ for every $f\in BL_\Omega(X)$, i.e. it is the smallest topology on $\mathcal{M}_\Omega(X)$ such that for every random bounded Lipschitz function $f$, $\mu\mapsto \mu(f)$ is continuous.
\end{lem}

\begin{proof}
This follows from the fact that bounded continuous functions can be approximated by bounded Lipschitz functions, see \cite[Proposition~4.9]{Crauel} for more details.
\end{proof}

We now prove the following generalization of the classical Portmanteau theorem in terms of convergent nets.
We refer to \cite[Definition~2.11, Definition~2.15]{AliprantisBorder} for more details on nets.

\begin{prop}\label{randomPortmanteau}
Let $\mu_{\alpha}=((\mu_\alpha)_\omega)_\omega$ be a net of random finite measures and let $\mu=(\mu_\omega)_\omega$ be a random finite measure.
Then, the following assertions are equivalent.
\begin{enumerate}
    \item The net $\mu_\alpha$ converges to $\mu$ for the weak topology.
    \item For all random closed set $C=\omega\mapsto C(\omega)$, $\limsup_\alpha\mu_\alpha(C)\leq \mu(C)$ and $\mu_\alpha(X\times \Omega)$ converges to $\mu(X\times \Omega)$.
    \item For all random open set $U=\omega\mapsto U(\omega)$, $\liminf_\alpha\mu_n(U)\geq \mu(U)$ and $\mu_\alpha(X\times \Omega)$ converges to $\mu(X\times \Omega)$.
\end{enumerate}
\end{prop}

\begin{proof}
Taking complements, the second and third assertions are equivalent.
Also, if $\mu_\alpha$ converges to $\mu$, applying the definition of the weak topology to the constant function $\mathbf 1$, we see that $\mu_\alpha(X\times \Omega)$ converges to $\mu(X\times \Omega)$.

Let us assume that $\mu_\alpha$ converges to $\mu$ and let $C$ be a closed random set.
For every $k\in \mathbb N$, set
$$f_k(x,\omega)=1-(1\wedge kd(x,C(\omega))).$$
Then, for every $k$, $f_k$ is a random bounded Lipschitz function and the sequence $f_k$ is non-increasing and converges to $\mathbf 1_{C(\omega)}(x)$.
Thus, for every $k$,
$$\limsup_\alpha \mu_\alpha(C)\leq \lim_\alpha \mu_\alpha(f_k)=\mu(f_k),$$
so
$$\limsup_\alpha \mu_\alpha(C)\leq \inf_k(\mu(f))=\mu(C).$$
Consequently, the first assertion implies the second one.

Assume now that for all closed random set $C$, $\limsup \mu_\alpha(C)\leq \mu(C)$ and that $\mu_\alpha(X\times \Omega)$ converges to $\mu(X\times \Omega)$.
We show that for every non-negative random bounded continuous function $f$,
\begin{equation}\label{equationPortmanteau}
\limsup \mu_\alpha(f)\leq \mu(f).
\end{equation}
Fix $m\in \mathbb N$ and set for every $0\leq k\leq m$
$$C_k(\omega)=\left \{x\in X,f(x,\omega)\geq \frac{k}{m}\|f(\cdot,\omega)\|_\infty\right \}.$$
Then, $C_k=\omega\mapsto C_k(\omega)$ is a random closed set.
By \cite[Lemma~1.4]{Crauel}, for every random finite measure $(\nu_\omega)_\omega$,
$$\frac{1}{m}\sum_{k=1}^m\nu(C_k)\leq \nu(f)\leq \frac{1}{m}\sum_{k=0}^m\nu(C_k).$$
Applying this both to $\mu_\alpha$ and $\mu$, we get

\begin{align*}
\mu(f)\geq \frac{1}{m}\sum_{k=1}^m\mu(C_k)&\geq \limsup_\alpha\frac{1}{m}\sum_{k=1}^m\mu_\alpha(C_k)\\
&=\limsup_\alpha\left (\frac{1}{m}\sum_{k=0}^m\mu_\alpha(C_k)-\frac{\mu_\alpha(X\times \Omega)}{m}\right )\\
&\geq \limsup_\alpha\mu_\alpha(f)-\frac{\mu(X\times\Omega)}{m}.
\end{align*}
Since $m$ is arbitrary, this proves~(\ref{equationPortmanteau}).
Using that $\mu_\alpha(\mathbf 1)=\mu_\alpha(X\times \Omega)$ converges to $\mu(\mathbf 1)=\mu(X\times \Omega)$ and applying~(\ref{equationPortmanteau}) to the function $\Vert f\Vert_{\infty}-f$, we get that
$\mu_\alpha(f)$ converges to $\mu(f)$.
This is true for all non-negative random bounded continuous function, so $\mu_\alpha$ converges to $\mu$ for the weak topology.
\end{proof}

\begin{rem}
By Lemma~\ref{measuredeterminedbylipschitz}, the weak topology is Hausdorff.
We do not attempt to study metrizability of the weak topology in here to avoid lengthily arguments,
but in \cite[Theorem~4.16]{Crauel}, the author proves that the weak topology on the space of random probability measures is metrizable, provided that the probability space $(\Omega,\mathcal{F},\mathbf P)$ is countably generated (mod. $\mathbf P$).
This might also holds in our situation.
\end{rem}

\subsection{A compactness criterion for random finite measures}
In all this section, $\mathcal{M}(X)$ and $\mathcal{M}_\Omega(X)$ are endowed with the weak topology.
Let us recall the following definition.
\begin{defn}\label{deftightnonrandom}
A subset $M$ of $\mathcal{M}(X)$ is tight if for every $\epsilon>0$, there exists a compact subset of $X$ such that for every $\mu\in M$, we have
$$\mu(K^c)\leq \epsilon.$$
\end{defn}

Following \cite{Crauel}, we define tightness for random finite measures as tightness under the projection map $\pi_X$.
\begin{defn}\label{deftightrandom}
A subset $M_\Omega$ of $\mathcal{M}_\Omega(X)$ is tight if $\pi_X(M_\Omega)$ is tight, i.e. for every $\epsilon>0$, there exists a compact subset of $X$ such that for all $(\mu_\omega)_\omega\in M_\Omega$, we have
$$\mathbf E[\mu_\omega(K^c)]\leq \epsilon.$$
\end{defn}

The classical Prokhorov theorem states that a set $M$ of probability measures on $X$ is relatively compact if and only if it is tight.
The following generalizes this result to finite measures.

\begin{thm}\label{classicalProkhorov}\textbf{Prokhorov Theorem for finite measures} \cite[Lemma~4.4]{K17}.
A subset $M$ of $\mathcal{M}(X)$ is relatively compact if and only if $M$ is tight and uniformly bounded, in the sense that $\sup_{\mu\in M} \mu(X)$ is finite.
\end{thm}

Our goal is to generalize this to random finite measures.
Unfortunately, we will not get a necessary and sufficient condition for compactness as in the Prokhorov Theorem but only a sufficient condition.

We follow the strategy of \cite[Theorem~4.4]{Crauel} and first prove the following representation result.
Recall that $BL_\Omega(X)$ denotes the set of random bounded Lipschitz functions and $BL(X)$ denotes the set of bounded Lipschitz functions on $X$.
Then, if $f\in BL(X)$, $f$ can be viewed as an element of $BL_\Omega(X)$ by setting $f(x,\omega)=f(x)$.
Similarly, if $f$ is $\mathbf P$ -essentially bounded, i.e.\ $f\in L_\infty(\Omega,\mathbf P)$, then $f$ can be viewed as an element of $BL_\Omega(X)$ by setting $f(x,\omega)=f(\omega)$.
If $L:BL_\Omega(X)\to \mathbb R$ is a function, we denote by $\pi_X(L)$, respectively $\pi_\Omega(L)$ its restriction to bounded Lipschitz functions, respectively to $\mathbf P$-essentially bounded functions.

\begin{lem}\label{StoneDaniell}
Let $L:BL_\Omega(X)\to \mathbb R$. be a function.
Assume that the following conditions hold
\begin{enumerate}[(a)]
    \item $L$ is linear,
    \item $L$ is non-negative, i.e.\ for every non-negative function $f$ in $BL_\Omega(X)$, we have $L(f)\geq 0$,
    \item there exists $\kappa\in \mathcal{M}(X)$ such that for every $f\in BL(X)$,
    we have
    $$\pi_X(L)(f)=\kappa(f)=\int fd\kappa,$$
    \item there exists a constant $C>0$ such that for every $f\in L_\infty(\Omega,\mathbf P)$, we have
    $$\frac{1}{C}\pi_\Omega(L)(f)\leq \mathbf E [f]\leq C\pi_\Omega(L)(f).$$
\end{enumerate}
Then, there exists a random finite measure $\mu=(\mu_\omega)_\omega$ such that $L(f)=\mu(f)$ for every $f\in BL_\Omega(X)$.
\end{lem}

\begin{proof}
The proof relies on the general Stone-Daniell representation theorem.
We claim that if $f_n$ is a non-increasing sequence of functions of $BL_\Omega(X)$ converging to 0, then $L(f_n)$ is non-increasing and converges to 0.
By \cite[Theorem~4.11]{Crauel}, we deduce that there exists a measure $\mu$ on $X\times \Omega$ such that $L(f)=\mu(f)$ for every $f\in BL_\Omega(X)$.
Conditions~(c) and~(d) ensure that $\mu$ is finite.
Moreover, the $\Omega$-marginal $\pi_\Omega(\mu)$ of $\mu$ is $\pi_\Omega(L)$ and so the fourth condition shows $\pi_\Omega(\mu)$ and $\mathbf P$ are absolutely continuous with respect to each other.
Therefore, we can apply Proposition~\ref{disintegrationrandom}, so that $\mu$ is defined by a random finite measure.

We just need to prove the claim to conclude the proof.
Let $f_n$ be a non-increasing sequence of random bounded Lipschitz functions converging to 0.
Since $L$ is linear and non-negative, we get that $L(f_n)$ is non-inscreasing.
We need to prove that $L(f_n)$ converges to 0.
Fix $\epsilon>0$.
Let $K$ be a compact such that $\kappa(K^c)\leq \epsilon$.
For every $n$, there exists $C_n$ such that $\|f_n(\cdot,\omega)\|_{BL}\leq C_n$, for $\mathbf P$-almost every $\omega$.
Set $\delta_n=\epsilon/C_n$ and consider the function
$$\chi_n:x\mapsto \chi_n(x)=1-\big(1\wedge \delta_n^{-1}d(x,K)\big).$$
Then, $\chi_n$ vanishes outside the $\delta_n$-neighborhood of $K$.
Moreover, the function
$g_n:(x,\omega)\mapsto f_n(x,\omega)\chi_n(x)$ is in $BL_\Omega(X)$.
Now, $f_n=g_n+f_n(1-\chi_n)$ and so
\begin{equation}\label{equationStoneDaniell1}
    L(f_n)=L(g_n)+L(f_n(1-\chi_n)).
\end{equation}

We first deal with the second term in the right-hand side of~(\ref{equationStoneDaniell1}).
Since $f_n$ is non-increasing, we have that $f_n(x,\omega)\leq \|f_1(\cdot,\omega)\|_\infty$ Moreover, there exists $M\geq 0$ such that the event $A=\{\|f_1(\cdot,\omega)\|_\infty\leq M\}$ satisfies $\mathbf P(A^c)=0$.
Note that we can write $f_n(1-\chi_n)\leq M(1-\chi_n)\mathbf 1_A+\|f_1(\cdot,\omega)\|\mathbf 1_{A^c}$.
By Condition~(d), $L(\|f_1(\cdot,\omega)\|\mathbf 1_{A^c})=0$.
Thus,
$$L(f_n(1-\chi_n))\leq ML(1-\chi_n)=M\int (1-\chi_n)d\kappa.$$
Note that $1-\chi_n$ vanishes in $K$, so
\begin{equation}\label{equationStoneDaniell2}
L(f_n(1-\chi_n))\leq M\kappa (K^c)\leq M\epsilon.
\end{equation}

We now deal with the first term in the right-hand side of~(\ref{equationStoneDaniell1}).
For every $\omega$ such that $\|f_n(\cdot,\omega)\|_{BL}\leq C_n$ and for every $x,y\in X$, we have
$$|g_n(x,\omega)-g_n(y,\omega)|\leq C_nd(x,y)+2f_n(y,\omega).$$
If $x$ is in the $\delta_n$-neighborhood of $K$, choose $y\in K$ such that
$d(x,y)\leq \delta_n$.
Then, by what precedes,
$$g_n(x,\omega)\leq 3\sup_{y\in K}f_n(y,\omega)+C_n\delta_n\leq 3\sup_{y\in K}f_n(y,\omega)+\epsilon.$$
Since $\chi_n$ vanishes outside the $\delta_n$-neighborhood of $K$, this yields
$$\sup_{x\in X}g_n(x,\omega)\leq 3\sup_{y\in K}f_n(y,\omega)+\epsilon.$$
We write $h_n(\omega)=\sup_{y\in K}f_n(y,\omega)$.
The event $A_n= \{\omega$, $\|f_n(\cdot,\omega)\|_{BL}>C_n\}$ satisfies $\mathbf P(A_n)=0$, hence the same manipulation as above shows that
\begin{equation}\label{equationStoneDaniell3}
L(g_n)\leq 3L(h_n)+\epsilon L(\mathbf 1)\leq 3C\mathbf E(h_n)+\kappa(X)\epsilon,
\end{equation}
using Conditions~(c) and~(d).

Combining~(\ref{equationStoneDaniell1}),~(\ref{equationStoneDaniell2}) and~(\ref{equationStoneDaniell3}), we get
$$L(f_n)\leq 3C\mathbf E(h_n)+(\kappa(X)+M)\epsilon.$$
Using that $f_n$ is non-increasing and converges to 0, that for all $\omega$, $f_n(\cdot,\omega)$ is continuous and that $K$ is compact, we get that $h_n$ is non-increasing and converges to 0.
By monotone convergence, $\mathbf E(h_n)$ converges to 0,
so for large enough $n$, $\mathbf E(h_n)\leq \epsilon$.
Thus, for large enough $n$,
$$L(f_n)\leq (3C+\kappa(X)+M)\epsilon.$$
Since $\epsilon$ is arbitrary, this concludes the proof of the claim.
\end{proof}

The main ingredient in proving a compactness criterion for random finite measures is the following proposition.

\begin{prop}
Let $C>0$.
If $K$ is compact in $\mathcal{M}(X)$, then
$$\pi_X^{-1}(K)\cap \left \{\mu \in \mathcal{M}_\Omega(X), \frac{1}{C}\pi_\Omega(\mu)\leq \mathbf P\leq C\pi_\Omega(\mu)\right \}$$ is compact in $\mathcal{M}_\Omega(X)$.
\end{prop}

\begin{proof}
Let $f\in BL_\Omega(X)$ and set
$$M_f^+=\sup_{\mu\in \pi_X^{-1}(K)}\mu(f),$$
$$M_f^-=\inf_{\mu\in \pi_X^{-1}(K)}\mu(f).$$
Recall that if $f\in BL_\Omega(X)$, then in particular, $f\in C_{\Omega,b}(X)$ and so for $\mathbf P$-almost every $\omega$, $|f(\cdot,\omega)|$ is bounded and
$\mathbf E[\|f(\cdot,\omega)\|_\infty]$ is bounded.
Thus, $\pi_X(\mu)(f)$ is finite.
Since $K$ is compact, this proves that $M_f^+$ and $M_f^-$ are finite.

We set
$$\mathbf C=\prod_{f\in BL_\Omega(X)}[M_f^-,M_f^+]$$
and we endow $\mathbf C$ with the product topology.
Then $\mathbf C$ is compact by the Tychonoff Theorem \cite[Theorem~2.61]{AliprantisBorder}.
We can identify $\mathbf C$ with the set of all functions $L:BL_\Omega(X)\to \mathbb R$ such that $L(f)\in [M_f^-,M_f^+]$.
A neighborhood basis of an element $L_0$ of $\mathbf C$ is given by
$$U_{\delta}(f_1,...f_n)(L_0)=\{L\in \mathbf C,|L(f_k)-L_0(f_k)|<\delta,1\leq k\leq n\},$$
where $\delta>0$ and $f_1,...,f_n\in BL_\Omega(X)$.
We define the map
$$I:\mu\in \pi_X^{-1}(K)\mapsto (\mu(f))_{f\in BL_\Omega(X)}\in \mathbf C.$$
Then, $I$ is one-to-one and continuous.
Moreover, by Lemma~\ref{lipschitzgenerateweaktopology}, a neighborhood basis of an element $\mu_0\in \pi_X^{-1}(K)$ for the weak topology is given by
$$V_\delta(f_1,...,f_n)(\mu_0)=\{\mu \in \pi_X^{-1}(K), |\mu(f_k)-\mu_0(f_k)|<\delta,1\leq k\leq n\},$$
where $\delta>0$ and $f_1,...,f_n\in BL_\Omega(X)$.
Since
$$I\big (V_\delta(f_1,...,f_n)(\mu_0)\big)=U_\delta(f_1,...,f_n)(I(\mu_0)),$$
we see that $I$ is an open map and thus a homeomorphism onto its image in $\mathbf C$.

We make a similar construction for non-random finite measures.
Recall that $BL(X)$ is the set of (non-random) bounded Lipschitz functions on $X$ and that $BL(X)$ can be seen as a subset of $BL_\Omega(X)$ by setting $f(x,\omega)=f(x)$ for any $f\in BL(X)$.
We define
$$\mathbf c=\prod_{f\in BL(X)}[M_f^-,M_f^+]$$
and
$$i(\kappa)=(\kappa(f))_{f\in BL(X)}.$$
Then, $i$ is a homeomorphism onto its image in $\mathbf c$.
If $L\in \mathbf C$, denote by $\pi(L)\in \mathbf c$ its restriction to non-random bounded Lipschitz functions.

For every $f\in BL(X)$, seen as an element of $BL_\Omega(x)$ and for every random finite measure $\mu$, we have
$$\mu(f)=\mathbf E\left [\int f(x)d\mu_\omega(x)\right ]=\pi_X(\mu)(f).$$
Thus, $\pi \circ I=i\circ \pi_X$.
Now, fix $C>0$.
By Lemma~\ref{StoneDaniell}, we have
\begin{align*}
&I\left (\pi_X^{-1}(K)\cap \left \{\mu, \frac{1}{C}\leq \pi_\Omega(\mu)\leq \mathbf P\leq C\pi_\Omega(\mu)\right\}\right)\\
=&\hspace{.2cm}\pi^{-1}(i(K))\cap \{L \text{ linear}\}\cap \{ L\text{ non-negative}\}\\
&\hspace{.3cm}\cap \left \{L, \frac{1}{C}\pi_\Omega(L)\leq \mathbf E[\cdot]\leq C\pi_\Omega(L)\right \}.
\end{align*}
The four sets on the right-hand side are closed.
Since $\mathbf C$ is compact and $I$ is a homeomorphism onto its image, we get that
$$ \pi_X^{-1}(K)\cap \left \{\mu, \frac{1}{C}\leq \pi_\Omega(\mu)\leq \mathbf P\leq C\pi_\Omega(\mu)\right\}$$
is compact.
\end{proof}


\begin{cor}\label{coroProkhorov}
Consider a subset $M_\Omega$ of $\mathcal{M}_\Omega(X)$.
Assume that $M_\Omega$ is tight.
Also assume that there exists $C\geq 0$ such that for every $\mu\in M_\Omega$,
$$\frac{1}{C}\pi_\Omega(\mu)\leq \mathbf P \leq C \pi_\Omega(\mu).$$
Then, $M_\Omega$ is relatively compact.
\end{cor}

\begin{proof}
If $M_\Omega$ satisfies the assumptions of the corollary, then $\pi_X(M_\Omega)\subset \mathcal{M}(X)$ is tight.
Also, the condition $\pi_\Omega(\mu)\leq C \mathbf P$ can be reformulated by
$$\frac{d\pi_\omega(\mu)}{d\mathbf P}(\omega)=\mu_\omega(X)\leq C$$
$\mathbf P$-almost surely.
In particular,
$$\pi_X\mu(X)=\mathbf E[\mu_\omega(X)]\leq C,$$
so $\pi_X(\mu)(X)$ is uniformly bounded.
By Theorem~\ref{classicalProkhorov}, $\pi_X(M_\Omega)$ is relatively compact, i.e.\ $\mathbf{cl}(\pi_X(M_\Omega))$ is compact.
Moreover,
$$M_\Omega\subset \pi_X^{-1}\big(\mathbf{cl}(\pi_X(M_\Omega))\big )\cap \left \{\mu,\frac{1}{C}\pi_\Omega(\mu)\leq \mathbf P\leq C\pi_\Omega(\mu)\right \}.$$
Thus, $\mathbf{cl}(M_\Omega)$ is compact.
\end{proof}

This corollary does not give a necessary and sufficient condition for compactness, but only a sufficient one.
On the contrary, \cite[Theorem~4.4]{Crauel} gives a necessary and sufficient condition for compactness in the context of random probability measures : a set is relatively compact if and only if it is tight.
However, assuming that for $\mathbf P$-almost every $\omega$, $\mu_\omega$ is a probability measure, we get that $\pi_X(\mu)$ is a probability measures and that the $\Omega$-marginal of $\mu$ is exactly $\mathbf P$.
This ensures that the others conditions of Corollary~\ref{coroProkhorov} are automatically satisfied and so tightness is sufficient to get compactness.

Recall that the Radon-Nikodym derivative of $\pi_\Omega(\mu)$ with respect to $\mathbf P$ is given by $\mu_\omega(X)$.
In our context, on the one hand, the assumption
$$\frac{1}{C}\pi_\Omega(\mu)\leq \mathbf P\leq C\pi_\Omega(\mu)$$
which is equivalent to
$$\frac{1}{C}\leq \mu_\omega(X)\leq C$$
for $\mathbf P$-almost every $\omega$
seems strong.
Assuming only $\mathbf P$-regularity, there is no reason for this Radon-Nikodym derivative to be $\mathbf P$-essentially bounded from above and below.
Even when restricting our attention to random finite measures satisfying this property, the maps
$$\mu\mapsto \mathbf P-\textbf{ess}\sup \mu_\omega(X)$$
and
$$\mu\mapsto \mathbf P-\textbf{ess}\inf \mu_\omega(X)$$
have no reason to be continuous.
Thus, compactness for the weak topology does not seem to imply the existence of a uniform $C$ such that the condition 
$$\frac{1}{C}\pi_\Omega(\mu)\leq \mathbf P\leq C\pi_\Omega(\mu)$$
holds.
On the other hand, it is not realistic to expect a sufficient condition for compactness without assuming anything on the $\Omega$-marginals.

\medskip
In the particular case where $X$ has an isolated point $x_0$, then the following trick allows us to weaken a bit this assumption.
For any subset of $M_\Omega$ of $\mathcal{M}_\Omega(X)$, define
$$\widetilde{M}_\Omega=\left \{\nu\in \mathcal{M}_\Omega(X), \nu=\mathbf D(x_0)+\mu,\mu\in M_\Omega\right \},$$
where $\mathbf D$ is the Dirac measure on $x_0$.
Assume that for every $\mu\in M_\Omega$, we have $\mu_\omega(X)\leq C$ for $\mathbf P$-almost every $\omega$, for some constant $C$.
Then for every $\nu\in \widetilde{M}_\Omega$,
\begin{equation}\label{coboundedMtilde}
1\leq \nu_\omega(X)\leq 1+C
\end{equation}
for $\mathbf P$-almost every $\omega$.

\begin{cor}\label{coroProkhorovisolated}
Assume that $X$ has an isolated point $x_0$.
Consider a subset $M_\Omega$ of $\mathcal{M}_\Omega(X)$.
Assume that $M_\Omega$ is tight and that there exists $C\geq 0$ such that for every $\mu\in M_\Omega$,
$\pi_\Omega(\mu)\leq C \mathbf P$.
Then, $M_\Omega$ is relatively compact.
\end{cor}

\begin{proof}
The set $M_\Omega$ is tight, so there is a compact $K_0$ such that for every $\mu\in M_\Omega$,
$\pi_X(\mu)(K_0^c)\leq \epsilon$.
Then, the set $K=K_0\cup \{x_0\}$ is also compact and for every $\nu\in \widetilde{M}_\Omega$,
$\pi_X(\nu)(K^c)\leq \epsilon$.
Thus, $\widetilde{M}_\Omega$ is also tight.
By~(\ref{coboundedMtilde}), $\widetilde{M}_\Omega$ satisfies the assumptions of Corollary~\ref{coroProkhorov},
so it is relatively compact.

Let $\nu\in \mathbf{cl}(\widetilde{M}_\Omega)$.
By \cite[Theorem~2.14]{AliprantisBorder}, there exists a net $(\nu_\alpha)_{\alpha\in A}$ converging to $\nu$.
Assume by contradiction that $\mathbf P(\nu_\omega(\{x_0\})<1)>0$.
Then, there exists $c<1$ such that
the event $A=\{\nu_\omega(\{x_0\})\leq c\}$ satisfies $\mathbf P(A)>0$.
Since $x_0$ is isolated and $X$ is Hausdorff, $\{x_0\}$ is both closed and open, hence the function $\mathbf 1_{x_0}$ is bounded continuous
and $(x,\omega)\mapsto \mathbf{1}_A(\omega)\mathbf{1}_{x_0}(x)$ is a random bounded continuous function.
Since for every $\alpha$, $\nu_\alpha(\{x_0\})\geq 1$, applying convergence to this function, we get that
$$\mathbf P(A)\leq \mathbf E[\mathbf 1_A\nu_\alpha(\{x_0\})]\underset{\alpha\to \infty}{\longrightarrow}\mathbf E[\mathbf 1_A\nu(\{x_0\})]\leq c\mathbf P(A),$$
which is a contradiction.
Thus,
$$\mu_\omega=\nu_\omega-\mathbf D(x_0)$$
is a well-defined random finite measure on $X$.
In other words, the map
$$F:\nu\in \mathbf{cl}(\widetilde{M}_\Omega)\mapsto \nu-\mathbf{D}(x_0)\in \mathcal{M}_\Omega(X)$$
is well defined.
Moreover, it is continuous, therefore $F(\mathbf{cl}(\widetilde{M}_\Omega))$ is compact and by Lemma~\ref{measuredeterminedbylipschitz}, the weak topology is Hausdorff, so $F(\mathbf{cl}(\widetilde{M}_\Omega))$ is closed.
Now, $M_\Omega\subset F(\mathbf{cl}(\widetilde{M}_\Omega))$ and so  $\mathbf{cl}(M_\Omega)\subset F(\mathbf{cl}(\widetilde{M}_\Omega))$.
Thus, $\mathbf{cl}(M_\Omega)$ is compact, i.e.\ $M_\Omega$ is relatively compact.
\end{proof}

\bibliographystyle{alpha}
\bibliography{main}
\end{document}